\documentclass{article}

\RequirePackage{amsthm,amsmath,amsfonts,amssymb}
\RequirePackage[numbers]{natbib}
\RequirePackage[colorlinks,citecolor=blue,urlcolor=blue]{hyperref}
\RequirePackage{graphicx}

\usepackage[ruled, vlined]{algorithm2e}
\makeatletter
\renewcommand*{\@algocf@post@ruled}{}
\makeatother

\newcommand\blfootnote[1]{%
  \begingroup
  \renewcommand\thefootnote{}\footnote{#1}%
  \addtocounter{footnote}{-1}%
  \endgroup
}

\usepackage{multirow}
\usepackage{arydshln}
\usepackage[pagewise]{lineno}
\usepackage{stmaryrd}
\usepackage[dvipsnames]{xcolor} 
\usepackage{enumitem}
\usepackage{float}
\usepackage{subcaption}
\usepackage[utf8]{inputenc}
\usepackage{authblk}
\usepackage{orcidlink}
\usepackage{indentfirst}
\usepackage[includeheadfoot,margin=2.54cm]{geometry}

\numberwithin{equation}{section}

\theoremstyle{plain}

\newtheorem{corollary}{Corollary}
\newtheorem{lemma}{Lemma}
\newtheorem{proposition}{Proposition}

\theoremstyle{remark}
\newtheorem{remark}{Remark}
\newtheorem{assumption}{Assumption}
\newtheorem{definition}{Definition}

\DeclareMathOperator*{\argmin}{arg\,min}
\DeclareMathOperator*{\argmax}{arg\,max}

\DeclareMathOperator*{\logdet}{log\,det}
\DeclareMathOperator*{\domain}{dom}
\DeclareMathOperator*{\support}{supp}
\DeclareMathOperator*{\tr}{tr}

\DeclareMathOperator*{\minimize}{minimize}
\DeclareMathOperator*{\maximize}{maximize}
\newcommand{\prox}{\text{prox}}

\title{\bf On variational inference and maximum likelihood estimation with the $\lambda$-exponential family}
\author[1,a,$\ast$]{Thomas Guilmeau}
\author[1,b]{Emilie Chouzenoux}
\author[2]{Víctor Elvira}

\affil[1]{Université Paris-Saclay, CentraleSupélec, INRIA, CVN, France}
\affil[ ]{$^{\textrm{a}}$ \texttt{thomas.guilmeau@inria.fr} \orcidlink{0000-0002-8484-6550}}
\affil[ ]{$^{\textrm{b}}$ \texttt{emilie.chouzenoux@centralesupelec.fr} \orcidlink{0000-0003-3631-6093}}
\affil[2]{School of Mathematics, University of Edinburgh, United Kingdom}
\affil[ ]{\texttt{victor.elvira@ed.ac.uk} \orcidlink{0000-0002-8967-4866}}
\date{}

\begin{document}

\maketitle

\begin{abstract}

The $\lambda$-exponential family has recently been proposed to generalize the exponential family. While the exponential family is well-understood and widely used, this is not the case yet for the $\lambda$-exponential family. However, many applications require models that are more general than the exponential family, and the $\lambda$-exponential family is often a good alternative. In this work, we propose a theoretical and algorithmic framework to solve variational inference and maximum likelihood estimation problems over the $\lambda$-exponential family. We give new sufficient optimality conditions for variational inference problems. Our conditions take the form of generalized moment-matching conditions and generalize existing similar results for the exponential family. We exhibit novel characterizations of the solutions of maximum likelihood estimation problems, that recover optimality conditions in the case of the exponential family. For the resolution of both problems, we propose novel proximal-like algorithms that exploit the geometry underlying the $\lambda$-exponential family. These new theoretical and methodological insights are tested on numerical examples, showcasing their usefulness and interest, especially on heavy-tailed target distributions.
\blfootnote{T.G. and E.C. acknowledge support from the ERC Starting Grant MAJORIS ERC-2019-STG-850925. The work of V. E. is supported by the \emph{Agence Nationale de la Recherche} of France under PISCES (ANR-17-CE40-0031-01), the Leverhulme Research Fellowship (RF-2021-593), and by ARL/ARO under grant W911NF-22-1-0235.}\blfootnote{$\ast$ Corresponding author: Thomas Guilmeau.}
\end{abstract}

\section{Introduction}

Variational inference and maximum likelihood estimation are two classes of statistical problems arising in many applications. In variational inference, one aims at approaching an intractable target distribution of interest by a distribution from a family of (usually parametric) approximating densities. This is done by minimizing a discrepancy measure, such as the Kullback-Leibler \cite{kullback1951} or the Rényi \cite{renyi1961} divergence, between the target distribution and its approximation over the approximating family. In maximum likelihood estimation, one gets data samples, selects a parametric model to account for the unknown data-generating distribution, and then searches for the parameter maximizing the model likelihood over the data samples. These two optimization tasks are deeply related as maximum likelihood estimation is equivalent to minimizing a Kullback-Leibler divergence in the large number of samples limit \cite{white1982}.

In variational inference and maximum likelihood estimation, a popular choice for the approximating family is the exponential family \cite{barndorff-nielsen2014}. The exponential family is a family of probability distributions indexed by a finite-dimensional parameter, with the parameter appearing in the definition of the density through a scalar product with a sufficient statistics. Many well-known families of distributions can be written as instances of the exponential family, such as the Gaussian distributions. The exponential family benefits from numerous theoretical properties, many of them coming from convex analysis \cite{barndorff-nielsen2014}. For instance, the exponential family contains the distributions with a sufficient statistics, a fact known as the Pitman-Koopman-Darmois theorem \cite{tichochinsky1984}. This implies that the maximum likelihood estimator over the exponential family is reached when a moment-matching condition on sufficient statistics is satisfied \cite{campbell1970}. In variational inference, minimizing the Kullback-Leibler divergence over the exponential family leads to optimality conditions which can also be written as a moment-matching condition on sufficient statistics (see \cite{bishop2006, cappe2008, wainwright2008}). Thus, variational inference and maximum likelihood problems over the exponential family are both solved when moment-matching conditions are satisfied.

The exponential family also benefits from many geometric properties \cite{amari1985, nielsen2010}. Indeed, the Kullback-Leibler divergence between two distributions from the exponential family can be seen as the Bregman divergence induced by the log-partition function of the family. Bregman divergences generalize the Euclidean distance, and can be plugged in optimization algorithms, leading for instance to the so-called Bregman proximal gradient algorithms \cite{teboulle2018}. These properties can be leveraged to design more efficient algorithms over the exponential family in many settings \cite{banerjee2005, hoffman2013, khan2017, guilmeau2022}.

Despite the advantages of using the exponential family, there exists some contexts where it is better to use other types of distributions. For instance, the exponential family cannot represent physical systems governed by large fluctuations, such as cold atoms in optical lattices \cite{douglas2006}. In ecology, using Gaussian kernels to account for the diffusion of a population does not allow to represent species invading a territory with increasing speed, while heavier-tailed kernels can \cite{kot1996}. In signal processing and statistics, Student priors have been used to enforce signal sparsity \cite{chantas2008} or for logistic regression \cite{gelman2008}, and Cauchy distributions to model noise \cite{laus2018}. Using Student distributions rather than Gaussian ones have also been proven beneficial to cluster heavy-tailed data in \cite{peel2000}, while Student distributions have been used successfully in importance sampling \cite{cappe2008, elvira2019, wang2022}. 

Motivated by these situations, several works generalize the exponential family and extend its properties. These generalizations are often indexed by a scalar parameter, with the value zero corresponding to the exponential family. One can mention the $q$-exponential family studied in \cite{amari2011}, the $\mathcal{F}^{(\alpha)}$-family and $\mathcal{F}^{(-\alpha)}$-family of \cite{wong2018}, and the unifying $\lambda$-exponential family studied in \cite{wong2022}. We focus on the latter in this paper as it recovers the two former ones. The densities of distributions from the $\lambda$-exponential family are similar to those from the standard exponential family, but the scalar product between the parameter and what plays the role of sufficient statistics is replaced by a non-linear coupling. Instances of the $\lambda$-exponential family are the Student distributions (including Cauchy distributions), the Student Wishart distributions \cite{ayadi2023}, the $\beta$-Gaussian distributions \cite{martins2022}, or the Dirichlet perturbation model \cite{wong2022}. The geometric properties of these families have also been studied in the above papers. More precisely, and similarly to the situation for the standard exponential family, the authors of \cite{wong2022} established strong links between the $\lambda$-exponential family, the Rényi divergence, and a quantity that generalizes the Bregman divergence. Note that while the exponential family is studied using convex duality, the authors of \cite{wong2022} proposed the theory of $\lambda$-duality to study the $\lambda$-exponential family.

Generalizations of the exponential family have already been used in several tasks in statistics. Let us mention the creation of paths between distributions \cite{masrani2021}, neural attention mechanisms and regression problems with bounded noise \cite{martins2022}, or the understanding of generative adversarial networks based on $f$-divergences \cite{nowozin2016, nock2017}. Let us also mention the work of \cite{kainth2022} in which an optimization algorithm using a generalization of the Bregman divergence is studied and applied for maximum likelihood estimation over the $\lambda$-exponential family.

However, the $\lambda$-exponential family has been less studied than the standard exponential family. Indeed, to our knowledge, $(i)$ variational inference problems over generalizations of the exponential family have not been studied, $(ii)$ maximum likelihood estimation problems are usually solved within a particular $\lambda$-exponential family (see the works of \cite{hasanasab2021, ayadi2023} for instance), and $(iii)$ no algorithm exploits explicitly the geometry of these models (see \cite{kainth2022} for an exception).

As a summary, we propose a theoretical analysis and a novel methodological framework that allows to tackle variational inference and maximum likelihood estimation problems on the $\lambda$-exponential family. Our contributions are as follows:
\begin{itemize}
    \item[$(i)$] We give new optimality conditions for variational inference problems on the $\lambda$-exponential family that generalize the existing moment-matching conditions for the exponential family.
    \item[$(ii)$] We propose novel characterizations for the solutions of maximum likelihood estimation problems. We show that these are optimal conditions in the case of the exponential family, and related (in a sense we explicit) to optimal ones in the case of the $\lambda$-exponential family.
    \item[$(iii)$] We introduce new algorithms generalizing moment-matching to solve the considered variational inference and maximum likelihood problems, including an expectation-maximization algorithm. Our algorithms are shown to be related to proximal algorithms in the geometry induced by the Rényi divergence.
    \item[$(iv)$] All the aforementioned results are obtained using a novel theoretical framework to study the exponential family and the $\lambda$-exponential family based on non-convex duality. This new framework allows us to recover known results for the exponential family and to generalize them in a simple and unified way.
    \item[$(v)$] We illustrate numerically the behavior of our algorithms on variational inference and maximum likelihood estimation problems involving heavy-tailed distributions, showing the benefits of our novel theoretical results.
\end{itemize}

The paper is organized as follows. We present some background in Section \ref{section:preliminaries}. In Section \ref{section:firstResults}, we state our main assumptions, an important example, and our main technical results. In Section \ref{section:applicationsStatisticalProblems}, we apply these novel results to analyze, in a systematic way, variational inference and maximum likelihood estimation problems. We also propose proximal-like algorithms to solve these problems and compare the situation between the $\lambda$-exponential family and the standard exponential family. We illustrate our findings in Section \ref{section:numerics} through numerical experiments. Finally, we present future research developments and conclude in Section \ref{section:conclusion}.

\section{Preliminaries}
\label{section:preliminaries}

We introduce some preliminary background on divergences \cite{vanErven2014}, the $\lambda$-exponential family \cite{wong2022}, and convex analysis \cite{bauschke2011} that we will use throughout the rest of the paper.

\subsubsection*{Notation} 

We introduce some notation that will hold throughout the paper. $\mathcal{H}$ is a real Hilbert space in finite dimension with scalar product $\langle \cdot, \cdot \rangle$. Given a natural number $d$, $\mathcal{S}_{+}^d$ denotes the set of positive semi-definite matrices in dimension $d$, $\mathcal{S}_{++}^d$ denotes the set of positive definite matrices in dimension $d$, and $\mathcal{S}^d_{--}$ denotes the set of matrices obtained as the opposite of matrices in $\mathcal{S}_{++}^d$. Finally, $\mathbb{R}_{++}$ is the set of positive real numbers and $\bar{ \mathbb{R}}$ is the extended real line.

The set $\mathcal{X}$ with its Borel algebra is a measurable space, $m$ is a measure on this space, and $\mathcal{P}(\mathcal{X},m)$ is the set of probability measures on this space which admit a density with respect to $m$. We will often use the same notation for a distribution of $\mathcal{P}(\mathcal{X},m)$ and its density. The letter $S$ will be used to denote the support of a distribution. The restriction of a probability density $q \in \mathcal{P}(\mathcal{X},m)$ to a set $\mathcal{Y}$ is denoted by $q_{| \mathcal{Y}}$. We denote the Lebesgue measure by $dx$. The family of Gaussian distributions in dimension $d$ will be denoted by $\mathcal{G}^d$, and the family of Student distributions in dimension $d$ with degree of freedom parameter $\nu > 0$ by $\mathcal{T}_{\nu}^d$ (the formal definition is recalled in the remaining).

Generally, we used sub-scripts to describe the dependence over a scalar parameter, an index, or an iteration count, while we used super-scripts to denote escort distributions or conjugate functions, two notions that will be defined later on in the paper.

\subsection{Entropies and statistical divergences}
Let us introduce statistical notions that we will leverage through the rest of the paper. The first one is the entropy of a probability distribution, which is related to the information the distribution encodes.

\begin{definition}
    \label{def:entropy}
    Consider $\alpha > 0$, $\alpha \neq 1$, and a probability distribution $p \in \mathcal{P}(\mathcal{X}, m)$. Then the \emph{Rényi entropy} is defined by
    \begin{equation}
        H_{\alpha}(p) = \frac{1}{1-\alpha} \log \left( \int p(x)^{\alpha} m(dx) \right).
    \end{equation}
    When $\alpha = 1$, we define $H_1$ as the standard \emph{Shannon entropy}, that is
    \begin{equation}
        H_1(p) = - \int \log \left(p(x) \right) p(x) m(dx).
    \end{equation}
    If the integrals above do not converge, then the corresponding entropies are equal to $+\infty$.
\end{definition}

We now introduce the Rényi and Kullback-Leibler divergences. These divergences measure the discrepancy between two probability densities. Although they are not distances, they are non-negative, and they are null if and only if the two considered densities are equal almost everywhere.

\begin{definition}
    \label{def:divergences}
    Consider $\alpha > 0$, $\alpha \neq 1$, and probability distributions $p_1, p_2 \in \mathcal{P}(\mathcal{X}, m)$. Then the \emph{Rényi divergence} between $p_1$ and $p_2$ is defined by
    \begin{equation}
        RD_{\alpha}(p_1,p_2) = \frac{1}{\alpha-1} \log \left( \int p_1(x)^{\alpha} p_2(x)^{1-\alpha} m(dx) \right).
    \end{equation}
    When $\alpha = 1$, we define $RD_{1}$ as the \emph{Kullback-Leibler divergence} through
    \begin{equation}
        RD_1(p_1,p_2) = KL(p_1, p_2) = \int \log \left( \frac{p_1(x)}{p_2(x)} \right) p_1(x) m(dx).
    \end{equation}
    If these quantities are not defined, then the divergence takes the value $+\infty$.
\end{definition}

The notations $H_1$ and $RD_1$ in Definitions \ref{def:entropy} and \ref{def:divergences}, respectively, are motivated by the property that when $\alpha \rightarrow 1$, the Rényi entropy identifies with the Shannon entropy and the Rényi divergence with the Kullback-Leibler divergence \cite{vanErven2014}.

We conclude this section by defining a transformation that, for a given probability density, leads to its so-called escort distribution, parametrized by a scalar parameter $\alpha > 0$. When $\alpha = 1$, this transformation is simply the identity (i.e., the distribution identifies with its escort).

\begin{definition}
    \label{def:escortProb}
    Consider $\alpha > 0$ and $p \in \mathcal{P}(\mathcal{X}, m)$. The \emph{escort probability distribution} with exponent $\alpha$ of $p$ is the probability $p^{(\alpha)} \in \mathcal{P}(\mathcal{X}, m)$ defined by
    \begin{equation}
        p^{(\alpha)}(x) = \frac{1}{\int p(x)^{\alpha}m(dx)} p(x)^{\alpha},
    \end{equation}
    assuming the normalization constant $\int p(x)^{\alpha}m(dx)$ is finite.
\end{definition}

\subsection{The exponential family and the $\lambda$-exponential family}

We introduce the $\lambda$-exponential family, which is a generalization of the standard exponential family. Such family is obtained by replacing the scalar product $\langle\cdot,\cdot\rangle$, in the definition of the standard exponential family, by a non-linear coupling $c_{\lambda}$ defined as
\begin{equation}
    \label{eq:coupling}
    c_{\lambda}(u,v) = \frac{1}{\lambda} \log( 1 + \lambda \langle u,v \rangle),\,\forall u,v \in \mathcal{H}.
\end{equation}
Since $c_{\lambda}(u,v) \xrightarrow[]{\lambda \rightarrow 0} \langle u,v \rangle$, we denote $c_{0}(u,v) = \langle u,v \rangle$, $\forall u,v \in \mathcal{H}$.

We now turn to the definition of the $\lambda$-exponential family, following the formalism of \cite{wong2022}. This definition encompasses the definition of the standard exponential family. We set the conventions that $\log(s) = -\infty$ when $s \leq 0$ and $\exp(-\infty) = 0$. We also give examples in Figure \ref{fig:densitiesLambda} of densities from the $\lambda$-exponential family for different values of $\lambda$.

\begin{definition}
    \label{def:lambda-expFamily}
    Consider $\lambda \in \mathbb{R}$. The \emph{$\lambda$-exponential family $\mathcal{Q}_{\lambda}$ with sufficient statistics $T$ and base measure $m$} is the family $\mathcal{Q}_{\lambda} = \{q_{\vartheta} \in \mathcal{P}(\mathcal{X},m),\, \vartheta \in \domain \varphi_{\lambda} \}$, with
    \begin{equation}
        \label{eq:lambdaExpFamilyDensity}
        q_{\vartheta}(x) = \exp \left( c_{\lambda}(\vartheta,T(x)) - \varphi_{\lambda}(\vartheta) \right),
    \end{equation}
    where $c_{\lambda}$ is the non-linear coupling defined in Equation \eqref{eq:coupling}. Function $\varphi_{\lambda}$ in \eqref{eq:lambdaExpFamilyDensity} is the $\lambda$-log-partition function, defined for any $\vartheta \in \domain \varphi_{\lambda}$ by
    \begin{equation}
        \label{eq:lambdaLogPartition}
        \varphi_{\lambda}(\vartheta) = \log \left( \int \exp( c_{\lambda}(\vartheta, T(x)) m(dx) \right).
    \end{equation}
    The support of $q_{\vartheta}$ is the set $S_{\vartheta} = \{x \in \mathcal{X},\, 1 + \lambda \langle \vartheta, T(x) \rangle > 0\}$.
    When $\alpha = 1-\lambda$ is positive, we introduce, for any $\vartheta \in \domain \varphi_{\lambda}$, the entropy function 
    \begin{equation}
        \label{eq:entropyFunctionDef}
        \psi_{\lambda}(\vartheta) = -H_{\alpha}(q_{\vartheta}),\,\forall \vartheta \in \domain \varphi_{\lambda}.
    \end{equation}
\end{definition}

\begin{remark}
When $\lambda = 0$, we have $c_0(\cdot,\cdot) = \langle \cdot, \cdot \rangle$, and we recover in \eqref{eq:lambdaExpFamilyDensity} the standard notion of exponential family, that is $q_{\vartheta}(x) = \exp(\langle \vartheta, T(x)\rangle - \varphi_0( \vartheta ))$. In this case, the family is denoted by $\mathcal{Q}$ and we drop the subscript $\lambda$.
\end{remark}

\begin{figure}[H]
    \centering
    \begin{subfigure}[b]{0.32\textwidth}
        \includegraphics[width = \textwidth]{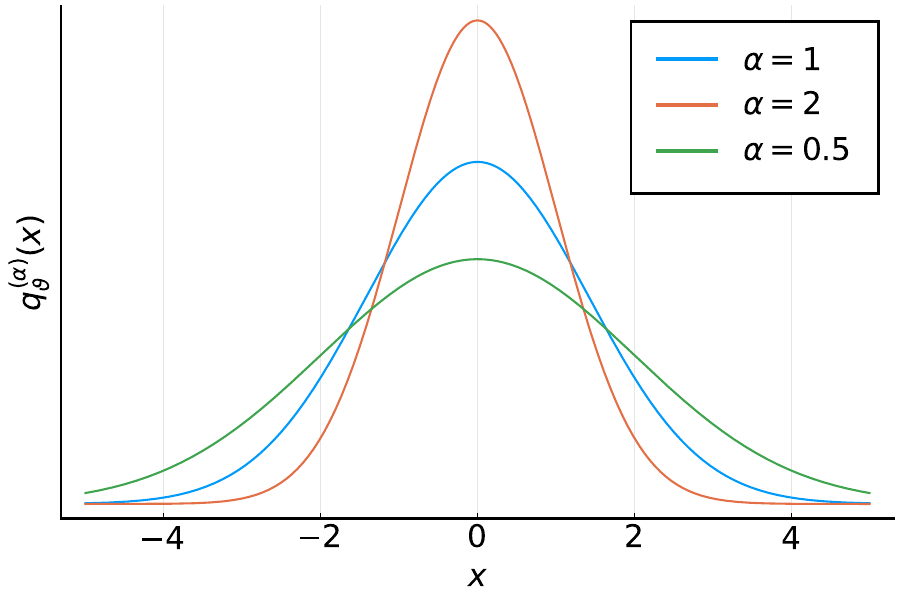}
        \caption{$\lambda=0$}
    \end{subfigure}  
    \hfill
    \begin{subfigure}[b]{0.32\textwidth}
        \includegraphics[width = \textwidth]{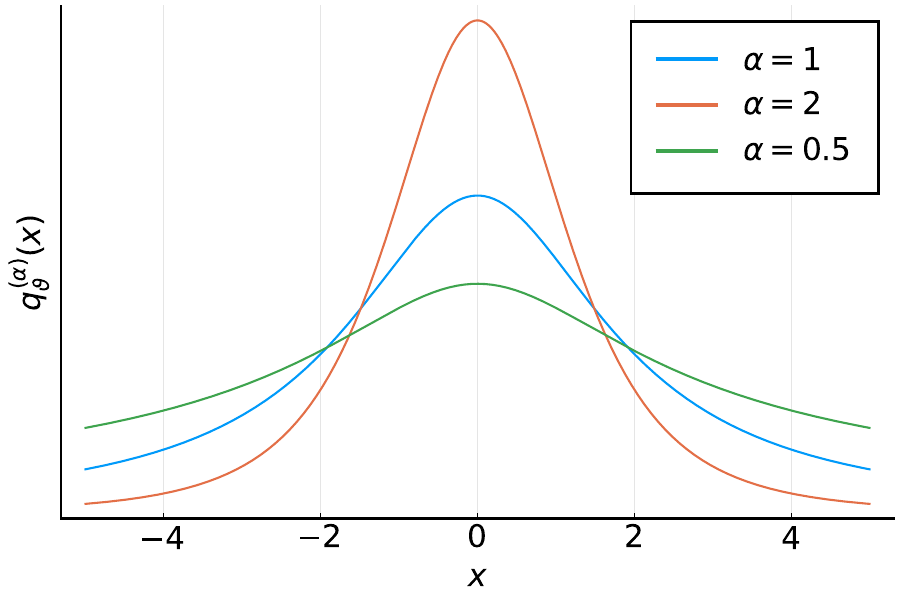}
        \caption{$\lambda = -1$}
    \end{subfigure}  
    \hfill
    \begin{subfigure}[b]{0.32\textwidth}
        \includegraphics[width = \textwidth]{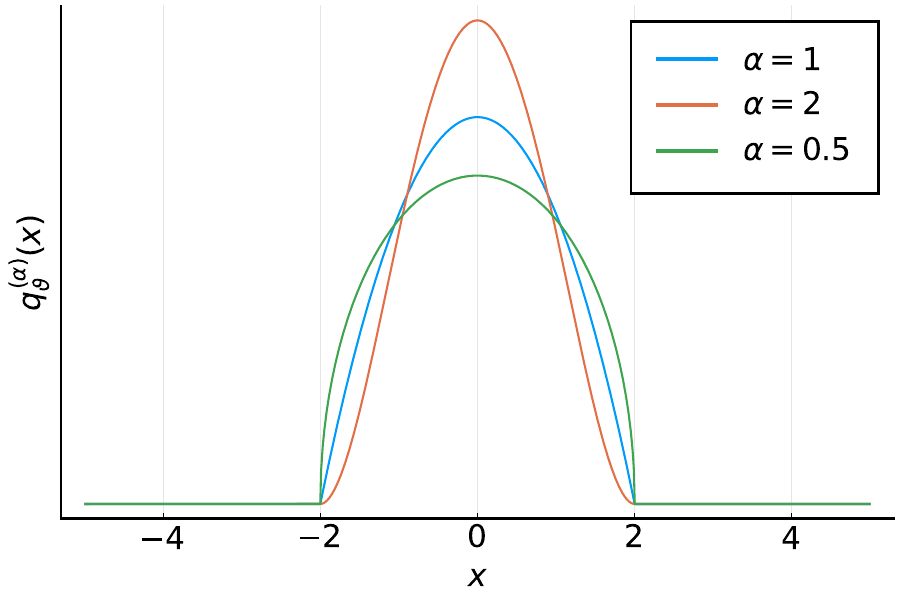}
        \caption{$\lambda=1$}
    \end{subfigure} 
    \caption{Plots of the densities $q_{\vartheta}^{(\alpha)}$, for the $\lambda$-exponential family obtained with sufficient statistics $T(x) = x^2$ and $\vartheta = 2$, for different values of $\lambda \in \{-1,0,1\}$ and $\alpha \in \{0.5,1,2\}$. When $\lambda = 0$, we recover a Gaussian distribution, while we obtain distributions with respectively heavier tails for $\lambda=-1$ and lighter tails for $\lambda > 0$. Also, values of $\alpha > 1$ lighten the tails while values $\alpha < 1$ make them heavier.}
    \label{fig:densitiesLambda}
\end{figure}

\subsection{$\lambda$-duality and proximal operators}
\label{subsection:lambdaDuality}

We now introduce elements of the concept of $\lambda$-duality, that will play an important role in our analysis of the considered optimization problems and the derivation of their optimality conditions. 

The $\lambda$-duality, initially introduced in \cite{wong2018, wong2022}, can be viewed as an extension of the usual convex duality~\cite{bauschke2011} (sometimes called Fenchel-Rockafellar duality). Let us remind that the convex duality relies on a coupling between primal and dual variables through the scalar product $\langle \cdot, \cdot \rangle$. This leads in particular to the notion of convex (or Fenchel) conjugate of a function $f : \mathcal{H} \rightarrow \bar{\mathbb{R}}$, defined at $v \in \mathcal{H}$ by
\begin{equation}
    \label{eq:fenchelConjugate}
    f^*(v) = \sup_{u \in \mathcal{H}} \langle u,v\rangle - f(u).
\end{equation}
Such conjugate can then used to define the subgradient of function $f$, by saying that $v$ is a subgradient of $f$ at $u$, denoted by $v \in \partial f(u)$, if and only if
\begin{equation}
    f^*(v) + f(u) = \langle u,v \rangle.
\end{equation}
One can then verify that $v \in \partial f(u)$ is equivalent to having that
\begin{equation}
    f(u') \geq f(u) + \langle v, u' - u \rangle,\,\forall u' \in \domain f,
\end{equation}
meaning that the right-hand side is a linear tangent minorant of $f$. The subdifferential can also be used to state optimality conditions through the Fermat rule \cite{bauschke2011}.

The $\lambda$-duality is constructed by replacing the scalar product of $\mathcal{H}$, appearing for instance in~\eqref{eq:fenchelConjugate}, by the non-linear coupling $c_{\lambda}(\cdot, \cdot)$ introduced in Equation \eqref{eq:coupling}. This leads to the definition of several mathematical notions, given hereafter.

\begin{definition}
    \label{def:conjugacy}
    Consider a proper function $f : \mathcal{H} \rightarrow \bar{\mathbb{R}}$ and $\lambda \in \mathbb{R}$.
    \begin{itemize}
        \item[$(i)$] We define its \emph{$c_{\lambda}$-conjugate} $f^{c_{\lambda}} : \mathcal{H} \rightarrow \bar{\mathbb{R}}$ by
        \begin{equation}
            \label{eq:lambda-conjugate-def}
            f^{c_{\lambda}}(v) = \sup_{u \in \mathcal{H}}  c_{\lambda}(u,v) - f(u).
        \end{equation}

        \item[$(ii)$] We say that $v \in \mathcal{H}$ is a \emph{$c_{\lambda}$-subgradient} of $f$ at $u$ and belongs to the \emph{$c_{\lambda}$-subdifferential} of $f$ at $u$, denoted by $\partial^{c_{\lambda}} f(u)$ if and only if
        \begin{equation}
            \label{eq:fenchelYoung}
            f^{c_{\lambda}}(v) + f(u) = c_{\lambda}(u,v).
        \end{equation}
    \end{itemize}

\end{definition}

As already mentioned, the above definitions correspond to generalizations of convex analysis theory. Similar constructions were achieved for instance in \cite{delara2020, lefranc2022} using the so-called CAPRA couplings, in \cite{fajardo2022} to study evenly convex functions, or in \cite{rachev1998} for general couplings in optimal transport. The standard notions of convexity have also been generalized by considering alternative notions of subgradients, such as in \cite{bednarczuk2022}. Let us relate this latter work to the notions introduced in Definition \ref{def:conjugacy}. Consider $f : \mathcal{H} \rightarrow \bar{\mathbb{R}}$, such that $v \in \partial^{c_{\lambda}} f(u)$. Equation \eqref{eq:fenchelYoung} can be rewritten in the following way:
\begin{align*}
    &f(u) + f^{c_{\lambda}}(v) = c_{\lambda}(u,v)\\
    \Leftrightarrow\, &c_{\lambda}(u, v) - f(u) \geq c_{\lambda}(u', v) - f(u'),\,\forall u' \in \domain f\\
    \Leftrightarrow\, &f(u') \geq f(u) + c_{\lambda}(u', v) - c_{\lambda}(u, v),\,\forall u' \in \domain f.
\end{align*}
This shows that $c_{\lambda}(\cdot, v)$ is a subgradient of $f$ at $u$ in the sense of the framework of abstract convexity, as outlined in \cite{bednarczuk2022} for instance.

Let us emphasize that Definition \ref{def:conjugacy} does not focus on the same objects than the ones in the study of \cite{wong2018, wong2022}. The latter also relies on $\lambda$-duality, but the so-called $\lambda$-gradient of $f$ is introduced before showing the fulfillment of Equation \eqref{eq:fenchelYoung}. This requires differentiability and regularity assumptions on $f$. We take the opposite direction in our Definition \ref{def:conjugacy}, as we define the $c_{\lambda}$-subdifferential assuming only the properness of $f$. As a consequence, we lose explicit expressions for $c_{\lambda}$-subgradients, while the $\lambda$-gradients in \cite{wong2018, wong2022} could be computed from the gradients of $f$. We will show in the following that Definition \ref{def:conjugacy} is sufficient to solve the considered optimization problems and that it is actually possible to exhibit $c_{\lambda}$-subgradients in our cases of interest, under mild hypotheses that are easy to check.

The above elements of $\lambda$-duality will be used subsequently to solve optimization problems of variational inference and maximum likelihood over a $\lambda$-exponential family providing explicit optimality conditions. We will also rely on proximal operators \cite{bauschke2011}, which are an essential tool for the algorithmic resolution of the considered problems. In order to fit the geometry induced by the $\lambda$-exponential family, we will rely on the Rényi proximal operator defined below.

\begin{definition}
    \label{def:proximalOperatorGen}
    Consider $\lambda \in \mathbb{R}$ such that $\alpha = 1 - \lambda$ is positive, the family $\mathcal{Q}_{\lambda}$ with $\lambda$-log-partition $\varphi_{\lambda}$, and an objective function $f : \mathcal{H} \rightarrow \bar{\mathbb{R}}$. Then the \emph{Rényi proximal operator} of $f$ with step-size $\tau > 0$ is defined by
    \begin{equation}
    \label{eq:proxRenyi}
        \prox_{\tau}^f(\vartheta') = \argmin_{\vartheta \in \domain \varphi_{\lambda}} f(\vartheta) + \frac{1}{\tau} RD_{\alpha}(q_{\vartheta'}, q_{\vartheta}).
    \end{equation} 
\end{definition}
When $\lambda = 0$, i.e.,~the $\lambda$-exponential family recovers the standard exponential family, the Rényi divergence appearing in the definition of $\prox_{\tau}^f$ reduces to the Kullback-Leibler divergence~\cite{nielsen2010}. In this case, $\prox_{\tau}^f$ can be seen as a Bregman proximal operator~\cite{bauschke2003} (see also \cite{guilmeau2022} for some examples of explicit Bregman proximal operators in the case of the exponential family). Note also that in \cite{kainth2022}, a proximal operator in the geometry defined by the Rényi divergence is mentioned but not studied.

In the following, we will refer to the operator \eqref{eq:proxRenyi} simply as proximal operator, except otherwise stated.

\section{Novel results on the $\lambda$-exponential family}
\label{section:firstResults}

In this section, we present a first set of novel results about the $\lambda$-exponential family, using the notion of $\lambda$-duality introduced in Definition \ref{def:conjugacy}. We first state our main assumptions and recover with our framework some known results including a key reformulation of the Rényi divergence in Section \ref{subsection:studyQ_lambda}. We then discuss the important example of Student distributions in Section \ref{subsection:studentDistributions}, before presenting in Section \ref{subsection:technicalOptResults} new technical optimality conditions that we will apply in subsequent sections to statistical problems.

\subsection{Assumptions and properties of the $\lambda$-exponential family}
\label{subsection:studyQ_lambda}

We now introduce our main assumptions and recover known results about the $\lambda$-exponential family under mild hypotheses, including a rewriting of the Rényi divergence in a way that will be crucial to solve statistical inference problems later on.

\begin{assumption}
    \label{assumption:expFamily}
    The $\lambda$-exponential family $\mathcal{Q}_{\lambda}$ is such that $\alpha = 1 - \lambda$ is positive and the function $\varphi_{\lambda}$ in \eqref{eq:lambdaLogPartition} is proper.
\end{assumption}

Assumption \ref{assumption:expFamily} implies in particular that $\domain \varphi_{\lambda} \neq \emptyset$ and that any $\vartheta \in \domain \varphi_{\lambda}$ is such that $q_{\vartheta}$ is well-defined and belongs to $\mathcal{P}(\mathcal{X},m)$. Note also that under Assumption \ref{assumption:expFamily}, $\varphi_{\lambda}$ cannot take the value $-\infty$, meaning in particular that, for any $\vartheta \in \domain \varphi_{\lambda}$, $S_{\vartheta} \neq \emptyset$. 

\begin{definition}
\label{def:compatible}
    Consider the $\lambda$-exponential family $\mathcal{Q}_{\lambda}$, the scalar $\alpha = 1 - \lambda$, and a probability density $p \in \mathcal{P}(\mathcal{X},m)$. We say that $p$ is \emph{$q_{\vartheta}$-compatible} for $q_{\vartheta} \in \mathcal{Q}_{\lambda}$ if the restriction of $p$ to the support of $q_{\vartheta}$, denoted by $S_{\vartheta}$, is such that $\int p_{|S_{\vartheta}}(x)^{\alpha}m(dx) \in (0, +\infty)$ and $\int T(x) p_{|S_{\vartheta}}(x)^{\alpha}m(dx)$ have finite components. If $p$ is $q_{\vartheta}$-compatible for any $q_{\vartheta} \in \mathcal{Q}_{\lambda}$, then we say that $p$ is \emph{$\mathcal{Q}_{\lambda}$-compatible}.
\end{definition}

The notion of compatibility in Definition \ref{def:compatible} is a technical condition that allows in particular to ensure the following well-posedness result.

\begin{lemma}
    \label{lemma:wellPosednessCouplingMoment}
    Consider the $\lambda$-exponential family $\mathcal{Q}_{\lambda}$, and $q_{\vartheta} \in \mathcal{Q}_{\lambda}$. Assume that Assumption \ref{assumption:expFamily} is satisfied and consider $\vartheta \in \domain \varphi_{\lambda}$ and $p \in \mathcal{P}(\mathcal{X},m)$. If $p$ is $q_{\vartheta}$-compatible, then $c_{\lambda}(\vartheta, p_{|S_{\vartheta}}^{(\alpha)}(T)) \in \mathbb{R}$.
\end{lemma}

\begin{proof}
    If $\lambda = 0$, $c_{\lambda}(\vartheta, p_{|S_{\vartheta}}^{(\alpha)}(T)) = \langle \vartheta, p_{|S_{\vartheta}}(T) \rangle$ and the result is straightforward. Now, consider $\lambda \neq 0$. The support of $q_{\vartheta}$ is the set $S_{\vartheta} = \{x \in \mathcal{X},\, 1 + \lambda \langle \vartheta, T(x) \rangle > 0\}$. Then we can compute
    \begin{equation}
        \label{eq:intermediateLemmaWellPosedness}
        1 + \lambda \langle \vartheta, p_{|S_{\vartheta}}^{(\alpha)}(T) \rangle = \int (1 + \lambda \langle \vartheta, T(x) \rangle ) p_{|S_{\vartheta}}^{(\alpha)}(x) m(dx).
    \end{equation}
    We get from the compatibility assumption that $p_{|S_{\vartheta}}^{(\alpha)}$ is well-defined and belongs to $\mathcal{P}(\mathcal{X},m)$. This ensures that the quantity in \eqref{eq:intermediateLemmaWellPosedness} is positive. Also by assumption, $p_{|S_{\vartheta}}^{(\alpha)}(T)$ is well-defined, ensuring that the quantity in \eqref{eq:intermediateLemmaWellPosedness} is also finite, hence the result.
\end{proof}

We now introduce an extra assumption stating that all the densities $q_{\vartheta} \in \mathcal{Q}_{\lambda}$ share the same support. In \cite{wong2022}, this property is also assumed and called the support condition. This assumption ensures that $q_{\vartheta}^{(\alpha)}(T)$ is well-defined for any $\vartheta \in \domain \varphi_{\lambda}$ and $\alpha = 1 - \lambda$ as we will show.

\begin{assumption}
    \label{assumption:supportCondition}
    There exists a non-empty set $S_{\lambda} \subset \mathcal{X}$ such that 
    \begin{equation}
        S_{\vartheta} = S_{\lambda},\,\forall \vartheta \in \domain \varphi_{\lambda},
    \end{equation}
    with $S_{\vartheta}$ being the support of $q_{\vartheta}$. Moreover, every $q_{\vartheta} \in \mathcal{Q}_{\lambda}$ is $\mathcal{Q}_{\lambda}$-compatible. 
\end{assumption}

We now state a property that links the coupling $c_{\lambda}$, the log-partition function $\varphi_{\lambda}$, and the Rényi divergence $RD_{\alpha}$. This technical property is used in the proof of a Rényi entropy maximization property in \cite{wong2022}, and we will exploit it further in our subsequent developments.

\begin{proposition}
\label{prop:rewritingRenyi}
    Consider the $\lambda$-exponential family $\mathcal{Q}_{\lambda}$ under Assumption \ref{assumption:expFamily} with $\alpha = 1-\lambda$. Consider a probability distribution $p \in \mathcal{P}(\mathcal{X}, m)$. For any $q_{\vartheta} \in \mathcal{Q}_{\lambda}$, $RD_{\alpha}(p, q_{\vartheta})$ satisfies
    \begin{equation}
    \label{eq:rewritingRenyi}
        RD_{\alpha}(p, q_{\vartheta}) = \varphi_{\lambda}(\vartheta) - c_{\lambda}(\vartheta, p^{(\alpha)}_{| S_{\vartheta}}(T))- H_{\alpha}(p_{| S_{\vartheta}}).
    \end{equation}
    Further, under Assumptions \ref{assumption:expFamily} and \ref{assumption:supportCondition}, for every $\vartheta' \in \domain \varphi_{\lambda}$,
    \begin{equation}
        \label{eq:rewritingRenyiLambdaExp}
        RD_{\alpha}(q_{\vartheta'}, q_{\vartheta}) = \varphi_{\lambda}(\vartheta) - c_{\lambda}(\vartheta, q_{\vartheta'}^{(\alpha)}(T))+ \psi_{\lambda}(\vartheta').
    \end{equation}
\end{proposition}

\begin{proof}
When $\lambda = 0$, recall that $q_{\vartheta}$ has full support. In this case, we have
\begin{align*}
    RD_{\alpha}(p, q_{\vartheta}) &= KL(p, q_{\vartheta})\\
    &= \int \log \left( \frac{p(x)}{q_{\vartheta}(x)} \right) p(x) m(dx),
\end{align*}
from which we can straightforwardly obtain the result using that $c_{\lambda}(\cdot, \cdot) = \langle \cdot, \cdot \rangle$ for $\lambda = 0$ and $\alpha = 1$.

For $\lambda \neq 0$, we compute the Rényi divergence $RD_{\alpha}$ (defined in Definition \ref{def:divergences}) between $p$ and $q_{\vartheta}$. Using the definitions of $q_{\vartheta}$ given in Definition \ref{def:lambda-expFamily}, of the Rényi entropy $H_{\alpha}$ given in Definition \ref{def:entropy}, and of the coupling $c_{\lambda}$ from Equation \eqref{eq:coupling}, we obtain the following result.
\begin{align*}
    RD_{\alpha}(p, q_{\vartheta}) 
    &= \frac{1}{\alpha - 1} \log \left( \int p_{|S_{\vartheta}}(x)^{\alpha} q_{\vartheta}(x)^{1-\alpha}m(dx)\right)\\
    &= \frac{1}{\alpha - 1} \log \left( \int p_{|S_{\vartheta}}(x)^{\alpha} \exp( (1-\alpha) c_{\lambda}(\vartheta, T(x)) - (1-\alpha)\varphi_{\lambda}(\vartheta) )  \nu(dx) \right)\\
    &= \frac{1}{\alpha - 1} \log \left( \int p_{|S_{\vartheta}}(x)^{\alpha} \exp( \lambda c_{\lambda}(\vartheta, T(x)))  \nu(dx) \right) + \varphi_{\lambda}(\vartheta)\\
    &= \frac{1}{\alpha - 1} \log \left( \int p_{|S_{\vartheta}}(x)^{\alpha} (1 + \lambda \langle \vartheta, T(x)\rangle) \nu(dx) \right) + \varphi_{\lambda}(\vartheta)\\
    &= \frac{1}{\alpha - 1} \log \left( \int p_{| S_{\vartheta}}(x)^{\alpha}\nu(dx) \left( 1 + \lambda \langle \vartheta, p^{(\alpha)}_{| S_{\vartheta}}(T)\rangle \right) \right) + \varphi_{\lambda}(\vartheta)\\
    &= \frac{1}{\alpha - 1} \log \left( \int p_{| S_{\vartheta}}(x)^{\alpha}\nu(dx) \right) + \frac{1}{\alpha - 1} \log \left( 1 + \lambda \langle \vartheta, p_{| S_{\vartheta}}^{(\alpha)}(T)\rangle \right) + \varphi_{\lambda}(\vartheta)\\
    &= - H_{\alpha}(p_{| S_{\vartheta}}) - c_{\lambda}(\vartheta, p^{(\alpha)}_{| S_{\vartheta}}(T)) + \varphi_{\lambda}(\vartheta),
\end{align*}
which proves the first part of the property in Equation \eqref{eq:rewritingRenyi}. The second part  in Equation \eqref{eq:rewritingRenyiLambdaExp} follows using the assumptions and Equation \eqref{eq:entropyFunctionDef}.
\end{proof}

We establish a second property, describing the $\lambda$-duality objects associated to $\mathcal{Q}_{\lambda}$ in terms of moments and entropy and recovering the results of \cite{wong2022} in our framework.

\begin{proposition}
    \label{prop:gradientMoment}
    Suppose that Assumptions \ref{assumption:expFamily} and \ref{assumption:supportCondition} are satisfied. Then, for every $\vartheta \in \domain \varphi_{\lambda}$, 
    \begin{align}
        \varphi_{\lambda}^{c_{\lambda}}(q_{\vartheta}^{(\alpha)}(T)) &= \psi_{\lambda}(\vartheta),\label{eq:entropyConjugate}\\
        q_{\vartheta}^{(\alpha)}(T) &\in \partial^{c_{\lambda}} \varphi_{\lambda}(\vartheta).\label{eq:escortMomentSubgradient}
    \end{align}
\end{proposition}

\begin{proof}
    We denote $\eta = q_{\vartheta}^{(\alpha)}(T)$ for sake of concision. We begin with the proof for \eqref{eq:entropyConjugate}. Using the result of Proposition \ref{prop:rewritingRenyi} and the non-negativity of the Rényi divergence,  
    \begin{equation}
        c_{\lambda}(\vartheta', \eta) - \varphi_{\lambda}(\vartheta') \leq \psi_{\lambda}(\vartheta),
    \end{equation}
    with equality if and only if $\vartheta' = \vartheta$. This shows that $\varphi_{\lambda}^{c_{\lambda}}(\eta) = \psi_{\lambda}(\vartheta)$ following Equation \eqref{eq:lambda-conjugate-def}, hence the result.

    We now turn to the proof for \eqref{eq:escortMomentSubgradient}. Consider the rewriting of the Rényi divergence from Proposition \ref{prop:rewritingRenyi}:
    \begin{equation}
        \label{eq:conjugateSubgradInter}
        0 = \varphi_{\lambda}(\vartheta) - c_{\lambda}(\vartheta, q_{\vartheta}^{(\alpha)}(T)) + \psi_{\lambda}(\vartheta).
    \end{equation}
    Then, by Equation \eqref{eq:entropyConjugate}, Equation \eqref{eq:conjugateSubgradInter} can be written as $\varphi_{\lambda}(\vartheta) + \varphi_{\lambda}^{c_{\lambda}}(\eta) = c_{\lambda}(\vartheta, \eta)$. This concludes the proof, using Equation \eqref{eq:fenchelYoung}.
\end{proof}

\begin{remark}
Assumption \ref{assumption:supportCondition} and Proposition \ref{prop:gradientMoment} ensure that, for every $\vartheta \in \domain \varphi_{\lambda}$, $q_{\vartheta}^{(\alpha)}(T)$ is well-defined and thus that $\partial^{c_{\lambda}} \varphi_{\lambda}(\vartheta)$ is non-empty. This can be viewed as a form of convexity result on $\varphi_{\lambda}$. Indeed, for $\lambda = 0$, which corresponds to the classical Fenchel duality theory, having a non-empty subdifferential at every point of $\domain \varphi$ implies that $\varphi(\vartheta) = \varphi^{**}(\vartheta)$ on $\domain \varphi$ \cite[Proposition 16.4]{bauschke2011}. This last equality shows that $\varphi$ is equal to its biconjugate and hence that it is convex.
\end{remark}

\subsection{The example of Student distributions}
\label{subsection:studentDistributions}

We now show that the Student distributions can be seen as a particular example of the $\lambda$-exponential family that satisfies the assumptions outlined in Section \ref{section:firstResults} and whose escort distributions have easily computable moments. This means that Student distributions will be an importance use-case of our coming theoretical results of Section \ref{section:applicationsStatisticalProblems}, as we will illustrate on numerical experiments in Section \ref{section:numerics}. Student distributions form an important class of distributions arising in several applications from statistics and signal processing \cite{peel2000, chantas2008, gelman2008, cappe2008, laus2018, wang2022}.

\begin{definition}
    \label{def:StudentDistributions}
    Consider the family of \emph{multivariate Student distributions} on $\mathbb{R}^d$ with fixed degree of freedom parameter $\nu > 0$. We denote this family by $\mathcal{T}^d_{\nu}$. Densities with respect to the Lebesgue measure are of the form
        \begin{equation}
            q_{\mu, \Sigma}(x) = \frac{1}{Z_{\nu}} \det(\Sigma)^{-\frac{1}{2}} \left( 1 + \frac{1}{\nu} (x-\mu)^{\top}\Sigma^{-1}(x-\mu) \right)^{- \frac{\nu + d}{2}},\, \forall x \in \mathbb{R}^d
        \end{equation}
    with location parameter $\mu \in \mathbb{R}^d$, scale matrix $\Sigma \in \mathcal{S}_{++}^d$, and normalization constant $Z_{\nu} = \frac{\Gamma(\nu/2) \nu^{d/2} \pi^{d/2}}{\Gamma((\nu+d)/2)}$ where $\Gamma$ denotes the Gamma function.
\end{definition}

The degree of freedom parameter $\nu > 0$ controls the tail behavior of the distributions. In particular, higher values of $\nu$ lead to distributions in $\mathcal{T}_{\nu}^d$ with lighter (but still heavy) tails, with the limit $\nu \rightarrow +\infty$ corresponding to the family of Gaussian distributions, which is an example of the exponential family. On the contrary, distributions in $\mathcal{T}_{\nu}^d$ for low $\nu$ have heavier tails, an example being that $\mathcal{T}_1^d$ is the family of multivariate Cauchy distributions. In particular, distributions in $\mathcal{T}_{\nu}^d$ have well-defined first order moments if $\nu > 1$ and well-defined second order moments if $\nu > 2$.

The next proposition shows that the $\lambda$-exponential family, with sufficient statistics being the first and second order moments, is the family of Student distribution when $\lambda < 0$ and that it satisfies Assumption \ref{assumption:expFamily}. We further compute the escort moments of Student distributions, which are $c_{\lambda}$-subgradients of $\varphi_{\lambda}$. We also compute the Rényi entropy of Student distributions, which is the $c_{\lambda}$-conjugate of $\varphi_{\lambda}$. Finally, we describe the distributions that are compatible with the Student distributions (following Definition \ref{def:compatible}) and show that Student distributions satisfy Assumption \ref{assumption:supportCondition}.

\begin{proposition}
    \label{prop:studentFamily}
    Consider the Student family $\mathcal{T}_{\nu}^d$.
    \begin{itemize}
        \item[$(i)$] The family $\mathcal{T}^d_{\nu}$ is an instance of the $\lambda$-exponential family (see Equation \eqref{eq:lambdaExpFamilyDensity}) for $\lambda = - \frac{2}{\nu + d}$ and with sufficient statistics $T(x) = (x, xx^{\top})$. Its natural parameters are $\vartheta = (\vartheta_1, \vartheta_2) \in \mathbb{R}^d \times \mathcal{S}_{--}^d$. It satisfies Assumption \ref{assumption:expFamily} and $\domain \varphi_{\lambda} = \{(\vartheta_1,\vartheta_2) \in \mathbb{R}^d \times \mathcal{S}_{--}^d,\,2(\nu+d) + \vartheta_1^{\top}\vartheta_2^{-1}\vartheta_1 > 0\}$.

        \item[$(ii)$] Consider $\alpha = 1 + \frac{\nu + d}{2}$. Then, for any $q_{\mu, \Sigma} \in \mathcal{T}_{\nu}^d$, $q_{\mu, \Sigma}^{(\alpha)}(T) = (q_{\mu, \Sigma}^{(\alpha)}(x), q_{\mu, \Sigma}^{(\alpha)}(xx^{\top}))^{\top}$ is such that
        \begin{equation}
            \begin{cases}
                q_{\mu, \Sigma}^{(\alpha)}(x) = \mu,\\
                q_{\mu, \Sigma}^{(\alpha)}(xx^{\top}) = \Sigma + \mu \mu^{\top}.
            \end{cases}
        \end{equation}
        The mapping $\vartheta \longmapsto q_{\vartheta}^{(\alpha)}(T)$ is a bijection from $\domain \varphi_{\lambda}$ to $\mathbb{R}^d \times \mathcal{S}_{++}^d$. Moreover, $\psi_{\lambda}(\vartheta) = \frac{1}{2} \logdet(\Sigma) + C$, where $C$ is a scalar depending only on $\nu$ and $d$.

        \item[$(iii)$] The $\mathcal{T}_{\nu}^d$-compatible distributions are the probability densities $p \in \mathcal{P}(\mathcal{X},dx)$ such that $p^{(\alpha)}$ has finite first and second order moments. The family $\mathcal{T}_{\nu}^d$, seen as a $\lambda$-exponential family, satisfies Assumption \ref{assumption:supportCondition}. 
    \end{itemize}
\end{proposition}

\begin{proof}
    The proof is deferred to Appendix \ref{appendix:proof}.
\end{proof}

\begin{remark}
    Proposition \ref{prop:studentFamily} generalizes analogous results for Gaussian distributions. Indeed, Gaussian distributions in dimension $d$, denoted by $\mathcal{G}^d$, form an example of the exponential family with sufficient statistics $T(x) = (x, xx^{\top})$, satisfying Assumptions \ref{assumption:expFamily} and \ref{assumption:supportCondition}. Furthermore, for any $q_{\mu, \Sigma} \in \mathcal{G}^d$ with $\mu \in \mathbb{R}^d, \Sigma \in \mathcal{S}^d_{++}$, we have 
    \begin{equation}
        \begin{cases}
            q_{\mu, \Sigma}(x) = \mu,\\
            q_{\mu, \Sigma}(xx^{\top}) = \Sigma + \mu \mu^{\top}.
        \end{cases}
    \end{equation}
    Finally, the $\mathcal{G}^d$-compatible distributions are the probability densities in $\mathcal{P}(\mathcal{X},dx)$ with finite first and second order moments. While the Gaussian case corresponds to $\lambda = 0$, we remark that the case $\lambda > 0$ corresponds to the $\beta$-Gaussians distributions discussed in \cite{martins2022}. However, these distributions do not satisfy the support condition of Assumption \ref{assumption:supportCondition}.
\end{remark}

We now establish some novel properties that state how two families of Student distributions with different degree of freedom parameters relate to each other, including the computation of some escort moments and compatibility conditions. This provides a mechanism to construct an escort distribution with lighter tails than the original distribution.

\begin{proposition}
    \label{prop:escortStudent}
    Let $p \in \mathcal{T}_{\nu_p}^d$ a Student distribution with dimension $d$, location $\mu_p$ and shape $\Sigma_p$. Set $\nu>0$ and consider the Student family $\mathcal{T}_{\nu}^d$ with associated $\alpha = 1 + \frac{2}{\nu+d}$. Then, the distribution $p$ is $\mathcal{T}_{\nu}^d$-compatible if and only if $\nu_p + 2 \frac{\nu_p + d}{\nu +d} > 2$, and the escort probability $p^{(\alpha)}$ is a Student distribution with $\nu^{(\alpha)}$ degrees of freedom, location $\mu^{(\alpha)}$, and shape $\Sigma^{(\alpha)}$ such that
    \begin{equation}
        \begin{cases}
            \nu^{(\alpha)} &= \nu_p + 2 \frac{\nu_p + d}{\nu + d},\\
            \mu^{(\alpha)} &= \mu_p,\\
            \Sigma^{(\alpha)} &= \frac{\nu_p}{\nu^{(\alpha)}} \Sigma_p.
        \end{cases}
    \end{equation} 
\end{proposition}

\begin{proof}
    By Proposition \ref{prop:studentFamily}, $\mathcal{T}_{\nu}^d$ is a $\lambda$-exponential family, with $\lambda = - \frac{2}{\nu+d}$. For such setting, $\alpha = 1 - \lambda$. The compatibility property requires $p^{(\alpha)}$ to have finite first and second order moments. Consider $x \in \mathbb{R}^d$, we compute
    \begin{align*}
        p(x)^{\alpha} &\propto \left(1 + \frac{1}{\nu_p} (x-\mu_p)^{\top}\Sigma_p^{-1}(x-\mu_p)\right)^{-\left(\frac{\nu_p + d}{2} \right)\left(1 + \frac{2}{\nu +d} \right)}\\
        &\propto \left(1 + \frac{1}{\nu_p} (x-\mu_p)^{\top}\Sigma_p^{-1}(x-\mu_p)\right)^{-\frac{1}{2} \left(\nu_p + 2\frac{\nu_p + d}{\nu +d} +d \right)}\\
        &\propto \left(1 + \frac{1}{\nu^{(\alpha)}} (x-\mu_p)^{\top}\left(\frac{\nu_p}{\nu^{(\alpha)}}\Sigma_p\right)^{-1}(x-\mu_p)\right)^{-\frac{\nu^{(\alpha)}+d}{2}}.
    \end{align*}
    We recognize that $p^{(\alpha)}$ is a Student distribution with $\nu^{(\alpha)}$ degrees of freedom, location $\mu^{(\alpha)}$, and shape $\Sigma^{(\alpha)}$, and that $p^{(\alpha)}$ has finite first and second order moments if and only if $\nu^{(\alpha)} > 2$, showing the result.
\end{proof}

Proposition \ref{prop:escortStudent} provides a systematic way to construct, from an initial distribution with possibly infinite moments, an escort distribution for which these moments are defined. Indeed, if $p \in \mathcal{T}_{\nu_p}^d$ for some $\nu_p > 0$, we can construct $p^{(\alpha)}$ where $\alpha = 1 + \frac{2}{\nu+d}$ and $\nu > 0$. The resulting escort distribution $p^{(\alpha)}$ has $\nu^{(\alpha)} > \nu_p$ degrees of freedom, i.e., a lighter tail than the one of $p$ itself. Indeed, we can have $\nu_p \leq 2$, meaning that $p$ has infinite variance and $\nu^{(\alpha)} > 2$, in which case $p^{(\alpha)}$ has finite variance.

\subsection{Novel technical optimality results}
\label{subsection:technicalOptResults}

We present in this section two new technical results, that will later be used to study the optimality conditions of the optimization problems arising in variational inference and maximum likelihood estimation.

\begin{proposition}
    \label{prop:optResult1}
    Consider the $\lambda$-exponential family $\mathcal{Q}_{\lambda}$ under Assumption \ref{assumption:expFamily}, and $\Bar{T} \in \mathcal{H}$ such that $c_{\lambda}(\vartheta, \bar{T}) \in \mathbb{R}$ for any $\vartheta \in \domain \varphi_{\lambda}$. Then $\vartheta_* \in \domain \varphi_{\lambda}$ minimizes $\vartheta \longmapsto \varphi_{\lambda}(\vartheta) - c_{\lambda}(\vartheta, \bar{T})$ if and only if $\bar{T} \in \partial^{c_{\lambda}} \varphi_{\lambda}(\vartheta_*)$.
\end{proposition}

\begin{proof}
    Suppose that $\vartheta_* \in \domain \varphi_{\lambda}$ minimizes $\vartheta \longmapsto \varphi_{\lambda}(\vartheta) - c_{\lambda}(\vartheta, \bar{T})$. This is equivalent to 
    \begin{equation}
        c_{\lambda}(\vartheta_*, \bar{T}) - \varphi_{\lambda}(\vartheta_*)  \geq c_{\lambda}(\vartheta, \bar{T}) - \varphi_{\lambda}(\vartheta), \,\forall \vartheta \in \domain \varphi_{\lambda}.
        \label{eq:prop5}
    \end{equation}
    By definition of the $c_{\lambda}$-conjugate, \eqref{eq:prop5} can be summarized as
    \begin{equation}
        c_{\lambda}(\vartheta_*, \bar{T}) - \varphi_{\lambda}(\vartheta_*) \geq \varphi_{\lambda}^{c_{\lambda}}(\bar{T}).
    \end{equation}
    Since the opposite inequality is true by definition, the above statement is equivalent to
    \begin{equation}
        c_{\lambda}(\vartheta_*, \bar{T}) - \varphi_{\lambda}(\vartheta_*) = \varphi_{\lambda}^{c_{\lambda}}(\bar{T}).
    \end{equation}
    That yields $\bar{T} \in \partial^{c_{\lambda}} \varphi_{\lambda}(\vartheta_*)$, which concludes the proof.
\end{proof}

\begin{lemma}
    \label{lemma:convexityCoupling}
    Consider $u \in \mathcal{H}$ and the function $c_{\lambda}(\cdot, u): v \longmapsto c_{\lambda}(v,u)$ for $\lambda \neq 0$. 
    \begin{itemize}
        \item[$(i)$] If $\lambda = 0$, the function $c_{\lambda}(\cdot, u)$ is linear.
        \item[$(ii)$] If $\lambda > 0$, the function $c_{\lambda}(\cdot, u)$ is concave.
        \item[$(iii)$] If $\lambda < 0$, the function $c_{\lambda}(\cdot, u)$ is convex.
    \end{itemize}
\end{lemma}

\begin{proof}
   Case $(i)$ follows from $c_{0}(\cdot,\cdot) = \langle \cdot, \cdot\rangle$. We now assume $\lambda \neq 0$. Consider $v_1, v_2 \in \mathcal{H}$ and $s \in [0,1]$. Then we can compute
    \begin{align*}
        c_{\lambda}(s v_1 + (1-s) v_2, u) &= \frac{1}{\lambda} \log( 1 + \lambda \langle s v_1 + (1-s)v_2, v \rangle  )\\
        &= \frac{1}{\lambda} \log( s (1 + \lambda \langle v_1, u \rangle) + (1-s)(1 + \lambda \langle v_2, u \rangle )).
    \end{align*}
    We then get the results of cases $(ii)$ and $(iii)$, using the convexity (resp. concavity) of $v \longmapsto \lambda^{-1} \log v$, resulting from the positive (resp. negative) sign of $\lambda$.
\end{proof}

\begin{proposition}
    \label{prop:optResult2}
    Consider the $\lambda$-exponential family $\mathcal{Q}_{\lambda}$ under Assumption \ref{assumption:expFamily}. Let a collection $\{ \bar{T}_i \}_{i=1}^N$ of  $N > 1$ elements of $\mathcal{H}$, such that for any $i \in \{1,\dots, N\}$, $c_{\lambda}(\vartheta, \bar{T}_i) \in \mathbb{R}$ for any $\vartheta \in \domain \varphi_{\lambda}$, and a collection of non-negative values $\{ \rho_i \}_{i=1}^N$ such that $\sum_{i=1}^N \rho_i = 1$. Suppose that there exists $\vartheta_* \in \domain \varphi_{\lambda}$ such that $\sum_{i=1}^N \rho_i \bar{T}_i \in \partial^{c_{\lambda}} \varphi_{\lambda}(\vartheta_*)$. 
\begin{itemize}
    \item[$(i)$] If $\lambda = 0$, $\vartheta_*$ minimizes $\vartheta \longmapsto \varphi_{\lambda}(\vartheta) - \sum_{i=1}^N \rho_i c_{\lambda}(\vartheta, \bar{T}_i)$.
    \item[$(ii)$] If $\lambda < 0$,  $\vartheta_*$
    minimizes the function $\vartheta \longmapsto \varphi_{\lambda}(\vartheta) - c_{\lambda}\left(\vartheta, \sum_{i=1}^N \rho_i \bar{T}_i\right)$ over $\domain \varphi_{\lambda}$, itself being an upper bound of the function $\vartheta \longmapsto \varphi_{\lambda}(\vartheta) - \sum_{i=1}^N \rho_i c_{\lambda}(\vartheta, \bar{T}_i)$ over $\domain \varphi_{\lambda}$. Moreover,
    \begin{equation}
        \varphi_{\lambda}(\vartheta_*) - \sum_{i=1}^N \rho_i c_{\lambda}\left(\vartheta_*, \bar{T}_i\right) \leq - \varphi_{\lambda}^{c_{\lambda}}\left(\sum_{i=1}^N \rho_i \bar{T}_i\right).
    \end{equation}   
    \item[$(iii)$] If $\lambda > 0$, $\vartheta_*$ 
    minimizes $\vartheta \longmapsto \varphi_{\lambda}(\vartheta) - c_{\lambda}\left(\vartheta, \sum_{i=1}^N \rho_i \bar{T}_i\right)$ over $\domain \varphi_{\lambda}$, itself being an lower bound of the function $\vartheta \longmapsto \varphi_{\lambda}(\vartheta) - \sum_{i=1}^N \rho_i c_{\lambda}(\vartheta, \bar{T}_i)$ over $\domain \varphi_{\lambda}$.
    Moreover,
    \begin{equation}
        - \varphi_{\lambda}^{c_{\lambda}} \left( \sum_{i=1}^N \rho_i \bar{T}_i \right) \leq \varphi_{\lambda}(\vartheta) - \sum_{i=1}^N \rho_i c_{\lambda}(\vartheta, \bar{T}_i),\,\forall \vartheta \in \domain \varphi_{\lambda},
    \end{equation}
\end{itemize}
\end{proposition}

\begin{proof}
    Let $\vartheta \in \domain \varphi_{\lambda}$. Let us first show that, for any $\lambda$, $c_{\lambda}(\vartheta, \sum_{i=1}^N \rho_i \bar{T}_i) \in \mathbb{R}$. Due to the assumption on the $\{ \bar{T}_i \}_{i=1}^N$, such result trivially holds for $\lambda = 0$. When $\lambda \neq 0$, we have $1 + \lambda \langle \vartheta, \bar{T}_i \rangle > 0$ for any $i \in \{1,\dots,N\}$, hence $1 + \lambda \langle \vartheta, \sum_{i=1}^N \rho_i \bar{T}_i \rangle > 0$, showing $c_{\lambda}(\vartheta, \sum_{i=1}^N \rho_i \bar{T}_i) \in \mathbb{R}$. 

    Case $(i)$: Let $\lambda = 0$. By Lemma \ref{lemma:convexityCoupling}, 
    \begin{equation}
        \varphi_{\lambda}(\vartheta) - \sum_{i=1}^N \rho_i c_{\lambda}(\vartheta, \bar{T}_i) = \varphi_{\lambda}(\vartheta) - c_{\lambda}\left(\vartheta, \sum_{i=1}^N \rho_i \bar{T}_i\right),\,\forall \vartheta \in \domain \varphi_{\lambda}.
        \label{eq:proofprop6i}
    \end{equation}
    By Proposition \ref{prop:optResult1}, $\vartheta_*$ minimizes the right-hand side of \eqref{eq:proofprop6i}, showing the result.

    Case $(ii)$: Let $\lambda < 0$. Using Lemma \ref{lemma:convexityCoupling} in the case $\lambda < 0$, we get that 
    \begin{equation}
        \label{eq:intermediateOptResult2}
        \varphi_{\lambda}(\vartheta) - \sum_{i=1}^N \rho_i c_{\lambda}(\vartheta, \bar{T}_i) \leq \varphi_{\lambda}(\vartheta) - c_{\lambda}\left(\vartheta, \sum_{i=1}^N \rho_i \bar{T}_i\right),\,\forall \vartheta \in \domain \varphi_{\lambda}.
    \end{equation}
    Using the result of Proposition \ref{prop:optResult1}, we get that $\vartheta_*$ minimizes the right-hand side of \eqref{eq:intermediateOptResult2}. Moreover, as $\sum_{i=1}^N \rho_i \bar{T}_i \in \partial^{c_{\lambda}} \varphi_{\lambda}(\vartheta_*)$,
    \begin{equation}
        \label{eq:intermediateOptResult3}
        \varphi_{\lambda}(\vartheta_*) - c_{\lambda}\left(\vartheta_*, \sum_{i=1}^N \rho_i \bar{T}_i\right) = - \varphi_{\lambda}^{c_{\lambda}}\left(\sum_{i=1}^N \rho_i \bar{T}_i\right).
    \end{equation}
    Using the inequality in Equation \eqref{eq:intermediateOptResult2} and the identity in Equation \eqref{eq:intermediateOptResult3} yields the upper-bound property.

    Case $(iii)$: Let $\lambda > 0$. Using Lemma \ref{lemma:convexityCoupling} for $\lambda>0$ yields
    \begin{equation}
        \varphi_{\lambda}(\vartheta) - c_{\lambda}\left(\vartheta, \sum_{i=1}^N \rho_i \bar{T}_i\right) \leq \varphi_{\lambda}(\vartheta) - \sum_{i=1}^N \rho_i c_{\lambda}(\vartheta, \bar{T}_i),\,\forall \vartheta \in \domain \varphi_{\lambda}.
    \end{equation}
    The parameter $\vartheta_*$ minimizes the left-hand side of the above due to Proposition \ref{prop:optResult1}. This implies in particular that
    \begin{equation}
        \varphi_{\lambda}(\vartheta_*) - c_{\lambda}\left(\vartheta_*, \sum_{i=1}^N \rho_i \bar{T}_i\right) \leq \varphi_{\lambda}(\vartheta) - \sum_{i=1}^N \rho_i c_{\lambda}(\vartheta, \bar{T}_i),\,\forall \vartheta \in \domain \varphi_{\lambda}.
    \end{equation}
    Using Equation \eqref{eq:intermediateOptResult3}, which remains true for $\lambda>0$, we obtain the lower bound result.
\end{proof}

\section{Statistical problems over the $\lambda$-exponential family}
\label{section:applicationsStatisticalProblems}

We now leverage all the previous notions and new technical results from Section \ref{section:firstResults} to tackle variational inference and maximum likelihood estimation problems within the $\lambda$-exponential family. We derive novel optimality conditions and algorithms to solve these problems. Finally, we compare and discuss our new findings on the $\lambda$-exponential family with known results on the standard exponential family.

\subsection{Variational inference through Rényi divergence minimization}
\label{subsection:vi}

We consider in this section the problem
\begin{equation}
    \tag{$P_{\text{VI}}$}
    \label{pblm:VI}
    \minimize_{q_{\vartheta} \in \mathcal{Q}_{\lambda}} RD_{\alpha}(\pi, q_{\vartheta}),
\end{equation}
where $\lambda + \alpha = 1$, under Assumption \ref{assumption:expFamily}. Notice that in the case $\lambda = 0$, $\alpha = 1$, Problem \eqref{pblm:VI} corresponds to the minimization of the inclusive Kullback-Leibler divergence over the standard exponential family. We introduce the following additional assumption.

\begin{assumption}
    \label{assumption:targetExp}
    The target $\pi$ is in $\mathcal{P}(\mathcal{X}, m)$ and is $\mathcal{Q}_{\lambda}$-compatible.
\end{assumption}

\subsubsection{Optimality conditions}

We now derive novel optimality conditions for Problem \eqref{pblm:VI}. This is done by leveraging the technical optimality conditions introduced in Section \ref{subsection:technicalOptResults}. These conditions can be seen as a moment-matching conditions on escort probabilities and are discussed in greater extent in Section \ref{subsection:discussion}. We also show that they can be used straightforwardly in the case of Student distributions.

\begin{proposition}
    \label{prop:optimalityRenyi}
    Suppose that Assumptions \ref{assumption:expFamily}, \ref{assumption:supportCondition}, and \ref{assumption:targetExp} are satisfied. If $\vartheta_* \in \domain \varphi_{\lambda}$ is such that
    \begin{equation}
        q_{\vartheta_*}^{(\alpha)}(T) = \pi^{(\alpha)}_{|S_{\lambda}}(T),
    \end{equation}
    then $\vartheta_*$ is a solution to Problem \eqref{pblm:VI}
\end{proposition}

\begin{proof}
    We first prove that $c_{\lambda}(\vartheta, \pi^{(\alpha)}_{|S_{\lambda}}(T)) \in \mathbb{R}$ for any $\vartheta \in \domain \varphi_{\lambda}$. Due to Assumption \ref{assumption:supportCondition}, it is sufficient to check that $c_{\lambda}(\vartheta, \pi^{(\alpha)}_{|S_{\vartheta}}(T)) \in \mathbb{R}$ for any $\vartheta \in \domain \varphi_{\lambda}$. We then get this first result from Assumption \ref{assumption:targetExp} and Lemma \ref{lemma:wellPosednessCouplingMoment}. Assumption \ref{assumption:targetExp} also implies that $H_{\alpha}(\pi)$ is finite.
    
    Then, rewriting the Rényi divergence using Proposition \ref{prop:rewritingRenyi} shows that solving Problem \eqref{pblm:VI} is equivalent to solving
    \begin{equation}
        \minimize_{\vartheta \in \domain \varphi_{\lambda}} \, \varphi_{\lambda}(\vartheta) - c_{\lambda}(\vartheta, \pi_{|S_{\lambda}}^{(\alpha)}(T)).
    \end{equation}
    We can thus apply Proposition \ref{prop:optResult1} to obtain that $\vartheta \in \domain \varphi_{\lambda}$ is a solution of Problem \eqref{pblm:VI} if and only if $\pi^{(\alpha)}_{|S_{\lambda}}(T) \in \partial^{c_{\lambda}} \varphi_{\lambda}(\vartheta)$. Using the description of $\partial^{c_{\lambda}} \varphi_{\lambda}(\vartheta)$ from Proposition \ref{prop:gradientMoment}, we get that any $\vartheta_*$ satisfying the assumptions of this proposition is such that  $\pi^{(\alpha)}_{|S_{\lambda}}(T) \in \partial^{c_{\lambda}} \varphi_{\lambda}(\vartheta_*)$, hence a solution to Problem \eqref{pblm:VI}.
\end{proof}

\begin{corollary}
    \label{corollary:StudentVI}
    Consider a target $\pi \in \mathcal{P}(\mathbb{R}^d,dx)$, the family of Student distributions in dimension $d$ with $\nu$ degrees of freedom $\mathcal{T}_{\nu}^d$ and the family of Gaussian distributions $\mathcal{G}^d$. 
    \begin{itemize}
    \item[$(i)$] If the escort probability $\pi^{(\alpha)}$ exists and has finite first and second order moments for $\alpha = 1 + \frac{2}{\nu+d}$, then the distribution $q_{\mu_*, \Sigma_*} \in \mathcal{T}_{\nu}^d$ such that
    \begin{equation}
    \label{eq:corollaryStudentVI}
        \begin{cases}
            \mu_* &= \pi^{(\alpha)}(x),\\
            \Sigma_* &= \pi^{(\alpha)}(x x^{\top}) -\mu_* \mu_*^{\top},
        \end{cases}
    \end{equation}
    minimizes $q \longmapsto RD_{\alpha}(\pi, q)$ over $\mathcal{T}_{\nu}^d$.
\item[$(ii)$] If $\pi$ has finite first and second order moments, then the distribution $q_{\mu_*, \Sigma_*} \in \mathcal{G}^d$ such that $\mu_*$ and $\Sigma_*$ satisfy Equation \eqref{eq:corollaryStudentVI} with $\alpha = 1$ minimizes $q \longmapsto KL(\pi, q)$ over $\mathcal{G}^d$.
\end{itemize}
\end{corollary}

\subsubsection{An iterative variational inference algorithm}

In order to resolve the optimality conditions given in Proposition \ref{prop:optimalityRenyi}, we propose in this section an iterative approach relying on the following novel update operator, parametrized by the target $\pi$ and a stepsize $\tau>0$.

\begin{definition}
\label{def:proxVI}
Consider $\alpha > 0, \lambda = 1 - \alpha$ and the $\lambda$-exponential family $\mathcal{Q}_{\lambda}$ under Assumptions \ref{assumption:expFamily} and \ref{assumption:supportCondition}. Consider a target $\pi \in \mathcal{P}(\mathcal{X},m)$ satisfying Assumption \ref{assumption:targetExp}. For $\tau > 0$, we define the \emph{operator $P_{\tau}^{\pi}$} such that $\vartheta_{P} = P_{\tau}^{\pi}(\vartheta')$ satisfies
\begin{equation}
    q_{\vartheta_P}^{(\alpha)}(T) = \frac{\tau}{1 + \tau} \pi_{|S_{\lambda}}^{(\alpha)}(T) + \frac{1}{1+\tau} q_{\vartheta'}^{(\alpha)}(T),\,\forall \vartheta' \in \domain \varphi_{\lambda}.
\end{equation}
\end{definition}
The above operator shares close links with the proximal operator introduced in Definition \ref{def:proximalOperatorGen}, as shown below.

\begin{proposition}
    \label{prop:proxVI}
    Consider $\alpha > 0$, $\lambda = 1 - \alpha$, the $\lambda$-exponential family $\mathcal{Q}_{\lambda}$, and a target distribution $\pi \in \mathcal{P}(\mathcal{X},m)$ such that Assumptions \ref{assumption:expFamily}, \ref{assumption:supportCondition}, and \ref{assumption:targetExp} are satisfied. Let $\tau > 0$. Then, 
    \begin{itemize}
        \item[$(i)$] If $\lambda = 0$, then 
        \begin{equation}
            \forall \vartheta' \in \domain \varphi_{\lambda} \quad P_{\tau}^{\pi}(\vartheta') = \prox_{\tau}^{KL(\pi, q_{\cdot})}(\vartheta')
        \end{equation}
       \item[$(ii)$] If $\lambda < 0$ (resp.~$\lambda > 0$), then $P_{\tau}^{\pi}(\vartheta')$ approximates $\prox_{\tau}^{RD_{\alpha}(\pi, q_{\cdot})}(\vartheta')$ in the sense that it minimizes an upper bound (resp.~lower bound) of the proximal loss 
       \begin{equation}
            \vartheta \longmapsto RD_{\alpha}(\pi,q_{\vartheta}) + \frac{1}{\tau} RD_{\alpha}(q_{\vartheta'},q_{\vartheta}).
       \end{equation}
    \end{itemize}
    
\end{proposition}

\begin{proof}
We begin by decomposing the objective function appearing in the computation of $\prox_{\tau}^{RD_{\alpha}(\pi,q_{\cdot})}(\vartheta')$.
    \begin{align*}
        &RD_{\alpha}(\pi, q_{\vartheta}) + \frac{1}{\tau} RD_{\alpha}(q_{\vartheta'}, q_{\vartheta})\\ 
        &= \varphi_{\lambda}(\vartheta) - c_{\lambda}(\vartheta, \pi^{(\alpha)}_{|S_{\lambda}}(T)) - H_{\alpha}(\pi_{|S_{\lambda}}) + \frac{1}{\tau}\left(\varphi_{\lambda}(\vartheta) - c_{\lambda}(\vartheta, q_{\vartheta'}^{(\alpha)}(T)) + \psi_{\lambda}(\vartheta')\right) \\
        &=\left( \frac{1 + \tau}{\tau} \right) \left( \varphi_{\lambda}(\vartheta) - \frac{\tau}{1+\tau}c_{\lambda}(\vartheta, \pi^{(\alpha)}_{|S_{\lambda}}(T)) - \frac{1}{1+\tau} c_{\lambda}(\vartheta, q_{\vartheta'}^{(\alpha)}(T)) \right) - H_{\alpha}(\pi_{|S_{\lambda}}) + \frac{1}{\tau}\psi_{\lambda}(\vartheta').
    \end{align*}
    If we conserve only the terms depending on the variable $\vartheta$ and ignore the positive multiplicative factor, we thus obtain that $\prox_{\tau}^{RD_{\alpha}(\pi,q_{\cdot})}(\vartheta')$ is the set of solutions of the problem
    \begin{equation*}
        \minimize_{\vartheta \in \domain \varphi_{\lambda}} \varphi_{\lambda}(\vartheta) - \frac{\tau}{1+\tau}c_{\lambda}(\vartheta, \pi^{(\alpha)}_{|S_{\lambda}}(T)) - \frac{1}{1+\tau} c_{\lambda}(\vartheta, q_{\vartheta'}^{(\alpha)}(T)).
    \end{equation*}

    Now, remark that due to Assumption \ref{assumption:supportCondition} and Lemma \ref{lemma:wellPosednessCouplingMoment}, $c_{\lambda}(\vartheta, q_{\vartheta'}^{(\alpha)}(T)) \in \mathbb{R}$ for any $\vartheta \in \domain \varphi_{\lambda}$. The same holds with the term $c_{\lambda}(\vartheta, \pi_{|S_{\lambda}}^{(\alpha)}(T))$ by Assumption \ref{assumption:targetExp}. Since $\frac{\tau}{1 + \tau} + \frac{1}{1+\tau} = 1$ and all the involved terms are positive, we can thus apply Propositions \ref{prop:optResult2} and \ref{prop:gradientMoment} to get the result.    
\end{proof}

We are now ready to state our algorithm to solve Problem \eqref{pblm:VI}. We then study its convergence properties.

\begin{algorithm}[H]
     Let $\vartheta_0 \in \domain \varphi_{\lambda}$, and a sequence $\{ \tau_k\}_{k \in \mathbb{N}}$ of step parameters in $\mathbb{R}_{++}$.\\
     \For{$k = 0,\dots$}{
     Update $\vartheta_{k+1}$ using $P_{\tau_k}^{\pi}$, that is such that
     \begin{equation}
        \label{eq:proxLikeIterateVI}
        q_{\vartheta_{k+1}}^{(\alpha)} = \frac{\tau_k}{1+\tau_k}\pi_{| \support Q_{\lambda}}^{(\alpha)}(T) + \frac{1}{1+\tau_k} q_{\vartheta_k}^{(\alpha)}(T).
     \end{equation}
     }
     \caption{Proposed algorithm to solve Problem \eqref{pblm:VI}}
     \label{alg:proxVI}
\end{algorithm}

\begin{proposition}
    \label{prop:convergenceProxVI}
    If $\vartheta_k \in \domain \varphi_{\lambda}$ for every $k\in \mathbb{N}$, then the sequence generated by Algorithm~\ref{alg:proxVI} is well-defined and 
    \begin{equation}
        q_{\vartheta_k}^{(\alpha)}(T) \xrightarrow[k \rightarrow +\infty]{} \pi_{| S_{\lambda}}^{(\alpha)}(T).
    \end{equation}
\end{proposition}

\begin{proof}
    For any $K \in \mathbb{N} \setminus \{0\}$, we have
    \begin{equation}
        q_{\vartheta_K}^{(\alpha)}(T) = \left( \prod_{k=0}^{K-1} \frac{1}{1 + \tau_k} \right) q_{\vartheta_0}^{(\alpha)}(T) + \left(1 - \prod_{k=0}^{K-1} \frac{1}{1 + \tau_k} \right) \pi_{|S_{\lambda}}(T).
    \end{equation}
    Since $\frac{1}{1 + \tau_k} \in (0,1)$ for every $k \in \mathbb{N}$, $\prod_{k=0}^{K-1} \frac{1}{1 + \tau_k} \xrightarrow[K \rightarrow +\infty]{} 0$, showing the result.
\end{proof}

\begin{remark}
    When $\lambda = 0$, Algorithm \ref{alg:proxVI} identifies with the Bregman proximal algorithm from~\cite{bauschke2003, teboulle2018}. Further, Proposition \ref{prop:convergenceProxVI} shows that Algorithm \ref{alg:proxVI} produces iterates converging to the solution of Problem \eqref{pblm:VI}.
\end{remark}

Algorithm~\ref{alg:proxVI} involves at every iteration the computation of $\pi^{(\alpha)}_{|S_{\lambda}}(T)$. This quantity is in general unavailable. Actually, the updates in Algorithm~\ref{alg:proxVI} allow to build an alternative estimate for $\pi^{(\alpha)}_{|S_{\lambda}}(T)$ at every iteration, and to combine them using a step-size parameter $\tau_k \xrightarrow[k \rightarrow +\infty]{} 0$ in the spirit of stochastic approximation algorithms. This will be illustrated in Section \ref{subsection:numericsVI}.

\subsection{Maximum likelihood estimation}
\label{subVariationalsection:mle}

We consider now the maximum likelihood problem of estimating the parameters of a distribution from the $\lambda$-exponential family $\mathcal{Q}_{\lambda}$ based on observed data $\{ x_i \}_{i=1}^N$. This problem reads as follows.

\begin{equation}
    \label{pblm:MLE}
    \tag{$P_{\text{MLE}}$}
    \maximize_{\vartheta \in \domain \varphi_{\lambda}} \sum_{i=1}^N \log q_{\vartheta}(x_i).
\end{equation}

We need the following assumption on the data to ensure the well-posedness of Problem \eqref{pblm:MLE}.

\begin{assumption}
    \label{assumption:data}
    For every $q_{\vartheta} \in \mathcal{Q}_{\lambda}$, $x_i \in S_{\vartheta}$ for every $i \in \{ 1,\dots, N\}$.
\end{assumption}

\subsubsection{Optimality conditions and approximate solutions}

We now provide novel conditions for the resolution of Problem \eqref{pblm:MLE}. These conditions are optimal in the case of the standard exponential family. In the case of the $\lambda$-exponential family, the conditions are sub-optimal and we relate them explicitly to the optimal solutions of Problem \eqref{pblm:MLE}.

\begin{proposition}
    \label{prop:optimalityMLE}
    Consider $\lambda \in \mathbb{R}$ such that $\alpha = 1 - \lambda$ is positive, and the $\lambda$-exponential family and data $\{ x_i \}_{i=1}^N$ such that that Assumptions \ref{assumption:expFamily}, \ref{assumption:supportCondition}, and \ref{assumption:data} hold. Suppose that there exists $\vartheta_* \in \domain \varphi_{\lambda}$ such that 
    \begin{equation}
        q_{\vartheta_*}^{(\alpha)}(T) = \frac{1}{N} \sum_{i=1}^N T(x_i).
    \end{equation}

    \begin{itemize}
        \item[$(i)$] If $\lambda = 0$, $\vartheta_*$ maximizes Problem \eqref{pblm:MLE}.
        \item[$(ii)$] If $\lambda < 0$, $\vartheta_*$ maximizes a lower bound of $\vartheta \longmapsto \sum_{i=1}^N \log q_{\vartheta}(x_i)$ over $\domain \varphi_{\lambda}$. Moreover,
        \begin{equation}
            \frac{1}{N}\sum_{i=1}^N \log q_{\vartheta_*}(x_i) \geq \psi_{\lambda}(\vartheta_*).
        \end{equation}
        \item[$(iii)$] If $\lambda > 0$, $\vartheta_*$ maximizes an upper bound of $\vartheta \longmapsto \sum_{i=1}^N \log q_{\vartheta}(x_i)$ over $\domain \varphi_{\lambda}$. Moreover,
        \begin{equation}
            \frac{1}{N}\sum_{i=1}^N \log q_{\vartheta}(x_i) \leq \psi_{\lambda}(\vartheta_*),\,\forall \vartheta \in \domain \varphi_{\lambda}.
        \end{equation}
    \end{itemize}
\end{proposition}

\begin{proof}
    Remark first that solving Problem \eqref{pblm:MLE} is equivalent to solving 
    \begin{equation*}
        \minimize_{\vartheta \in \domain \varphi_{\lambda}} -\frac{1}{N}\sum_{i=1}^N \log q_{\vartheta}(x_i) = \varphi_{\lambda}(\vartheta)-\frac{1}{N}\sum_{i=1}^N c_{\lambda}(\vartheta, T(x_i)).
    \end{equation*}
    Assumption \ref{assumption:data} ensures that we can apply Proposition \ref{prop:optResult2}. The result comes from the results of this Proposition and the description of $\partial^{c_{\lambda}} \varphi_{\lambda}$ and $\varphi_{\lambda}^{c_{\lambda}}$ provided in Proposition \ref{prop:gradientMoment}.
\end{proof}

\begin{corollary}
\label{corollary:studentMLE}
Consider Problem \eqref{pblm:MLE} with data points $x_i \in \mathbb{R}^d$ for $i = \{1,\dots, N\}$. 
\begin{itemize}
    \item[$(i)$] If we consider Problem \eqref{pblm:MLE} over the family of Student distributions in dimension $d$ with $\nu$ degrees of freedom $\mathcal{T}_{\nu}^d$, the distribution $q_{\mu_*, \Sigma_*} \in \mathcal{T}_{\nu}^d$ such that
    \begin{equation}
        \label{eq:corollaryMLE_OptCond}
        \begin{cases}
            \mu_* &= \frac{1}{N} \sum_{i=1}^N x_i,\\
            \Sigma_* &= \frac{1}{N}\sum_{i=1}^N x_i x_i^{\top} - \mu_* \mu_*^{\top},
        \end{cases}
    \end{equation}
    maximizes a lower bound of Problem \eqref{pblm:MLE}. We also get that
    \begin{align}
        \frac{1}{N}\sum_{i=1}^N \log q_{\mu_*, \Sigma_*}(x_i) \geq \frac{1}{2}\logdet(\Sigma_*) + C,
    \end{align}
    where the constant $C$ depends only on $d$ and $\nu$.
    \item[$(ii)$] If we consider Problem \eqref{pblm:MLE} over the family of Gaussian distributions $\mathcal{G}^d$, the distribution $q_{\mu_*, \Sigma_*} \in \mathcal{G}^d$ with $\mu_*, \Sigma_*$ satisfying Equation \eqref{eq:corollaryMLE_OptCond} maximizes Problem \eqref{pblm:MLE}.
\end{itemize}
\end{corollary}

\subsubsection{An iterative algorithm for maximum likelihood estimation}

We now propose a new iterative algorithm to reach the (sub-optimal) solutions to Problem \eqref{pblm:MLE}, as characterized in Proposition \ref{prop:optimalityMLE}. To do so, we first introduce the following operator.

\begin{definition}
\label{def:proxMLE}
Consider $\alpha > 0, \lambda = 1 - \alpha$ and the $\lambda$-exponential family $\mathcal{Q}_{\lambda}$ under Assumptions \ref{assumption:expFamily} and \ref{assumption:supportCondition}. Consider data points $\{x_i\}_{i=1}^N$ satisfying Assumption \ref{assumption:data}. For $\tau > 0$, we define the \emph{operator $P_{\tau}^{\{x_i\}_{i=1}^N}$} such that for any $\vartheta' \in \domain \varphi_{\lambda}$, $\vartheta_{P} = P_{\tau}^{\{x_i\}_{i=1}^N}(\vartheta')$ satisfies
\begin{equation}
    q_{\vartheta_P}^{(\alpha)}(T) = \frac{N\tau}{1 + N\tau}\sum_{i=1}^N T(x_i) + \frac{1}{1 + N\tau} q_{\vartheta'}^{(\alpha)}(T).
\end{equation}
\end{definition}

This operator can be related to the proximal operator from Definition \ref{def:proximalOperatorGen}, as we show now.

\begin{proposition}
    \label{prop:proxMLE}
    Consider a $\lambda$-exponential family $\mathcal{Q}_{\lambda}$ with $\lambda \in \mathbb{R}$ such that $\alpha = 1-\lambda$ is positive, and observed data $\{ x_i \}_{i=1}^N$ such that Assumptions \ref{assumption:expFamily}, \ref{assumption:supportCondition}, and \ref{assumption:data} are satisfied. Let $\tau > 0$. Then,
    \begin{itemize}
        \item[$(i)$] If $\lambda = 0$, then 
        \begin{equation}
            P_{\tau}^{\{x_i\}_{i=1}^N}(\vartheta') = \prox_{\tau}^{-\sum_{i=1}^N \log q_{\cdot}(x_i)}(\vartheta'),\,\forall \vartheta' \in \domain \varphi_{\lambda}.
        \end{equation}
       \item[$(ii)$] If $\lambda < 0$ (resp.~$\lambda > 0$), then $P_{\tau}^{\{x_i\}_{i=1}^N}(\vartheta')$ approximates $\prox_{\tau}^{-\sum_{i=1}^N \log q_{\cdot}(x_i)}(\vartheta')$ in the sense that it minimizes an upper bound (resp.~lower bound) of the proximal loss 
       \begin{equation}
            \vartheta \longmapsto -\sum_{i=1}^N \log q_{\vartheta}(x_i) + \frac{1}{\tau} RD_{\alpha}(q_{\vartheta'},q_{\vartheta}).
       \end{equation}
    \end{itemize}
\end{proposition}

\begin{proof}
    We first decompose the objective function appearing in $\prox_{\tau}^{-\sum_{i=1}^N \log q_{\cdot}(x_i)}(\vartheta')$:
    \begin{align*}
        &-\sum_{i=1}^N \log q_{\vartheta}(x_i) + \frac{1}{\tau} RD_{\alpha}(q_{\vartheta'}, q_{\vartheta})\\
        &= N \varphi_{\lambda}(\vartheta) - \sum_{i=1}^N c_{\lambda}(\vartheta, T(x_i)) + \frac{1}{\tau} \left(\varphi_{\lambda}(\vartheta) - c_{\lambda}(\vartheta, q_{\vartheta'}^{(\alpha)}(T)) + \psi_{\lambda}(\vartheta')\right)\\
        &=\left(\frac{1+ N\tau}{\tau} \right) \left( \varphi_{\lambda}(\vartheta) - \sum_{i=1}^N \frac{\tau}{1 + N\tau}c_{\lambda}(\vartheta, T(x_i)) - \frac{1}{1+N\tau}c_{\lambda}(\vartheta, q_{\vartheta'}^{(\alpha)}(T)) \right) + \frac{1}{\tau}\psi_{\lambda}(\vartheta').
    \end{align*}

    The above calculation shows that computing $\prox_{\tau}^{\{x_i\}}(\vartheta')$ is equivalent to solving 
    \begin{equation*}
        \minimize_{\vartheta \in \domain \varphi_{\lambda}} \varphi_{\lambda}(\vartheta) - \sum_{i=1}^N \frac{\tau}{1 + N\tau}c_{\lambda}(\vartheta, T(x_i)) - \frac{1}{1+N\tau}c_{\lambda}(\vartheta, q_{\vartheta'}^{(\alpha)}(T)).
    \end{equation*}
    Then, one can conclude as in the proof of Proposition \ref{prop:proxVI}.
\end{proof}

We are now ready to introduce our algorithm to reach the solutions given in Proposition \ref{prop:optimalityMLE}, and as such, solving (approximatly) Problem \eqref{pblm:MLE}.

\begin{algorithm}[H]
     Let $\vartheta_0 \in \domain \varphi_{\lambda}$, and a sequence $\{ \tau_k\}_{k \in \mathbb{N}}$ of step parameters in $\mathbb{R}_{++}$.\\
     \For{$k = 0,\dots$}{
     Update $\vartheta_{k+1}$ using $P_{\tau_k}^{\{x_i\}_{i=1}^N}$, that is such that
     \begin{equation}
        \label{eq:proxLikeIterateMLE}
        q_{\vartheta_{k+1}}^{(\alpha)}(T) = \frac{N\tau_k}{1 + N\tau_k} \sum_{i=1}^N T(x_i) + \frac{1}{1+N\tau_k} q_{\vartheta_k}^{(\alpha)}(T).
     \end{equation}
     }
     \caption{Proposed algorithm to solve Problem \eqref{pblm:MLE}}
     \label{alg:proxMLE}
\end{algorithm}

\begin{proposition}
    \label{prop:convergenceProxMLE}
    If $\vartheta_k \in \domain \varphi_{\lambda}$ for every $k\in \mathbb{N}$, then the sequence generated by Algorithm~\ref{alg:proxMLE} is well-defined and 
    \begin{equation}
        q_{\vartheta_k}^{(\alpha)}(T) \xrightarrow[k \rightarrow +\infty]{} \frac{1}{N} \sum_{i=1}^N T(x_i).
    \end{equation}
\end{proposition}

\begin{proof}
    The proof follows the same step as the proof of Proposition \ref{prop:convergenceProxVI}.
\end{proof}

\begin{remark}
    In the case $\lambda = 0$, Algorithm \ref{alg:proxMLE} is a Bregman proximal algorithm \cite{bauschke2003, teboulle2018} that converges to the solution of Problem \eqref{pblm:MLE}. In the case $\lambda \neq 0$, we have the convergence to the sub-optimal solutions described in Proposition \ref{prop:optimalityMLE}.
\end{remark}

The algorithm obtained by applying the update of Equation \eqref{eq:proxLikeIterateMLE} can for instance be used in an online context, where all the data points are not available at every iteration. This will be illustrated in Section \ref{subsection:numericsMLE}.

\subsubsection{An expectation-maximization algorithm for mixture MLE}

We consider here a variant of Problem \eqref{pblm:MLE} where we aim at estimating the parameters of a mixture of $J \in \mathbb{N}$, $J > 0$, distributions from the $\lambda$-exponential family $\mathcal{Q}_{\lambda}$, based on observed data $\{x_i\}_{i=1}^N$. The problem is over the parameters of each component of the mixture, as well as over their weights, and reads as follows.
\begin{equation}
    \label{pblm:MixtMLE}
    \tag{$P_{\text{MLE-Mixt}}$}
    \maximize_{\substack{\xi_j \geq 0, \vartheta_j \in \domain \varphi_{\lambda},\,j=1,\dots,J\\\sum_{j=1}^J \xi_j = 1}} \sum_{i=1}^N \log \left( \sum_{j=1}^J \xi_j q_{\vartheta_j}(x_i) \right).
\end{equation}

A standard algorithm to solve this type of problem is the expectation-maximization (EM) algorithm \cite{bishop2006}, that generates a sequence of weights $\{\xi_{j,k}\}_{k \in \mathbb{N}}$ and parameters $\{ \vartheta_{j,k}\}_{k \in \mathbb{N}}$ for $j=1,\dots,J$. For any $j=1,\dots,J$ and iteration $k \in \mathbb{N}$, we denote by $\gamma_{k,j}$ the function defined by
\begin{equation}
    \label{eq:definitionGamma}
    \gamma_{k,j}(x) = \frac{\xi_{k,j}q_{\vartheta_{k,j}}(x)}{\sum_{j'=1}^J \xi_{k,j'}q_{\vartheta_{k,j'}}(x)}.
\end{equation}
It is then possible to apply the EM algorithm in our setting, yielding updates of the form
\begin{align}
    \xi_{k+1,j} &= \frac{1}{N} \sum_{n=1}^N \gamma_{k,j}(x_n)\\
    \vartheta_{k+1,j} &= \argmax_{\vartheta \in \domain \varphi_{\lambda}} \sum_{i=1}^N \gamma_{k,j}(x_i) \log p_{\vartheta}(x_i).\label{eq:varthetaEMupdate}
\end{align}

The maximization step, often called the M-step, appearing in the update \eqref{eq:varthetaEMupdate} does not always have a closed-form. We now give a result about explicit solutions of these steps, that are possibly optimal, leveraging tools from our Proposition \ref{prop:optimalityMLE}.

\begin{proposition}
    \label{prop:EM}
    Consider a $\lambda$-exponential family $\mathcal{Q}_{\lambda}$ with $\lambda \in \mathbb{R}$ such that $\alpha = 1-\lambda$ is positive, and observed data $\{ x_i \}_{i=1}^N$ such that Assumptions \ref{assumption:expFamily}, \ref{assumption:supportCondition}, and \ref{assumption:data} are satisfied. Suppose that there exists $\vartheta_{k+1,j} \in \domain \varphi_{\lambda}$ such that
    \begin{equation}
        q_{\vartheta_{k+1,j}}^{(\alpha)} = \sum_{i=1}^N \frac{\gamma_{k,j}(x_i)}{\sum_{i'=1}^N \gamma_{k,j}(x_{i'})}  T(x_i).
    \end{equation}
    \begin{itemize}
        \item[$(i)$] If $\lambda = 0$, then $\vartheta_{k+1,j}$ exactly solves the optimization problem in the update \eqref{eq:varthetaEMupdate}.
        
       \item[$(ii)$] If $\lambda < 0$ (resp.~$\lambda > 0$), then $\vartheta_{k+1,j}$ approximates the solution of the update \eqref{eq:varthetaEMupdate} in the sense that it maximizes a lower bound (resp.~upper bound) of the considered loss.
    \end{itemize}
\end{proposition}

\begin{proof}
    The maximization problem in the update \eqref{eq:varthetaEMupdate} is the following:
    \begin{equation}
        \maximize_{\vartheta \in \domain \varphi_{\lambda}} \sum_{i=1}^N \gamma_{k,j}(x_i) \log p_{\vartheta}(x_i),
    \end{equation}
    which is equivalent, since the functions $\gamma_{k,j}$ take non-negative values, to 
    \begin{equation}
        \maximize_{\vartheta \in \domain \varphi_{\lambda}} \sum_{i=1}^N \frac{\gamma_{k,j}(x_i)}{\sum_{i'=1}^N \gamma_{k,j}(x_{i'})} \log p_{\vartheta}(x_i).
    \end{equation}
    Finally, we re-write this optimization problem as
    \begin{equation}
        \minimize_{\vartheta \in \domain \varphi_{\lambda}} \varphi_{\lambda}(\vartheta) - \sum_{i=1}^N \frac{\gamma_{k,j}(x_i)}{\sum_{i'=1}^N \gamma_{k,j}(x_{i'})} c_{\lambda}(\vartheta, T(x_i)),
    \end{equation}
    which allows to conclude the proof as in the proof of Proposition \ref{prop:optimalityMLE}.    
\end{proof}

When $\lambda = 0$, the result of Proposition \ref{prop:EM} implies that the M-steps, that is the updates of the form \eqref{eq:varthetaEMupdate}, can be solved exactly. Since $\lambda = 0$ corresponds to the exponential family, which can represent Gaussian distributions, this result recovers the EM algorithm for Gaussian mixtures presented in \cite{bishop2006}. Otherwise, the result of Proposition \ref{prop:EM} leads to an approximate EM algorithm, where the M-steps, are only approximately solved through an explicit expression. The resulting algorithm, is summarized in Algorithm \ref{alg:EM}. 

\begin{algorithm}[H]
    Let $\vartheta_{0,j} \in \domain \varphi_{\lambda}$ and $\xi_{0,j} \geq 0$ for $j=1,\dots,J$ such that $\sum_{j=1}^J \xi_{0,j} = 1$.\\
    \For{$k = 0,\dots$}{
        For every $j=1,\dots,J$, define the function $\gamma_{k,j}$ following Equation \eqref{eq:definitionGamma}, and update the parameters $\xi_{k+1,j}$ and $\vartheta_{k+1,j}$ such that they satisfy
        \begin{align}
            \xi_{k+1,j} &= \frac{1}{N} \sum_{i=1}^N \gamma_{k,j}(x_i),\\
            q_{\theta_{k+1,j}}^{(\alpha)}(T) &= \sum_{i=1}^N \frac{\gamma_{k,j}(x_i)}{\sum_{i'=1}^N \gamma_{k,j}(x_{i'})}  T(x_i).
        \end{align}     
     }
     \caption{A sub-optimal EM algorithm to solve Problem \eqref{pblm:MixtMLE}}
     \label{alg:EM}
\end{algorithm}

\subsection{Discussion and comparison with the standard exponential family}
\label{subsection:discussion}

Let us now discuss our results for maximum likelihood, variational inference, and iterative algorithms obtained for the $\lambda$-exponential family $\mathcal{Q}_{\lambda}$. 

\subsubsection{The particular case of the exponential family}

We here discuss how our theoretical results position themselves, compared to existing results for the special case $\lambda = 0$.

We recall that for an exponential family $\mathcal{Q}$ with sufficient statistics $T$ (which is the $\lambda$-exponential family with $\lambda = 0$), the densities of the members of the family are given by Equation \eqref{eq:lambdaExpFamilyDensity} with $c_0(\cdot,\cdot) = \langle \cdot, \cdot \rangle$ and the log-partition function $\varphi$. We have proven in Proposition \ref{prop:rewritingRenyi} that for any $\pi \in \mathcal{P}(\mathcal{X},m)$ such that $H_1(\pi)$, which is the Shannon entropy of $\pi$, and $\pi(T)$ are well-defined,
\begin{equation}
    KL(\pi, q_{\vartheta}) = - H_1(\pi) - \langle \vartheta, \pi(T) \rangle + \varphi(\vartheta),\,\forall \vartheta \in \domain \varphi.
\end{equation}
In Proposition \ref{prop:gradientMoment}, we also uncovered the links between the Shannon entropy and the Fenchel conjugate of the log-partition function, and showed that the moments of a distribution from the exponential family are the subgradients of the log-partition function. These facts, although scattered in the literature, are well-known. In our Propositions \ref{prop:rewritingRenyi} and \ref{prop:gradientMoment}, we generalized them to the $\lambda$-exponential family, using $H_{\alpha}$ instead of $H_1$, $RD_{\alpha}$ instead of $KL$, $\varphi_{\lambda}$ instead of $\varphi$, and escort moments instead of standard moments.

In the case of Problem \eqref{pblm:VI}, we can have the same type of correspondence. We have proven in Proposition \ref{prop:optimalityRenyi} that Problem \eqref{pblm:VI} over $\mathcal{Q}$ with $RD_{\alpha} = KL$ is solved by $\vartheta_* \in \domain \varphi$ satisfying the moment-matching condition $q_{\vartheta_*}(T) = \pi(T)$. This optimality condition was already known, see for instance \cite{bishop2006, wainwright2008, cappe2008}. In Proposition \ref{prop:optimalityRenyi}, we generalized this optimality condition under the form of a moment-matching condition on escort probabilities, that is $\pi^{(\alpha)}(T) = q_{\vartheta_*}^{(\alpha)}(T)$.

So far, the analysis we proposed for $\lambda \neq 0$ strictly generalizes the case $\lambda = 0$. In fact, it uses the same proofs for $\lambda = 0$ and $\lambda \neq 0$. Let us now review situations where this correspondence breaks. Due to the linearity of the scalar product and the convex subdifferential, we could obtain in Proposition \ref{prop:optimalityMLE} that for $\lambda = 0$, Problem \eqref{pblm:MLE} is minimized for $\vartheta_* \in \domain \varphi$ such that the moments $q_{\vartheta_*}(T)$ match the sufficient statistics of the data $\frac{1}{N} \sum_{i=1}^N T(x_i)$. In the case $\lambda \neq 0$, the similar solution obtained by plugging escort moments $q_{\vartheta_*}^{(\alpha)}(T)$ instead of standard moments  $q_{\vartheta_*}(T)$ is only sub-optimal, as shown in Proposition \ref{prop:optimalityMLE}. More precisely, these are only the minimizers of upper or lower bounds, depending on the sign of $\lambda$. The situation is similar when designing EM algorithms, as shown in Proposition \ref{prop:EM} where optimality is only attained when $\lambda=0$, in which case Algorithm \ref{alg:EM} recovers the standard EM algorithm \cite{bishop2006}. Such results are to be expected as no closed-form solution is known for this type of maximum likelihood estimation problems, and solving these problems, notably over Student-like distribution, is still an active field of research \cite{hasanasab2021, ayadi2023}. The situation is similar for the operators defined in Definitions \ref{def:proxVI} and \ref{def:proxMLE}, since they can be seen as an exact proximal operator only for $\lambda = 0$, as shown in Propositions \ref{prop:proxVI} and \ref{prop:proxMLE}.

Let us now comment about the sub-optimality of the maximum likelihood estimator proposed in Proposition \ref{prop:optimalityMLE} by relating it with the solution of Problem \eqref{pblm:VI}. Suppose that $x_i \sim \pi$ for any $i \in \{1,\dots,N\}$ and $\frac{1}{N} \sum_{i=1}^N T(x_i) \xrightarrow[N \rightarrow +\infty]{} \pi(T)$. This means that in the limit $N \rightarrow + \infty$ and when $\lambda = 0$, Problems \eqref{pblm:MLE} and \eqref{pblm:VI} have the same solution $\vartheta_*$ such that $q_{\vartheta_*}(T) = \pi(T)$. This relation between maximum likelihood estimation and minimization of a Kullback-Leibler divergence is well-known and applies in fact in a more general setting \cite{white1982}. When $\lambda \neq 0$, the sub-optimal solution of Problem \eqref{pblm:MLE} described in Proposition \ref{prop:optimalityMLE} is such that $q_{\vartheta_*}^{(\alpha)}(T) = \pi(T)$ in the large number limit $N \rightarrow +\infty$, which is different from the solution of Problem \eqref{pblm:VI} given in Proposition \ref{prop:optimalityRenyi}. Notice however that the solution $\vartheta_* \in \domain \varphi_{\lambda}$ such that $q_{\vartheta_*}^{(\alpha)}(T) = \pi(T)$ is a solution to 
\begin{equation*}
    \minimize_{\vartheta \in \domain \varphi_{\lambda}} RD_{\alpha}(\pi^{(1/\alpha)}, q_{\vartheta}).
\end{equation*}
Thus, in the large number of samples regime, the sub-optimal solution of Problem \eqref{pblm:MLE} does not solve Problem \eqref{pblm:VI} but a similar variational inference problem with a deformed target.

Finally, we remark that Assumptions \ref{assumption:expFamily} and \ref{assumption:supportCondition} prevent us from straightforwardly applying our results to the $\lambda$-exponential family when $\lambda > 0$. Indeed, such value of $\lambda$ can lead to distributions whose support depends on the parameters (see for instance the distributions studied in \cite{martins2022} and in \cite[Example 3.17]{wong2022}). Although Proposition \ref{prop:rewritingRenyi} holds even for varying support, this behavior makes optimization much more challenging.

\subsubsection{Comparing our works with existing results in optimization}

First, remark that the proximal operators used in our algorithms can be considered as generalized Bregman proximal operators, where the scalar product of $\mathcal{H}$ is replaced by the non-linear coupling $c_{\lambda}$. Indeed, it is well-known that the Kullback-Leibler divergence between two members of the same exponential family can be written as a Bregman divergence \cite{barndorff-nielsen2014}. In our case, we can rewrite the Rényi divergence $RD_{\alpha}(q_{\vartheta'}, q_{\vartheta})$ under a similar form, using $c_{\lambda}$:
\begin{equation}
    \label{eq:RenyiRewritingAsBregman}
    RD_{\alpha}(q_{\vartheta'}, q_{\vartheta}) = \varphi_{\lambda}(\vartheta) - \varphi_{\lambda}(\vartheta') - c_{\lambda}(\vartheta, q_{\vartheta'}^{(\alpha)}(T)) + c_{\lambda}(\vartheta', q_{\vartheta'}^{(\alpha)}(T)),
\end{equation}
with $q_{\vartheta'}^{(\alpha)}(T) \in \partial^{c_{\lambda}} \varphi_{\lambda}(\vartheta')$, using Equation \eqref{eq:fenchelYoung} and Proposition \ref{prop:gradientMoment}. The particular re-writing of Equation \eqref{eq:RenyiRewritingAsBregman} was established in \cite{wang2022}.

Propositions \ref{prop:proxVI} and \ref{prop:proxMLE} show that the operators that we proposed in Definitions \ref{def:proxVI} and \ref{def:proxMLE} to build our algorithms are approximating a proximal operator when $\lambda \neq 0$. We are not aware of any optimization algorithms stated directly in a generalized convexity framework (i.e.,~a generalization of standard convexity theory using modified scalar product as in \cite{delara2020, lefranc2022, fajardo2022}, or modified subgradient as in \cite{bednarczuk2022}). Although our operators are not exactly proximal operators (except for $\lambda = 0$), they may be a first step leading to such algorithms. Note however that our construction heavily depends on the objective function having an expression like the ones described in Propositions \ref{prop:optResult1} and \ref{prop:optResult2}.

The authors of \cite{kainth2022} also faced the difficulty of computing proximal operators of the form introduced in Definition \ref{def:proximalOperatorGen}. While we proposed ad hoc operators that are shown to be sub-optimal solutions to these optimization problems in Propositions \ref{prop:proxVI} and \ref{prop:proxMLE}, they took another route. Indeed, they studied a continuous time Riemannian gradient flow, whose metric is given by the corrected Hessian of a function that is convex in the sense of the coupling $c_{\lambda}$. Note that the authors consider other types of objective functions than we did. They consider convex and differentiable objectives, while we consider specifically maximum likelihood and variational inference problems, whose objectives are not necessarily convex.

In the context of variational inference, a gradient descent algorithm within the geometry induced by the Kullback-Leibler divergence is studied in \cite{guilmeau2022} for the minimization of the Rényi divergence with $\alpha \in (0,1]$ over the standard exponential family, amounting to $\lambda = 0$. A gradient descent algorithm to minimize the $\chi^2$ divergence, which is linked to the Rényi divergence with $\alpha = 2$ over the exponential family has also been proposed in \cite{akyildiz2021} for adaptive importance sampling \cite{bugallo2017adaptive}. In this work, we have only considered the setting $\lambda + \alpha = 1$, imposing a strict relation between the approximating family and the divergence.

\section{Numerical experiments}
\label{section:numerics}

We now illustrate our findings through numerical experiments. Our examples are designed as proof-of-concepts, illustrating the advantage of considering the $\lambda$-exponential family, instead of the standard exponential one, in simple situations. To do so, we consider instances of Problems \eqref{pblm:VI}, \eqref{pblm:MLE}, and \eqref{pblm:MixtMLE} where the approximating family $\mathcal{Q}_{\lambda}$ is the Student family $\mathcal{T}_{\nu}^d$ (see Section \ref{subsection:studentDistributions}). We remind that this amounts to setting $\lambda = -\frac{2}{\nu+d}$ (see Proposition \ref{prop:studentFamily} $(i)$). In our comparisons, we will also consider the limiting case of Gaussian distributions, obtained by setting $\nu = +\infty$, in which case $\lambda = 0$. For pedagogical purpose, in all examples, the target distribution (in case of variational inference problem) and the distribution generating the samples (in case of maximum likelihood problem) is also a Student density, denoted $\pi \in \mathcal{T}_{\nu_{\pi}}^d$, and parametrized by $\nu_{\pi} > 0$ degrees of freedom, location parameter $\mu_{\pi} \in \mathbb{R}^d$ and shape matrix $\Sigma_{\pi} \in \mathcal{S}_{++}^d$. This controlled setting allows to access $\pi$ and its escort $\pi^{(\alpha)}$, sample from them, and compute Rényi divergences, making it possible to assess quantitatively the results.

\subsection{A variational inference problem with Student approximating densities}
\label{subsection:numericsVI}

We start our experiments by an instance of Problem \eqref{pblm:VI} described as
\begin{equation}
    \label{pblm:VI-stud}
    \tag{$P_{\text{VI-Student}}$}
    \minimize_{q_{\mu, \Sigma} \in \mathcal{T}_{\nu}^d} RD_{\alpha}(\pi, q_{\mu, \Sigma}),
\end{equation}
where $\alpha = 1 + \frac{2}{\nu + d}$, in light of $\mathcal{T}_{\nu}^d$ being a $\lambda$-exponential family with $\lambda = - \frac{2}{\nu+d}$ (see Proposition \ref{prop:studentFamily} $(ii)$) and the optimality result of Proposition \ref{prop:optimalityRenyi}. The optimality conditions of Problem \eqref{pblm:VI-stud} are given in Equation \eqref{eq:corollaryStudentVI}. These conditions amount to setting $\mu$ and $\Sigma$ such that the first and second order moments of $q_{\mu,\Sigma}^{(\alpha)}$ match those of the escort of the target $\pi$, that is $\pi^{(\alpha)}$. 

By Proposition \ref{prop:escortStudent}, if $\alpha = 1 + \frac{2}{\nu + d}$ for some $\nu > 0$, then $\pi^{(\alpha)}$ has first and second order moments if and only if
\begin{equation}
    \label{eq:well-posednessCondition}
    \nu_{\pi} + 2 \frac{\nu_{\pi}+d}{\nu+d} > 2.
\end{equation}
We consider $\pi \in \mathcal{T}_{\nu_{\pi}}^d$ with $\nu_{\pi} \in \{1, 3, 10\}$ and $d \in \{5,20\}$. The location vector $\mu_{\pi}$ is sampled uniformly in $[-1,1]^d$ and the shape matrix $\Sigma_{\pi} \in \mathcal{S}_{++}^d$ is constructed following \cite{moré1989} with a condition number $\kappa_{\pi} \in \{ 10, 1000\}$ (i.e., a well conditioned setting, and a poorly conditioned setting). Regarding our approximating families, we experiment various degrees of freedom $\nu \in \{1, 3, 10, +\infty\}$ such that Equation \eqref{eq:well-posednessCondition} is satisfied. The case $\nu = \infty$ corresponds to a Gaussian approximating family, which is an instance of the exponential family. In contrast, for finite $\nu$, we are working within an instance of the $\lambda$-exponential family, $\lambda = - \frac{2}{\nu + d}$. Our experimental scenarios cover the matched case where $\nu = \nu_{\pi}$, as well as various mismatched cases where $\nu \neq \nu_{\pi}$.

Using the results from Section \ref{subsection:vi}, we have actually two ways to solve Problem \eqref{pblm:VI-stud}. We can either follow Corollary \ref{corollary:StudentVI} and try to directly approximate the optimality conditions of Equation \eqref{eq:corollaryStudentVI}. This requires the computation of the first and second order moments of the escort of the target. Alternatively, we can implement Algorithm \ref{alg:proxVI}. This requires the computation of the same moments, but it allows to approximate them differently at each iteration and possibly average the errors and improve the estimators. We consider the two approaches in what follows. We also consider two distinct ways to approximate the first and second order moments of the escort of the target. In Section \ref{subsec:escortMomentMatching}, we consider that exact samples from $\pi^{(\alpha)}$ are used to approximate Equation \eqref{eq:corollaryStudentVI}. This idealized setting allows to illustrate the validity of our optimality conditions with an exact sampling procedure. In Section \ref{subsec:escortMomentMatchingMALA}, we consider a more realistic situation where only the unnormalized density of the target is available. In this situation, one needs an integration procedure to approximate the moments of the escort of the target in this setting. We choose here to use a Metropolis-adjusted Langevin algorithm (MALA) \cite{roberts2002langevin}. In this setting, we consider the approximation of Equation \eqref{eq:corollaryStudentVI} as well as the implementation of Algorithm \ref{alg:proxVI} with an adaptively scaled MALA \cite{martin2012, marnissi2020}.

\subsubsection{Using samples from the target}
\label{subsec:escortMomentMatching}

Problem \eqref{pblm:VI-stud} can be solved by approximating the optimality conditions of Corollary \ref{corollary:StudentVI} using a standard Monte Carlo algorithm with samples from $\pi^{(\alpha)}$. This is feasible as, in this experiment, $\pi^{(\alpha)}$ is a Student distribution with parameters described by Proposition \ref{prop:escortStudent}. This is an idealized setting since in practical scenarios of variational inference, one does not have the possibility to sample from the escort target. This leads to the following sampling algorithm.

\begin{algorithm}[H]
     Choose an approximating family $\mathcal{T}_{\nu}^d$ and set $\alpha = 1 + \frac{2}{\nu+d}$. Choose $N \in \mathbb{N}$.\\
     \For{$k = 0,\dots$}{
     Sample $\{x_{k+1}^{(1)},\dots,x_{k+1}^{(N)}\}$ from $\pi^{(\alpha)}$.\\

     Evaluate $(\pi^{(\alpha)}(x))_{k+1}, (\pi^{(\alpha)}(xx^{\top}))_{k+1}$ by
     \begin{equation}
         \begin{cases}
             (\pi^{(\alpha)}(x))_{k+1} &= \frac{1}{kN} \sum_{l=0}^k \sum_{i=1}^{N} x_{l+1}^{(i)},\\
             (\pi^{(\alpha)}(xx^{\top}))_{k+1} &= \frac{1}{kN} \sum_{l=0}^k \sum_{i=1}^{N} x_{l+1}^{(i)} (x_{l+1}^{(i)})^{\top}.
         \end{cases}
     \end{equation}\\
     Compute $\mu_{k+1}, \Sigma_{k+1}$ following
     \begin{equation}
         \begin{cases}
             \mu_{k+1} = (\pi^{(\alpha)}(x))_{k+1},\\
             \Sigma_{k+1} = (\pi^{(\alpha)}(xx^{\top}))_{k+1} - \mu_{k+1} \mu_{k+1}^{\top}.
         \end{cases}
     \end{equation}
     }
     \caption{Solving Problem \eqref{pblm:VI-stud} by approximating \eqref{eq:corollaryStudentVI} with samples from $\pi^{(\alpha)}$}
     \label{alg:samplesFromTarget}
\end{algorithm}

We now present the results, using $N = 10 d$ samples per iteration. Figure \ref{fig:VIwithMC_d20cond10} shows the performance of Algorithm \ref{alg:samplesFromTarget}, in terms of Rényi divergence value along iterations, when setting dimension $d=20$, and condition number $\kappa_{\pi} = 10$. We observe that the best values of the Rényi divergences are obtained for the matched case $\nu = \nu_{\pi}$, which is expected. Note also that the Gaussian approximations (i.e., $\nu  = +\infty$) perform very poorly. More generally, the closer $\nu$ is to $\nu_{\pi}$, the better the performance. Remark that when $\nu_{\pi} = \nu = 1$, the values reached by the Rényi divergences are more spread around the median. This could be because the degree of freedom parameter of $\pi^{(\alpha)}$ in this case is the lowest, and hence, $\pi^{(\alpha)}$ has heavier tails. In Figure \ref{fig:VIwithMC_d20cond10_CauchyTarget}, in the case when $\nu_{\pi} = 1$, some approximating families need to be excluded to comply with Equation \eqref{eq:well-posednessCondition}. In particular, standard moment-matching, corresponding to $\nu = +\infty$ is not defined in this case. In constrast, as soon as $\nu_{\pi} > 2$, Equation \eqref{eq:well-posednessCondition} is satisfied for any $\nu > 0$, so any approximating family can be chosen, as it is done in the plots for Figures \ref{fig:VIwithMC_d20cond10_nu3Target} and \ref{fig:VIwithMC_d20cond10_nu10Target}.

\begin{figure}[t]
    \centering
    \begin{subfigure}[b]{0.32\textwidth}
        \includegraphics[width = \textwidth]{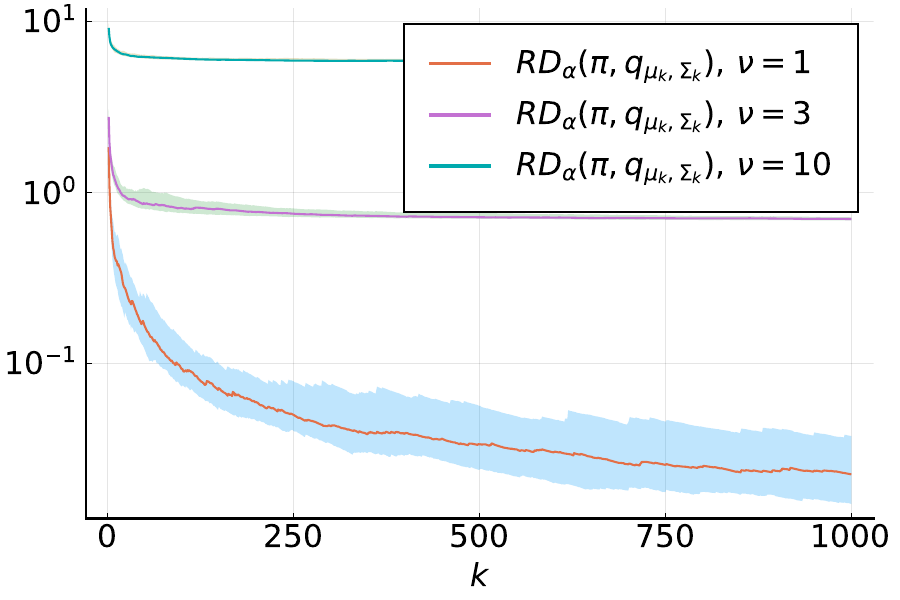}
        \caption{$\nu_{\pi}=1$, $\nu \in \{ 1, 3, 10\}$}
        \label{fig:VIwithMC_d20cond10_CauchyTarget}
    \end{subfigure}  
    \hfill
    \begin{subfigure}[b]{0.32\textwidth}
        \includegraphics[width = \textwidth]{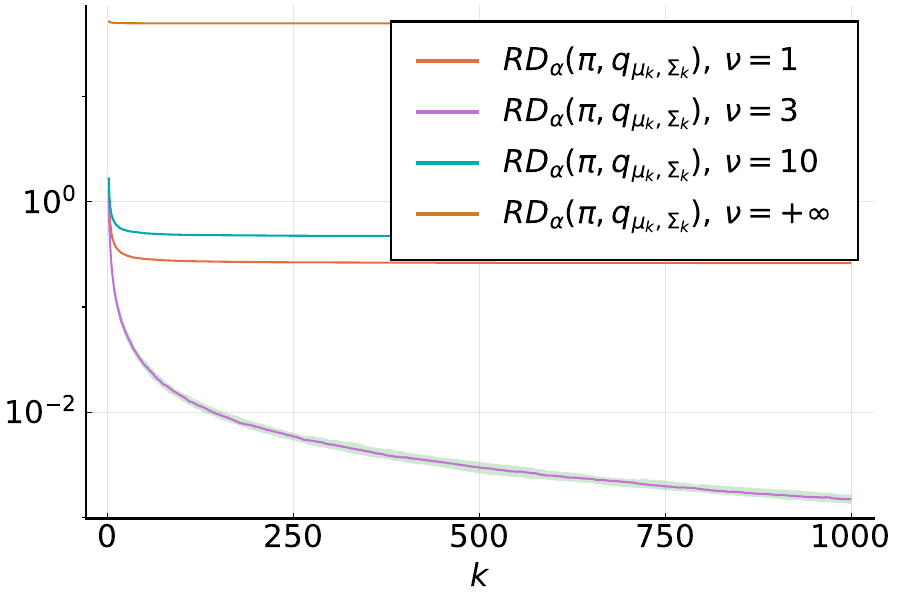}
        \caption{$\nu_{\pi}=3$, $\nu \in \{1,3,10,+\infty\}$}
        \label{fig:VIwithMC_d20cond10_nu3Target}
    \end{subfigure}  
    \hfill
    \begin{subfigure}[b]{0.32\textwidth}
        \includegraphics[width = \textwidth]{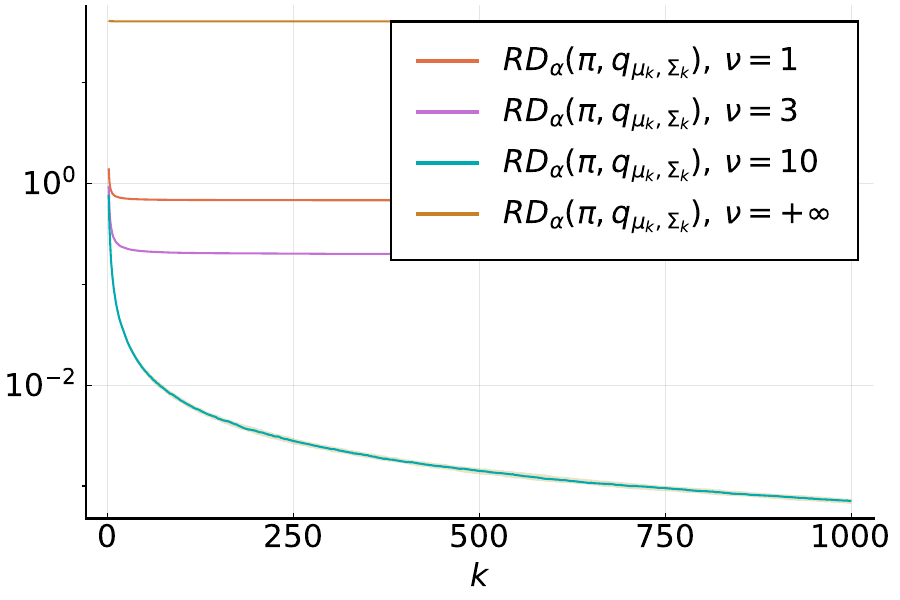}
        \caption{$\nu_{\pi}=10$, $\nu \in \{1,3,10,+\infty\}$}
        \label{fig:VIwithMC_d20cond10_nu10Target}
    \end{subfigure} 
    \caption{Rényi divergence between $q_{\mu_k, \Sigma_k} \in \mathcal{T}_{\nu}^d$ and $\pi$ in dimension $d=20$ with $\kappa_{\pi} = 10$ at every iteration $k$. The iterates $q_{\mu_k, \Sigma_k} \in \mathcal{T}_{\nu}^d$ are obtained using Algorithm \ref{alg:samplesFromTarget}. The line is the median Rényi divergence per iteration and the shaded area is the interval between the first and third quartiles. The quartiles are obtained by running the algorithm for $100$ runs of $1000$ iterations.}
    \label{fig:VIwithMC_d20cond10}
\end{figure}

In Figure \ref{fig:VIwithMC_d5cond1000}, we show performance in dimension $d=5$ and high condition number $\kappa_{\pi} = 1000$. Since the samples are generated directly from $\pi^{(\alpha)}$, the poor conditioning issue is mitigated. Since a low dimension has been used, more values of $\nu$ need to be excluded in order to comply with the condition in Equation \eqref{eq:well-posednessCondition} in the case $\nu_{\pi} = 1$.

\begin{figure}[H]
    \centering
    \begin{subfigure}[b]{0.32\textwidth}
        \includegraphics[width = \textwidth]{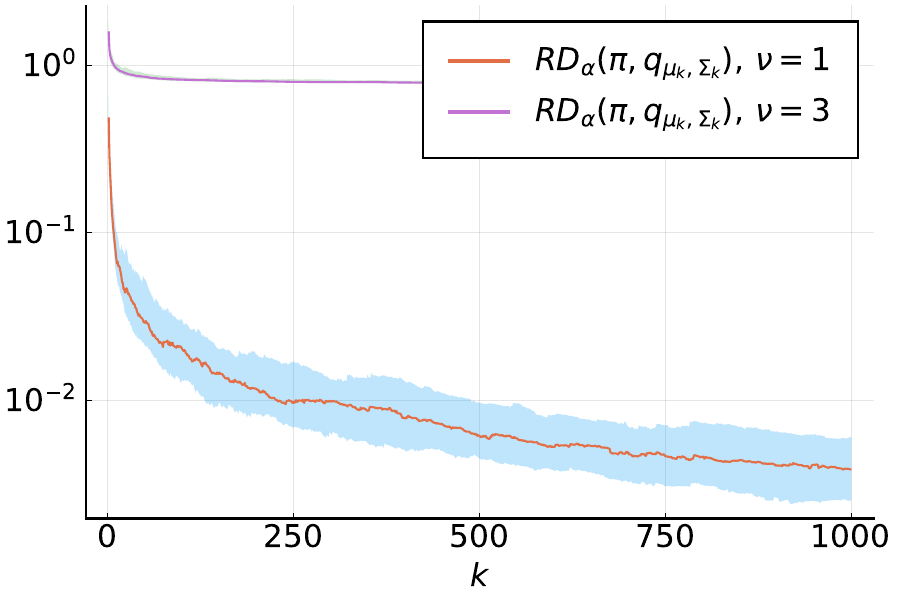}
        \caption{$\nu_{\pi}=1$, $\nu \in \{1,3\}$}
    \end{subfigure}  
    \hfill
    \begin{subfigure}[b]{0.32\textwidth}
        \includegraphics[width = \textwidth]{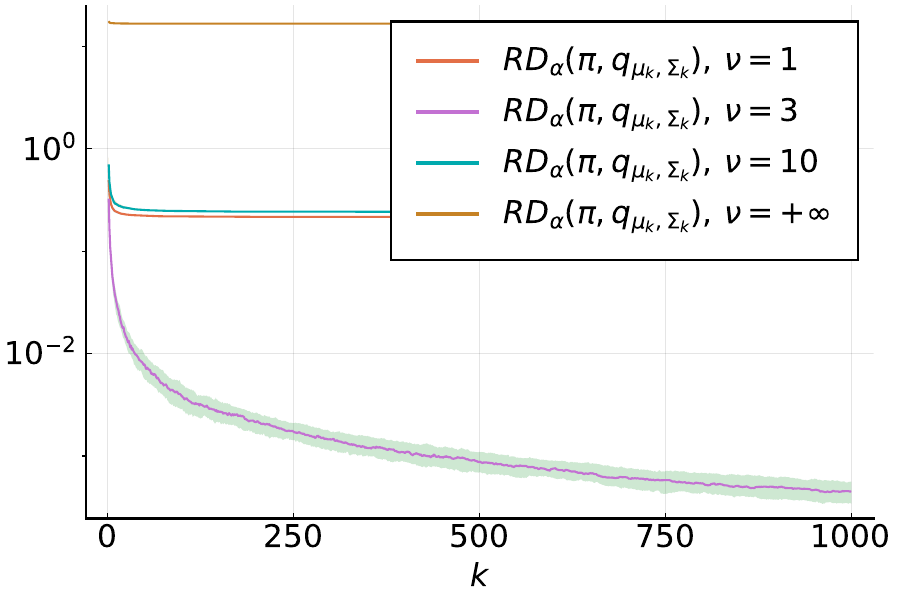}
        \caption{$\nu_{\pi}=3$, $\nu \in \{ 1, 3, 10, +\infty$\}}
    \end{subfigure}  
    \hfill
    \begin{subfigure}[b]{0.32\textwidth}
        \includegraphics[width = \textwidth]{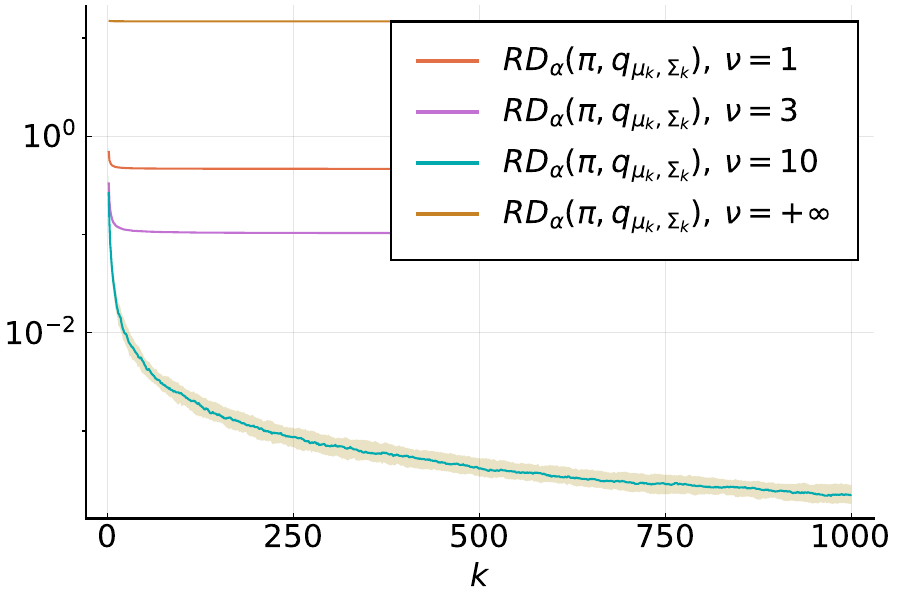}
        \caption{$\nu_{\pi}=10$, $\nu \in \{ 1, 3, 10, +\infty$\}}
    \end{subfigure} 
    \caption{Rényi divergence between $q_{\mu_k, \Sigma_k} \in \mathcal{T}_{\nu}^d$ and $\pi$ in dimension $d=5$ with $\kappa_{\pi} = 1000$ at every iteration $k$. The iterates $q_{\mu_k, \Sigma_k} \in \mathcal{T}_{\nu}^d$ are obtained using Algorithm \ref{alg:samplesFromTarget}. The line is the median Rényi divergence per iteration and the shaded area is the interval between the first and third quartiles. The quartiles are obtained by running the algorithm for $100$ runs of $1000$ iterations.}
    \label{fig:VIwithMC_d5cond1000}
\end{figure}

\subsubsection{Using Metropolis-adjusted Langevin algorithms}
\label{subsec:escortMomentMatchingMALA}

We now consider a more practical resolution of Problem \eqref{pblm:VI-stud}. We only assume that one has access to an oracle giving the unnormalized log-density $\log \Tilde{\pi}$ such that for any $x \in \mathbb{R}^d$, $\log \pi(x) = \log \Tilde{\pi}(x) - \log Z_{\pi}$ for some $Z_{\pi} > 0$. We also assume that one can evaluate the gradients $\nabla \log \tilde{\pi} (x)$ for any $x \in \mathbb{R}^d$. Under these assumptions, we propose to perform the computation of $\pi^{(\alpha)}(x), \pi^{(\alpha)}(x x^{\top})$ using a MALA approach, a particular Monte Carlo Markov Chain algorithm introduced in \cite{roberts2002langevin}. Let $x \in \mathbb{R}^d$ a starting point of the chain and suppose that we want to have samples approximately distributed following $\pi^{(\alpha)}$ for $\alpha > 0$. Then, MALA uses a proposal distribution of the form
\begin{equation}
    \label{eq:proposalPlainMALA}
    y \sim \mathcal{N}\left(x + \frac{1}{2} \sigma_d^2 \alpha A \nabla \log \tilde{\pi}(x), \sigma^2_d A\right).
\end{equation}
A typical choice is $\sigma_d^2 = \frac{0.574^2}{d^{1/3}}$, following the optimal settings described in \cite{roberts2001}. Moreover, hereabove, $A \in \mathcal{S}_{++}^d$ is the so-called scale matrix. The proposed sampled $y$ is then accepted or not following a Metropolis-Hastings step. The scale matrix $A$ in Equation \eqref{eq:proposalPlainMALA} will be chosen either as the identity matrix leading to the standard MALA algorithm, or as to reflect the curvature of $\log \pi$ around the current point $x$, as it is done in \cite{martin2012, marnissi2020} for instance. 

\paragraph{Standard MALA} We first consider the direct approximation of the optimality conditions \eqref{eq:corollaryStudentVI} by approximating the moments of $\pi^{(\alpha)}$ using samples generated with Equation \eqref{eq:proposalPlainMALA} with $A = I_d$. This leads to Algorithm \ref{alg:plainMALA} described below.

\begin{algorithm}[H]
     Choose an approximating family $\mathcal{T}_{\nu}^d$ and set $\alpha = 1 + \frac{2}{\nu+d}$. Choose $N \in \mathbb{N}$. Initialize $x_0$.\\
     \For{$k = 0,\dots$}{
     Sample $\{x_{k+1}^{(1)},\dots,x_{k+1}^{(N)}\}$ samples from $x_k$ using the MALA algorithm with proposal described in Equation \eqref{eq:proposalPlainMALA} with $A = I_d$, set $x_{k+1} = x_{k+1}^{(N)}$.\\

     Evaluate $(\pi^{(\alpha)}(x))_{k+1}, (\pi^{(\alpha)}(xx^{\top}))_{k+1}$ by
     \begin{equation}
         \begin{cases}
             (\pi^{(\alpha)}(x))_{k+1} &= \frac{1}{kN} \sum_{l=0}^k \sum_{i=1}^{N} x_{l+1}^{(i)},\\
             (\pi^{(\alpha)}(xx^{\top}))_{k+1} &= \frac{1}{kN} \sum_{l=0}^k \sum_{i=1}^{N} x_{l+1}^{(i)} (x_{l+1}^{(i)})^{\top}.
         \end{cases}
     \end{equation}\\
     Compute $\mu_{k+1}, \Sigma_{k+1}$ following
     \begin{equation}
         \begin{cases}
             \mu_{k+1} = (\pi^{(\alpha)}(x))_{k+1},\\
             \Sigma_{k+1} = (\pi^{(\alpha)}(xx^{\top}))_{k+1} - \mu_{k+1} \mu_{k+1}^{\top}.
         \end{cases}
     \end{equation}
     }
     \caption{Solving Problem \eqref{pblm:VI-stud} by approximating \eqref{eq:corollaryStudentVI} with MALA}
     \label{alg:plainMALA}
\end{algorithm}

We now turn to the experiments on the parameters described previously. We set $N = 10d$ for each experiment and initialize $x_0$ by sampling it uniformly in $[-5,5]^d$.

We display in Figure \ref{fig:VIwithMALA_d20cond10} the results obtained, for a target with low condition number $\kappa_{\pi} = 10$, in dimension $d=20$. We can observe that, as in Section \ref{subsec:escortMomentMatching}, the matched case $\nu = \nu_{\pi}$ yields the best results. Interestingly, the proposed MALA strategy works well even when the target is heavy-tailed. This could be surprising in light of negative results such as the ones in \cite{jarner2007}, but remark that we apply MALA on $\pi^{(\alpha)}$ and not $\pi$. Due to Equation \eqref{eq:well-posednessCondition}, $\pi^{(\alpha)}$ has well-defined first and second order moments, which explains the good performance of the MALA algorithm in this case. This illustrates the interest of the optimality conditions that we prove in Proposition \ref{prop:optimalityRenyi}, as they allow to handle heavy-tailed targets just as if they were light-tailed.

\begin{figure}[H]
    \centering
    \begin{subfigure}[b]{0.32\textwidth}
        \includegraphics[width = \textwidth]{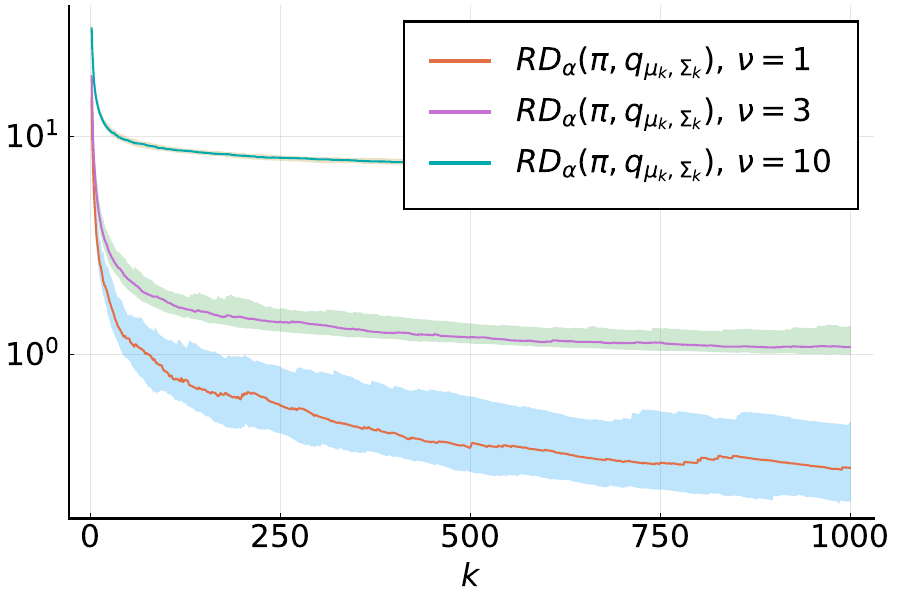}
        \caption{$\nu_{\pi}=1$, $\nu \in \{1,3,10\}$}
    \end{subfigure}  
    \hfill
    \begin{subfigure}[b]{0.32\textwidth}
        \includegraphics[width = \textwidth]{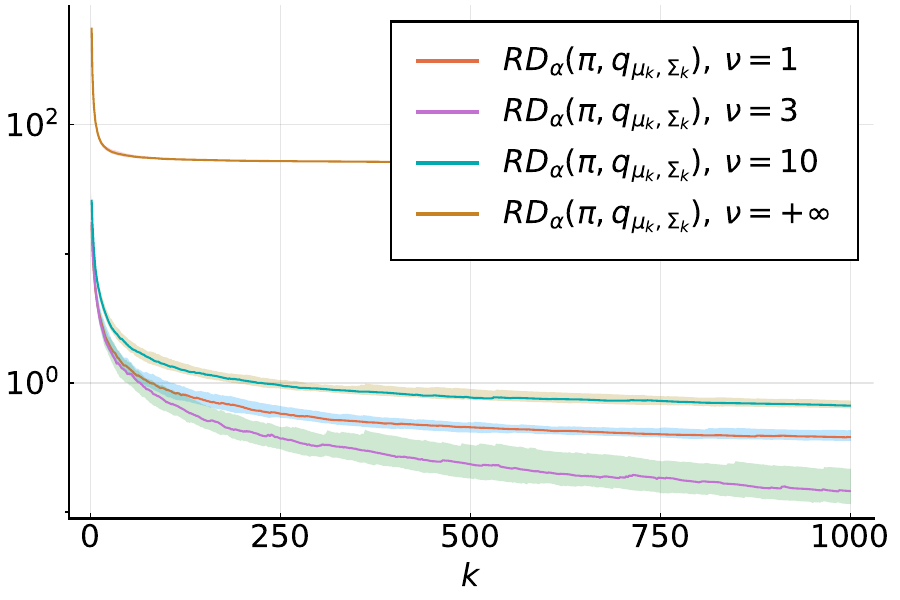}
        \caption{$\nu_{\pi}=3$, $\nu \in \{1,3,10,+\infty\}$}
    \end{subfigure}  
    \hfill
    \begin{subfigure}[b]{0.32\textwidth}
        \includegraphics[width = \textwidth]{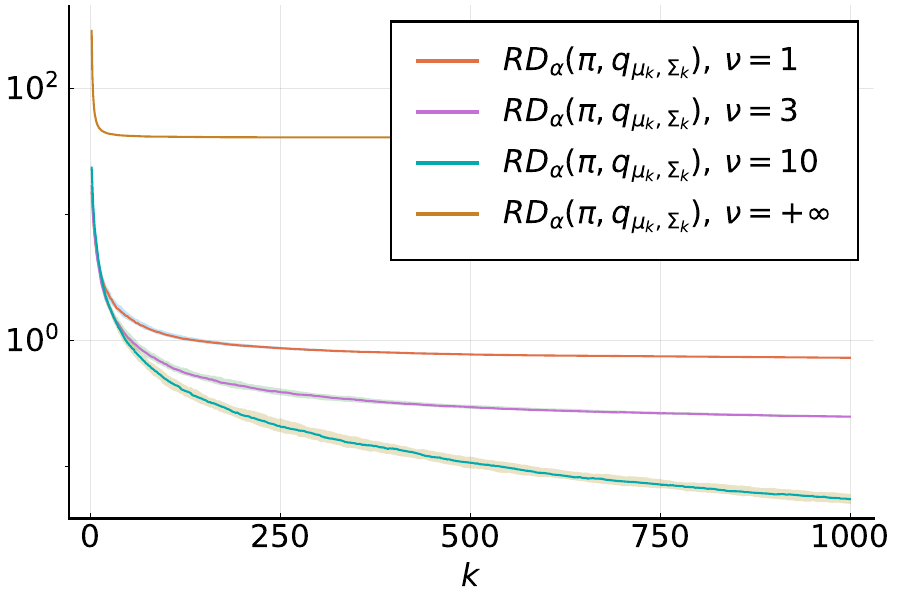}
        \caption{$\nu_{\pi}=10$, $\nu \in \{1,3,10,+\infty\}$}
    \end{subfigure} 
    \caption{Rényi divergence between $q_{\mu_k, \Sigma_k} \in \mathcal{T}_{\nu}^d$ and $\pi$ in dimension $d=20$ with $\kappa_{\pi} = 10$ at every iteration $k$. The iterates $q_{\mu_k, \Sigma_k} \in \mathcal{T}_{\nu}^d$ are obtained using Algorithm \ref{alg:plainMALA}. The line is the median Rényi divergence per iteration and the shaded area is the interval between the first and third quartiles. The quartiles are obtained by running the algorithm for $100$ runs of $1000$ iterations.}
    \label{fig:VIwithMALA_d20cond10}
\end{figure}

We now turn to a target with low dimension $d=5$, whose scale matrix has condition number $\kappa_{\pi} = 1000$. This is challenging given that the proposal distribution in our MALA algorithm is isotropic. Figure \ref{fig:VIwithMALA_d5cond1000} shows the results. Compared to the case of a low condition number in higher dimension, depicted in Figure \ref{fig:VIwithMALA_d20cond10}, we observe that the values of the Rényi divergence are higher, sometimes by an order of magnitude. The dispersal of the values around the median is also more pronounced. This can be explained by the fact that in the standard MALA algorithm, the proposals are isotropic Gaussian distributions, and hence not well adapted to the target at hand. Note also that when $\nu_{\pi}$ grows, the negative impact of having $\nu \neq \nu_{\pi}$ seems to diminish, especially compared to the situation of Figure \ref{fig:VIwithMALA_d20cond10}

\begin{figure}[H]
    \centering
    \begin{subfigure}[b]{0.32\textwidth}
        \includegraphics[width = \textwidth]{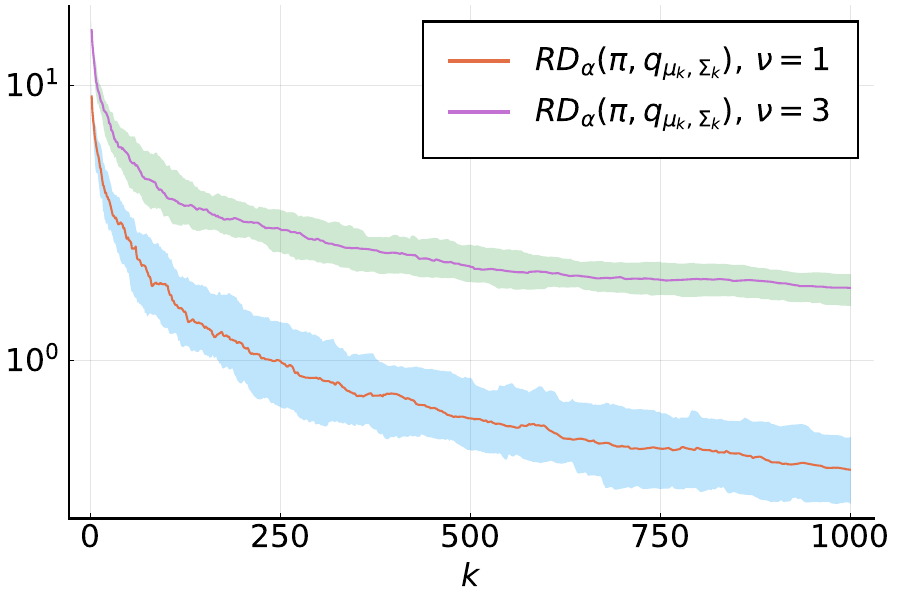}
        \caption{$\nu_{\pi}=1$, $\nu \in \{1,3\}$}
    \end{subfigure}  
    \hfill
    \begin{subfigure}[b]{0.32\textwidth}
        \includegraphics[width = \textwidth]{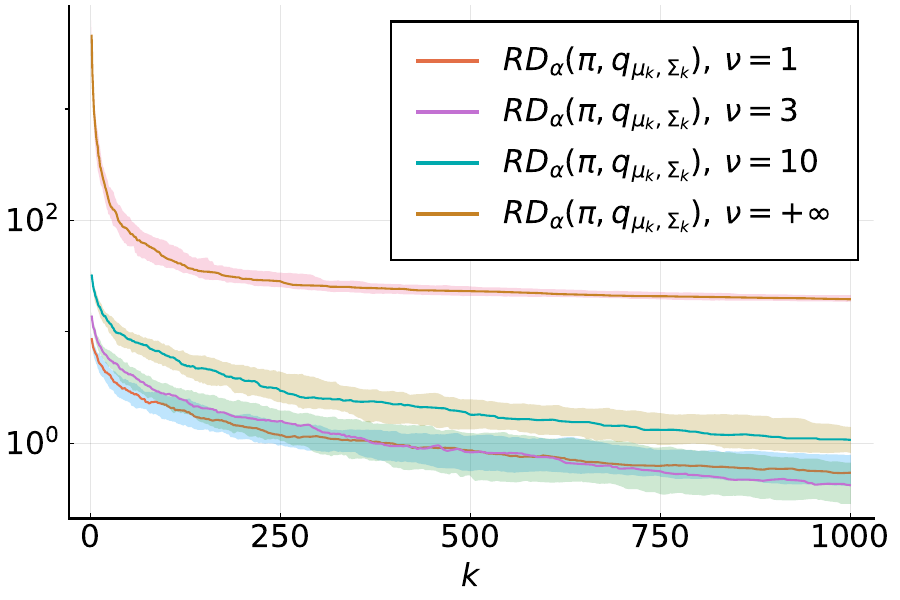}
        \caption{$\nu_{\pi}=3$, $\nu \in \{1,3,10,+\infty\}$}
    \end{subfigure}  
    \hfill
    \begin{subfigure}[b]{0.32\textwidth}
        \includegraphics[width = \textwidth]{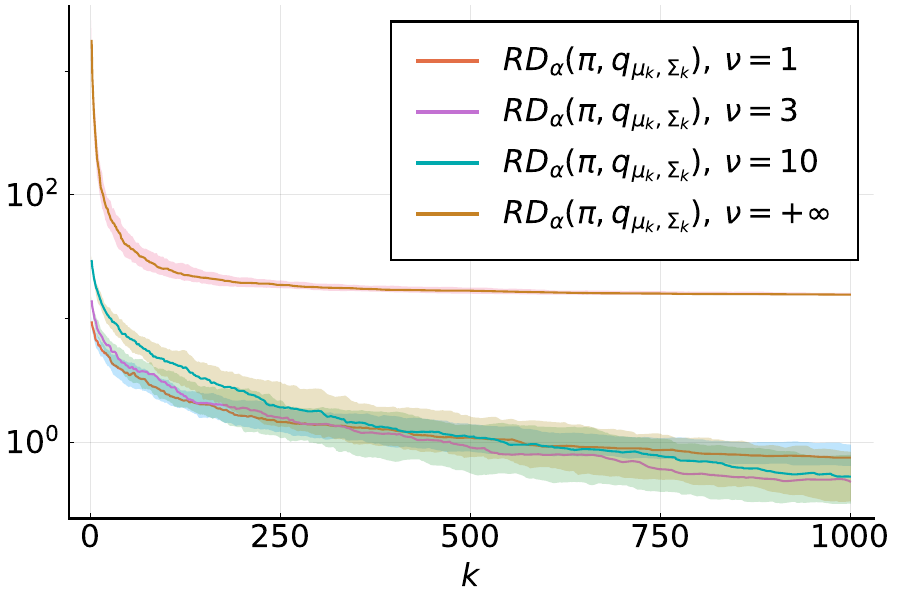}
        \caption{$\nu_{\pi}=10$, $\nu \in \{1,3,10,+\infty\}$}
    \end{subfigure} 
    \caption{Rényi divergence between $q_{\mu_k, \Sigma_k} \in \mathcal{T}_{\nu}^d$ and $\pi$ in dimension $d=5$ with $\kappa_{\pi} = 1000$ at every iteration $k$. The iterates $q_{\mu_k, \Sigma_k} \in \mathcal{T}_{\nu}^d$ are obtained using Algorithm \ref{alg:plainMALA}. The line is the median Rényi divergence per iteration and the shaded area is the interval between the first and third quartiles. The quartiles are obtained by running the algorithm for $100$ runs of $1000$ iterations.}
    \label{fig:VIwithMALA_d5cond1000}
\end{figure}

\paragraph{Scaled MALA} As shown in Figure \ref{fig:VIwithMALA_d5cond1000}, the use of an isotropic proposal in MALA might not be well suited for a poorly conditioned target. We now consider the implementation of Algorithm \ref{alg:proxVI} and the adaptation of the scale matrix $A$ in the MALA sampling step \eqref{eq:proposalPlainMALA}. To do so, we exploit the approximation $q_{\mu_k, \Sigma_k}$ of $\pi^{(\alpha)}$ by setting $A = \Sigma_k$ at each iteration $k \in \mathbb{N}$. The approximating distribution at iteration $k \in \mathbb{N}$, $q_{\mu_k, \Sigma_k}$ is itself updated following Algorithm \ref{alg:proxVI} with $\tau_k = \frac{1}{k}$ and $\pi^{(\alpha)}(T)$ being approximated by $N$ samples from the Markov chain. Therefore, the scaling matrix is updated every $N$ number of MALA steps and not at every iteration as in \cite{martin2012, marnissi2020}. The resulting procedure is detailed in Algorithm \ref{alg:adaptiveMALA}.

\begin{algorithm}[H]
     Choose an approximating family $\mathcal{T}_{\nu}^d$ and set $\alpha = 1 + \frac{2}{\nu+d}$. Choose $N \in \mathbb{N}$. Initialize $\mu_0$, $\Sigma_0$, and $x_0$.\\
     \For{$k = 0,\dots$}{
     Sample $\{x_{k+1}^{(1)},\dots,x_{k+1}^{(N)}\}$ samples from $x_k$ using the MALA algorithm with proposal as in Equation \eqref{eq:proposalPlainMALA} with $A = \Sigma_k$, set $x_{k+1} = x_{k+1}^{(N)}$.\\

     Evaluate $(\pi^{(\alpha)}(x))_{k+1}, (\pi^{(\alpha)}(xx^{\top}))_{k+1}$ by
     \begin{equation}
         \begin{cases}
             (\pi^{(\alpha)}(x))_{k+1} &= \frac{1}{N} \sum_{i=1}^{N} x_{k+1}^{(i)},\\
             (\pi^{(\alpha)}(xx^{\top}))_{k+1} &= \frac{1}{N} \sum_{i=1}^{N} x_{k+1}^{(i)} (x_{k+1}^{(i)})^{\top}.
         \end{cases}
     \end{equation}\\
     Update $\mu_{k+1}, \Sigma_{k+1}$ following
     \begin{equation}
         \begin{cases}
             \mu_{k+1} = \frac{1}{k+1} (\pi^{(\alpha)}(x))_{k+1} + \frac{k}{k+1} \mu_k,\\
             \Sigma_{k+1} = \frac{1}{k+1} (\pi^{(\alpha)}(xx^{\top}))_{k+1} + \frac{k}{k+1} (\Sigma_k + \mu_k \mu_k^{\top}) - \mu_{k+1} \mu_{k+1}^{\top}.
         \end{cases}
     \end{equation}
     }
     \caption{Solving Problem \eqref{pblm:VI-stud} with the updates \eqref{eq:proxLikeIterateVI} and scaled MALA.}
     \label{alg:adaptiveMALA}
\end{algorithm}

We now present our results, with $N = 10 d$. For each run, we initialize the algorithm with $\Sigma_0 = I_d$, and $\mu_0 = x_0$ sampled uniformly in $[-5,5]^d$. Figure \ref{fig:VIwithScaledMALA_d20cond10} shows the performance of Algorithm \ref{alg:adaptiveMALA} in dimension $d=20$ on a well-conditioned target. As in the previous cases, we observe that the best performance are reached when the approximating family contains the target when $\nu_{\pi} = 1$, while performance get more similar for other choices of $\nu_{\pi}$ as soon as $\nu$ is close to $\nu_{\pi}$.



\begin{figure}[H]
    \centering
    \begin{subfigure}[b]{0.32\textwidth}
        \includegraphics[width = \textwidth]{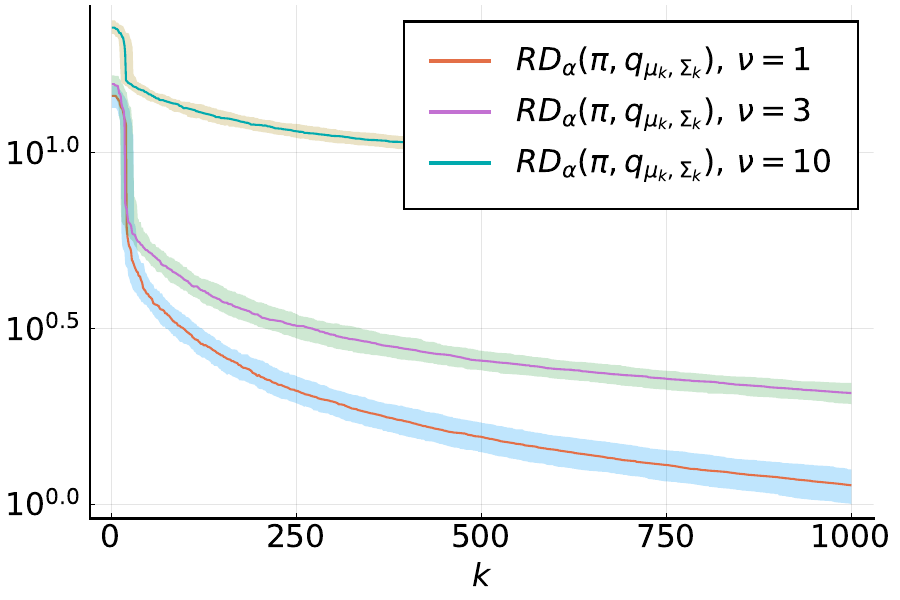}
        \caption{$\nu_{\pi}=1$, $\nu \in \{1,3,10\}$}
    \end{subfigure}  
    \hfill
    \begin{subfigure}[b]{0.32\textwidth}
        \includegraphics[width = \textwidth]{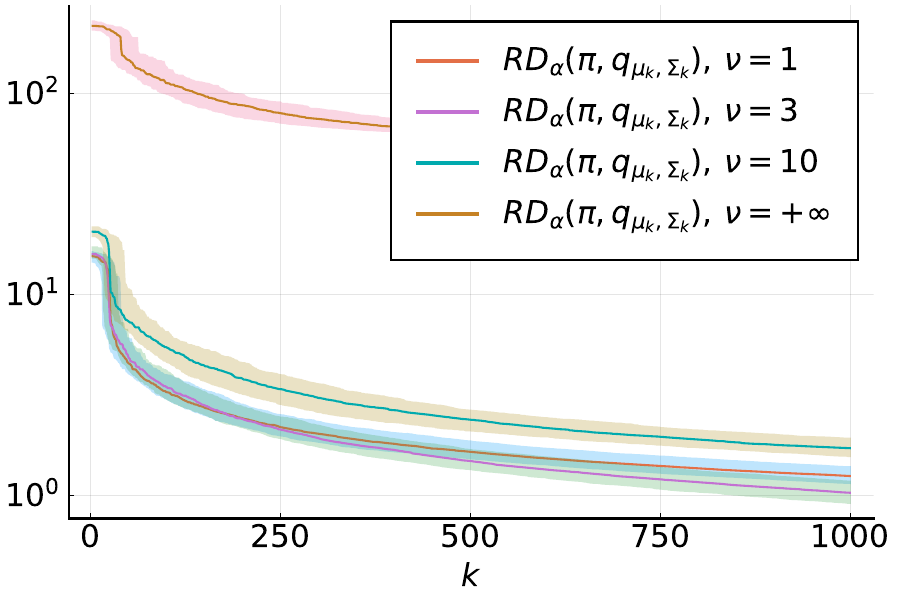}
        \caption{$\nu_{\pi}=3$, $\nu \in \{1,3,10,+\infty\}$}
    \end{subfigure}  
    \hfill
    \begin{subfigure}[b]{0.32\textwidth}
        \includegraphics[width = \textwidth]{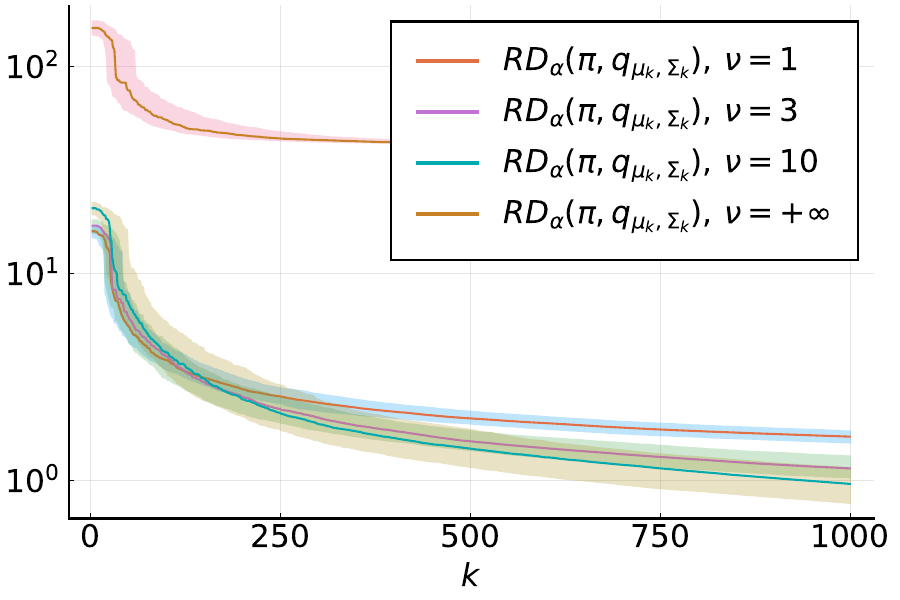}
        \caption{$\nu_{\pi}=10$, $\nu \in \{1,3,10,+\infty\}$}
    \end{subfigure} 
    \caption{Rényi divergence between $q_{\mu_k, \Sigma_k} \in \mathcal{T}_{\nu}^d$ and $\pi$ in dimension $d=20$ with $\kappa_{\pi} = 10$ at every iteration $k$. The iterates $q_{\mu_k, \Sigma_k} \in \mathcal{T}_{\nu}^d$ are obtained using Algorithm \ref{alg:adaptiveMALA}. The line is the median Rényi divergence per iteration and the shaded area is the interval between the first and third quartiles. The quartiles are obtained by running the algorithm for $100$ runs of $1000$ iterations.}
    \label{fig:VIwithScaledMALA_d20cond10}
\end{figure}

We now turn to a target $\pi$ that has a higher condition number $\kappa_{\pi} = 1000$, displaying the results on Figure \ref{fig:VIwithScaledMALA_d5cond1000}. We can observe that Algorithm \ref{alg:adaptiveMALA} reaches better performance than Algorithm \ref{alg:plainMALA} on this poorly conditioned target. Compared to the case of Figure \ref{fig:VIwithScaledMALA_d20cond10}, the values reached when the approximating family contains the target are now much better than the ones obtained with the other approximating families.

\begin{figure}[H]
    \centering
    \begin{subfigure}[b]{0.32\textwidth}
        \includegraphics[width = \textwidth]{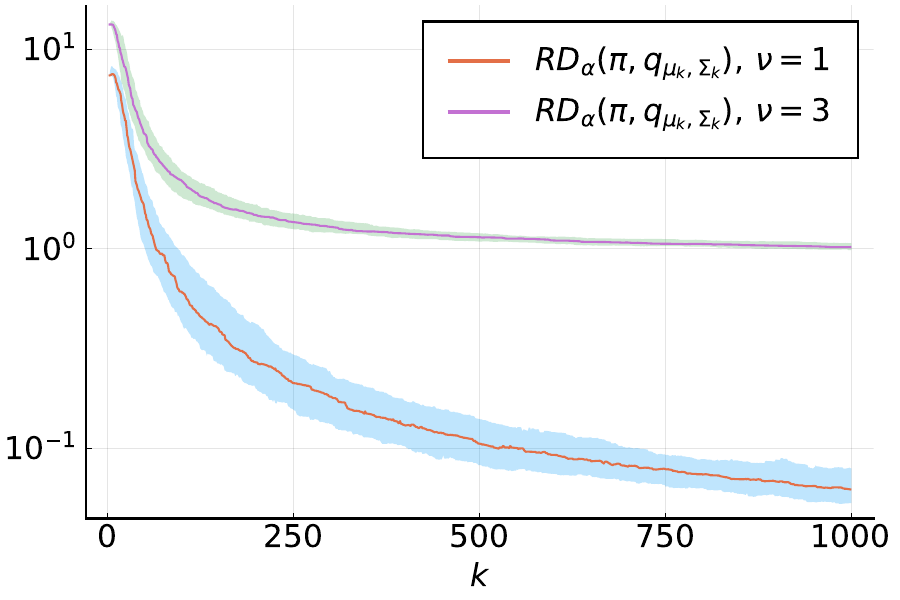}
        \caption{$\nu_{\pi}=1$, $\nu \in \{1,3\}$}
    \end{subfigure}  
    \hfill
    \begin{subfigure}[b]{0.32\textwidth}
        \includegraphics[width = \textwidth]{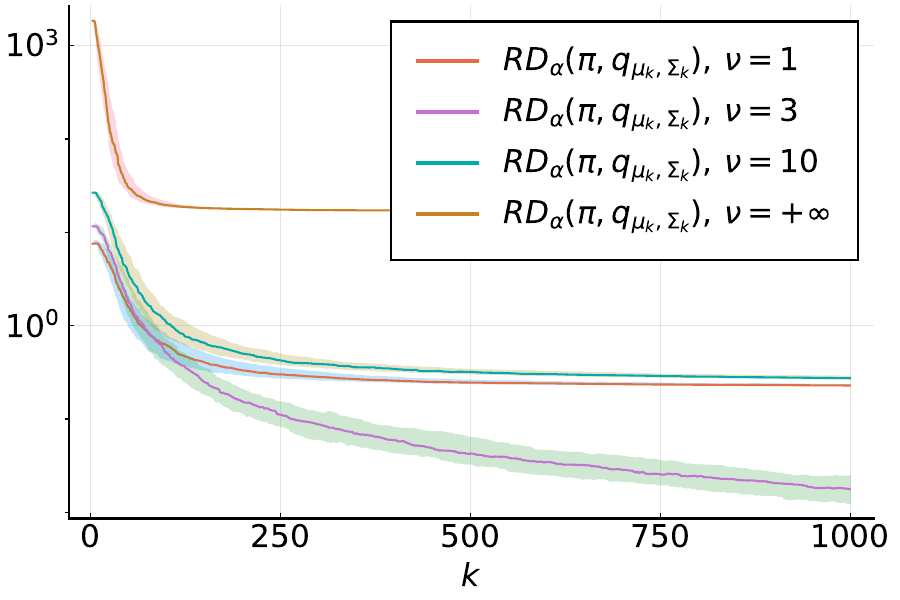}
        \caption{$\nu_{\pi}=3$, $\nu \in \{1,3,10,+\infty\}$}
    \end{subfigure}  
    \hfill
    \begin{subfigure}[b]{0.32\textwidth}
        \includegraphics[width = \textwidth]{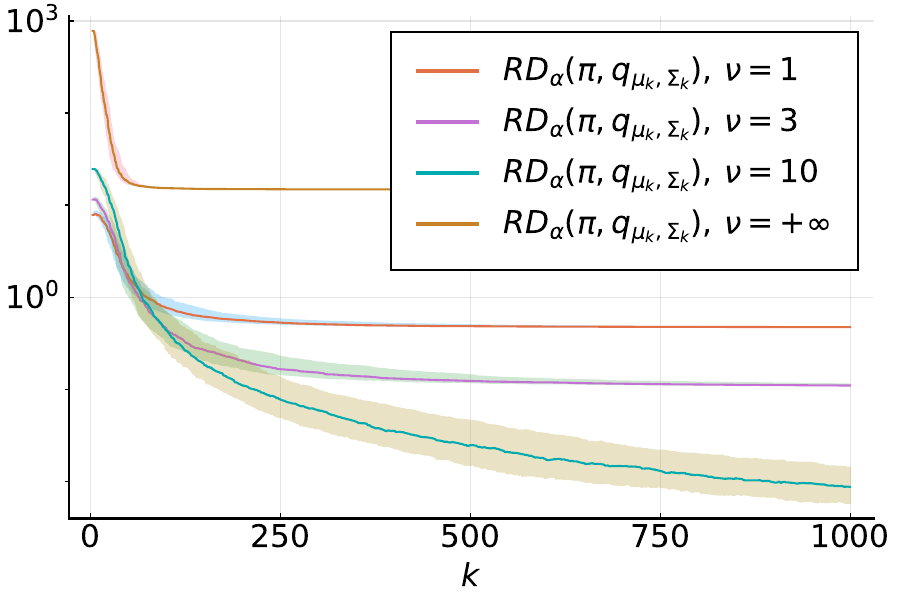}
        \caption{$\nu_{\pi}=10$, $\nu \in \{1,3,10,+\infty\}$}
    \end{subfigure} 
    \caption{Rényi divergence between $q_{\mu_k, \Sigma_k} \in \mathcal{T}_{\nu}^d$ and $\pi$ in dimension $d=5$ with $\kappa_{\pi} = 1000$ at every iteration $k$. The iterates $q_{\mu_k, \Sigma_k} \in \mathcal{T}_{\nu}^d$ are obtained using Algorithm \ref{alg:adaptiveMALA}. The line is the median Rényi divergence per iteration and the shaded area is the interval between the first and third quartiles. The quartiles are obtained by running the algorithm for $100$ runs of $1000$ iterations.}
    \label{fig:VIwithScaledMALA_d5cond1000}
\end{figure}

\paragraph{Synthesis of the results} 

Table \ref{table} summarizes the results, for the three algorithms, in terms of final value of the R\'enyi divergence, averaged over $100$ runs, after $10^3$ iterations. The results span the two scenarios, the first with a high-dimensional target with good conditioning, and the second with a poorly-conditioned target in lower dimension. The relative performance of Algorithms \ref{alg:plainMALA} and \ref{alg:adaptiveMALA} depends on the scenario. On targets that are well-conditioned but high-dimensional, Algorithm \ref{alg:plainMALA} seems to behave better than Algorithm \ref{alg:adaptiveMALA}. On the contrary, Algorithm \ref{alg:adaptiveMALA} yields here the best performance when the target is poorly conditioned, showing that our proposed scale adaptation mechanism is able to capture the geometry of the target distribution. We have also implemented the algorithms in the limit $\nu \rightarrow +\infty$, corresponding to the exponential family. In this case, the target is not properly captured by the proposals, showing the interest of our new result about the $\lambda$-exponential family. Finally, as expected, the idealized Algorithm \ref{alg:samplesFromTarget} reaches the best results in most cases, confirming the validity of our optimality conditions.

\begin{table}[H]
\centering
\begin{tabular}{ |c c|c|c|c|c|c|c| }

 \hline
 \multicolumn{2}{|c|}{} & \multicolumn{2}{|c|}{$\nu_{\pi}=1$} & \multicolumn{2}{|c|}{$\nu_{\pi}=3$} & \multicolumn{2}{|c|}{$\nu_{\pi}=10$}\\
 \cline{3-8}
  & & High $d$ & High $\kappa_{\pi}$ & High $d$ & High $\kappa_{\pi}$ & High $d$ & High $\kappa_{\pi}$ \\
 \hline
 \multirow{3}{4em}{$\nu=1$} & Alg.~\ref{alg:samplesFromTarget} & $2.25\cdot 10^{-2}$ & $3.84 \cdot 10^{-3}$ & $2.61 \cdot 10^{-1}$ & $2.15 \cdot 10^{-1}$ & $6.81 \cdot 10^{-1}$ & $4.65 \cdot 10^{-1}$  \\
 \cdashline{3-8}
 & Alg.~\ref{alg:plainMALA} & $\mathbf{3.01 \cdot 10^{-1}}$ & $4.00 \cdot 10^{-1}$ & $\mathbf{3.80 \cdot 10^{-1}}$ & $5.45 \cdot 10^{-1}$ & $\mathbf{7.30 \cdot 10^{-1}}$ & $7.48 \cdot 10^{-1}$ \\
 & Alg.~\ref{alg:adaptiveMALA} & $1.13 \cdot 10^{0}$ & $\mathbf{6.22 \cdot 10^{-2}}$ & $1.25 \cdot 10^{0}$ & $\mathbf{2.29 \cdot 10^{-1}}$ & $1.64 \cdot 10^{0}$ & $\mathbf{4.72 \cdot 10^{-1}}$ \\
 \hline
 \multirow{3}{4em}{$\nu=3$} & Alg.~\ref{alg:samplesFromTarget} & $7.02\cdot 10^{-1}$ & $7.78 \cdot 10^{-1}$ & $1.49 \cdot 10^{-3}$ & $4.50 \cdot 10^{-4}$ & $1.99 \cdot 10^{-1}$ & $1.03 \cdot 10^{-1}$  \\
 \cdashline{3-8}
 & Alg.~\ref{alg:plainMALA} & $\mathbf{1.08 \cdot 10^{0}}$ & $1.83 \cdot 10^{0}$ & $\mathbf{1.46 \cdot 10^{-1}}$ & $4.18 \cdot 10^{-1}$ & $\mathbf{2.49 \cdot 10^{-1}}$ & $4.76 \cdot 10^{-1}$ \\
 & Alg.~\ref{alg:adaptiveMALA} & $2.08 \cdot 10^{0}$ & $\mathbf{1.02 \cdot 10^{0}}$ & $1.03 \cdot 10^{0}$ & $\mathbf{1.78 \cdot 10^{-2}}$ & $1.15 \cdot 10^{0}$ & $\mathbf{1.10 \cdot 10^{-1}}$ \\
 \hline
 \multirow{3}{4em}{$\nu=10$} & Alg.~\ref{alg:samplesFromTarget} & $5.83 \cdot 10^{0}$ & $\times$ & $4.69 \cdot 10^{-1}$ & $2.41 \cdot 10^{-1}$ & $7.08 \cdot 10^{-4}$ & $2.24 \cdot 10^{-4}$ \\
 \cdashline{3-3}\cdashline{5-8}
 & Alg.~\ref{alg:plainMALA} & $\mathbf{7.20 \cdot 10^{0}}$ & $\times$ & $\mathbf{6.68 \cdot 10^{-1}}$ & $1.07 \cdot 10^{0}$ & $\mathbf{5.50 \cdot 10^{-2}}$ & $5.25 \cdot 10^{-1}$ \\
 & Alg.~\ref{alg:adaptiveMALA} & $9.36 \cdot 10^{0}$ & $\times$ & $1.71 \cdot 10^{0}$ & $\mathbf{2.74 \cdot 10^{-1}}$ & $9.67 \cdot 10^{-1}$ & $\mathbf{8.63\cdot 10^{-3}}$ \\
 \hline
 \multirow{3}{4em}{$\nu=+\infty$} & Alg.~\ref{alg:samplesFromTarget} & $\times$ & $\times$ & $4.91 \cdot 10^{1}$ & $1.67 \cdot 10^{1}$ & $4.06 \cdot 10^{1}$ & $1.48 \cdot 10^{1}$ \\
 \cdashline{5-8}
 & Alg.~\ref{alg:plainMALA} & $\times$ & $\times$ & $\mathbf{5.01 \cdot 10^{1}}$ & $1.96 \cdot 10^{1}$ & $\mathbf{4.07 \cdot 10^{1}}$ & $1.56 \cdot 10^{1}$ \\
 & Alg.~\ref{alg:adaptiveMALA} & $\times$ & $\times$ & $5.74 \cdot 10^{1}$ & $\mathbf{1.69 \cdot 10^{1}}$ & $4.18 \cdot 10^{1}$ & $\mathbf{1.48 \cdot 10^{1}}$ \\
 \hline
 
\end{tabular}
\caption{Median of the Rényi divergence $RD_{\alpha}(\pi, q_{\mu_K, \Sigma_K})$ over $100$ runs of $K = 10^3$ iteration. "High $d$" corresponds to $d=20, \kappa_{\pi} = 10$ and "High $\kappa_{\pi}$" to $d=5, \kappa_{\pi} = 10^3$. The symbol $\times$ denotes situations when Equation \eqref{eq:well-posednessCondition} is not satisfied. For each target and each approximating family $\mathcal{T}_{\nu}^d$, we highlighted in bold font the algorithm achieving the lowest value between Algorithm \ref{alg:plainMALA} and \ref{alg:adaptiveMALA}. The values obtained with the idealized Algorithm \ref{alg:samplesFromTarget} are indicated as a reference.}
\label{table}
\end{table}

We can observe on Table \ref{table} that Algorithms \ref{alg:plainMALA} and \ref{alg:adaptiveMALA} yield lower performance than Algorithm \ref{alg:samplesFromTarget}. However, implementing this last algorithm is unrealistic in practice, as it needs samples from the escort of the target. However, we can notice that, when $\nu_{\pi} \neq \nu$, i.e.,~the approximating family does not match with the target, the algorithms based on MALA are able to reach similar performance than Algorithm \ref{alg:samplesFromTarget}.

We see in Table \ref{table} that Algorithm \ref{alg:adaptiveMALA} outperforms Algorithm \ref{alg:plainMALA} on the high $\kappa_{\pi}$ scenario, sometimes by one or two orders of magnitude. This gain can be explained by the fact that Algorithm \ref{alg:adaptiveMALA} better handles the shape of the target. This indicates that as soon as the target may be poorly conditioned, it is best to turn to Algorithm \ref{alg:adaptiveMALA} instead of Algorithm \ref{alg:plainMALA}.

On the other hand, the situation is reversed on the high $d$ scenario, where the performance of Algorithm \ref{alg:adaptiveMALA} decreases. This indicates that on high-dimensional and well-conditioned targets, it may be beneficial to use Algorithm \ref{alg:plainMALA} instead of its scaled version, in Algorithm \ref{alg:adaptiveMALA}.

Finally, let us mention that when the algorithm matches the scenario, that is Algorithm \ref{alg:plainMALA} is used for high $d$ or Algorithm \ref{alg:adaptiveMALA} is used for high $\kappa_{\pi}$, it is especially important to choose $\nu = \nu_{\pi}$. Indeed, this is when we observe the biggest degradation if $\nu \neq \nu_{\pi}$.

\subsection{Maximum likelihood estimation with Student distributions}
\label{subsection:numericsMLE}

We consider now maximum likelihood problems of the form \eqref{pblm:MLE} and \eqref{pblm:MixtMLE} over the $\lambda$-exponential family $\mathcal{Q}_{\lambda}$. We will work in the case where $\mathcal{Q}_{\lambda}$ is the Student family $\mathcal{T}_{\nu}^d$.

\subsubsection{Online maximum likelihood with approximate proximal updates}

We now consider a maximum likelihood estimation problem of the form \eqref{pblm:MLE}. The approximating family is hereagain the Student family, $\mathcal{T}_{\nu}^d$. The samples processed for the maximum likelihood estimation are also distributed following a Student distribution $\pi \in \mathcal{T}_{\nu}^d$. Following \cite{kainth2022}, we consider an online setting, where one sample is delivered at each iteration of the algorithm. We implement Algorithm \ref{alg:proxMLE} in this setting and study how they approach the true maximum likelihood estimator, depending on the value of $\lambda$.

We assume that at every iteration $k \in \mathbb{N}$ one point $x_k \sim \pi$ is sampled. We implement Algorithm \ref{alg:proxMLE} and apply, at each iteration, the operator $P_{\tau_k}^{\{x_k\}}$, with a single data point, namely $x_k$, and we set $\tau_k = \frac{1}{k}$, ensuring an averaging effect. In our setting, this leads to Algorithm \ref{alg:onlineMLE_Student}.

\begin{algorithm}[H]
     Choose an approximating family $\mathcal{T}_{\nu}^d$ and initialize $\mu_0$ and $\Sigma_0$.\\
     \For{$k = 0,\dots$}{
     Using the new sample $x_k$, update $\mu_{k+1}, \Sigma_{k+1}$ following
     \begin{equation}
        \label{eq:onlineMLE_Student}
        \begin{cases}
            \mu_{k+1} = \frac{1}{k+1} x_k + \frac{k}{k+1} \mu_k,\\
            \Sigma_{k+1} = \frac{1}{k+1} x_k x_k^{\top} + \frac{k}{k+1} (\Sigma_k + \mu_k \mu_k^{\top}) - \mu_{k+1} \mu_{k+1}^{\top}.
        \end{cases}
     \end{equation}
     }
     \caption{Online algorithm to solve Problem \eqref{pblm:MLE} on Student families.}
     \label{alg:onlineMLE_Student}
\end{algorithm}

As discussed in Section \ref{subsection:discussion}, Algorithm \ref{alg:onlineMLE_Student} cannot exactly recover the parameters $(\mu_{\pi}, \Sigma_{\pi})$ of the distribution of the data points, even when $k \rightarrow +\infty$. From Propositions \ref{prop:proxMLE} and \ref{prop:escortStudent}, the sequence $\{(\mu_k, \Sigma_k)\}_{k \in \mathbb{N}}$ converges to $(\mu_*, \Sigma_*)$ satisfying $\mu_* = \mu_{\pi}$ and $\Sigma_* = \frac{\nu}{\nu-2}\Sigma_{\pi}$, provided that $\nu > 2$.

We illustrate the behavior of Algorithm \ref{alg:onlineMLE_Student} by showing several runs of it, in dimension $d= 1$, with $\nu \in \{ 3, 10\}$. This yields trajectories in the plane $(\mu, \sigma^2)$. In the Gaussian case, recovered when $\nu \rightarrow +\infty$, we have from Corollary \ref{corollary:studentMLE} that $q_* = \pi$. Trying different values of $\nu$ allows to explore situations that are far from the Gaussian setting when $\nu = 3$, or closer to it when $\nu=10$. In the latter case, we expect a lower mismatch between $\pi$ and $q_*$.

\begin{figure}[t]
    \centering
    \begin{subfigure}[b]{0.48\textwidth}
        \includegraphics[width = \textwidth]{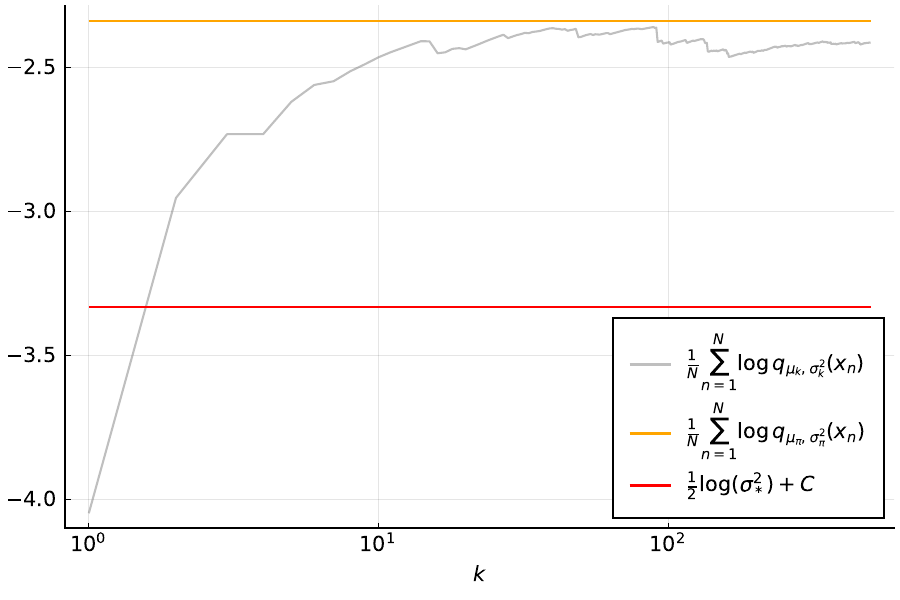}
        \caption{Plot of the log-likelihood of one trajectory of Algorithm \ref{alg:onlineMLE_Student}, with the log-likelihood of the true parameters in orange and the bound of Corollary \ref{corollary:studentMLE} in red.}
    \end{subfigure}  
    \hfill
    \begin{subfigure}[b]{0.48\textwidth}
        \includegraphics[width = \textwidth]{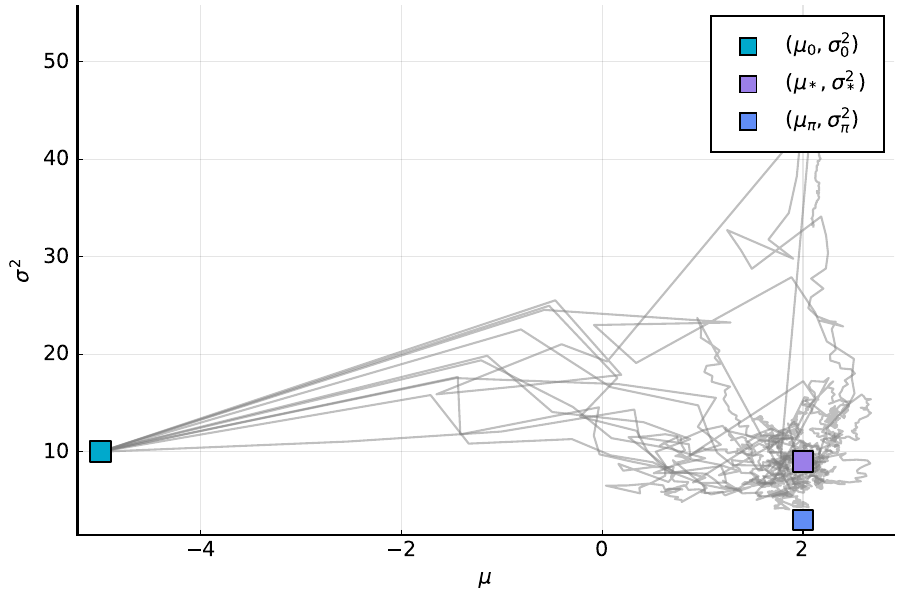}
        \caption{Plot of $10$ trajectories of Algorithm \ref{alg:onlineMLE_Student}, with the point $(\mu_*, \sigma_*^2)$ to which the trajectories converge and the point $(\mu_{\pi}, \sigma_{\pi}^2)$ encoding the distribution of the samples.}
    \end{subfigure} 
    \caption{Plots of trajectories of Algorithm \ref{alg:onlineMLE_Student}, initialized at $\mu_0 = -2$ and $\sigma^2_0 = 10$, in dimension $d=1$, with samples generated following $\pi \in \mathcal{T}_{\nu}^d$, $\nu = 3$, with parameters $(\mu_{\pi}, \sigma_{\pi}^2)$.}
    \label{fig:1dMLE-nu3}
\end{figure}

Figure \ref{fig:1dMLE-nu3} shows that the iterates $\{(\mu_k, \sigma_k^2)\}_{k \in \mathbb{N}}$ generated by Algorithm \ref{alg:onlineMLE_Student} converge to the point $(\mu_*, \sigma_*^2)$, which is different from the true parameters $(\mu_{\pi}, \sigma_{\pi}^2)$. We can also observe in this figure that the log-likelihood of the iterates gets very close to the one of $\pi$. The bound on the sub-optimal log-likelihood, predicted by Proposition \ref{prop:optimalityMLE} and Corollary \ref{corollary:studentMLE}, is satisfied by the iterates $\{(\mu_k, \sigma_k^2)\}_{k \in \mathbb{N}}$ after a small number of iterations.

\begin{figure}[H]
    \centering
    \begin{subfigure}[b]{0.48\textwidth}
        \includegraphics[width = \textwidth]{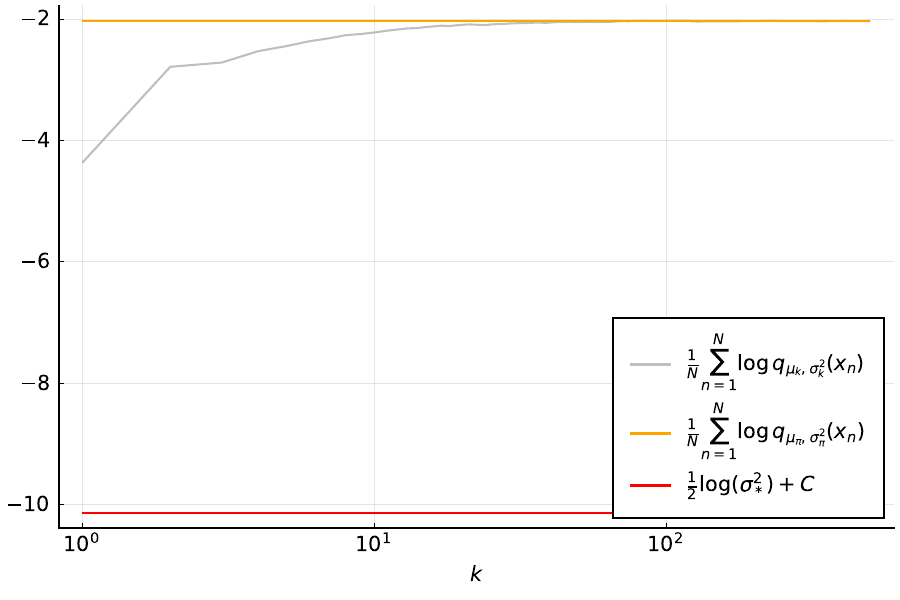}
        \caption{Plot of the log-likelihood of one trajectory of Algorithm \ref{alg:onlineMLE_Student}, with the log-likelihood of the true parameters in orange and the bound of Corollary \ref{corollary:studentMLE} in red.}
    \end{subfigure}  
    \hfill
    \begin{subfigure}[b]{0.48\textwidth}
        \includegraphics[width = \textwidth]{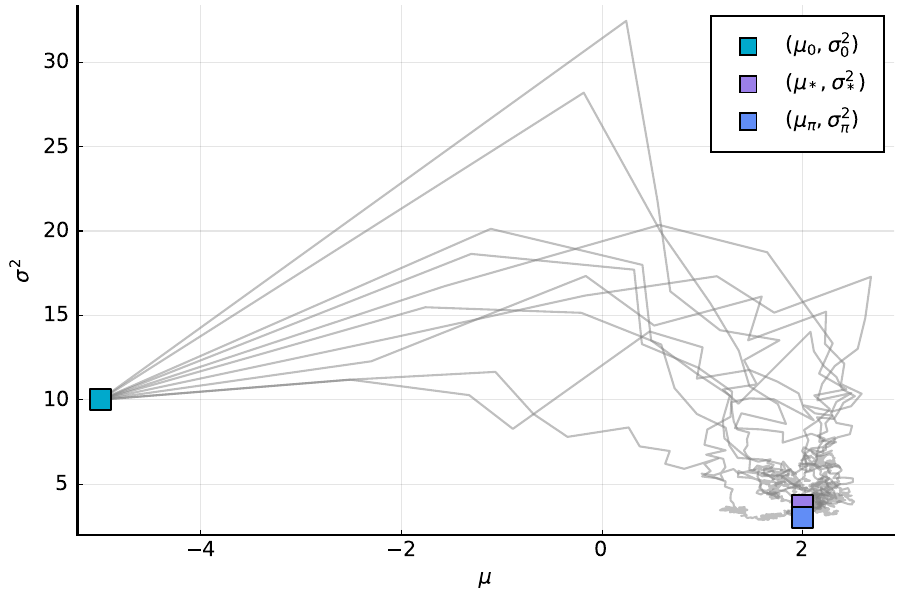}
        \caption{Plot of $10$ trajectories of Algorithm \ref{alg:onlineMLE_Student}, with the point $(\mu_*, \sigma_*^2)$ to which the trajectories converge and the point $(\mu_{\pi}, \sigma_{\pi}^2)$ encoding the distribution of the samples.}
    \end{subfigure} 
    \caption{Plots of trajectories of Algorithm \ref{alg:onlineMLE_Student} initialized at $\mu_0 = -2$ and $\sigma^2_0 = 10$, in dimension $d=1$, with samples generated following $\pi \in \mathcal{T}_{\nu}^d$, $\nu = 10$, with parameters $(\mu_{\pi}, \sigma_{\pi}^2)$.}
    \label{fig:1dMLE-nu10}
\end{figure}

Figure \ref{fig:1dMLE-nu10} considers a higher value of $\nu$. This setting is closer to the Gaussian case, reached in the limit $\nu \rightarrow +\infty$, for which our algorithm reaches the true distribution of the samples. We thus observe that in Figure \ref{fig:1dMLE-nu10}, the log-likelihood converge to the value of the log-likelihood of $\pi$. This is in contrast with Figure \ref{fig:1dMLE-nu3}, in which we can observe gap. We again observe that the iterates converge to the point $(\mu_*, \sigma_*^2)$, which is very close this time to the true parameters $(\mu_{\pi}, \sigma^2_{\pi})$. Compared to Figure \ref{fig:1dMLE-nu3} in the case $\nu=3$, we see that the bound predicted by Corollary \ref{corollary:studentMLE} is reached from the first iterates, meaning that it is not a tight bound for the log-likelihood of $q_*$.

According to our theoretical results, Algorithm \ref{alg:onlineMLE_Student} converges to a sub-optimal solution of Problem \eqref{pblm:MLE}. Such solution is very easy to implement and could be used to initialize a more complex but exact maximum likelihood estimation algorithm \cite{hasanasab2021, ayadi2023}. Moreover, the obtained sub-optimal solution has links with the probability distribution that generated the data, as discussed in Section \ref{subsection:discussion} and thus remains relevant for computing exact maximum likelihood estimators.

\subsubsection{Maximum likelihood estimation with mixtures using relaxed EM}

We consider here a maximum likelihood estimation problem over a mixture of Student distributions, that is, Problem \eqref{pblm:MixtMLE} where $Q_{\lambda} = \mathcal{T}_{\nu}^d$. The samples are also considered to be distributed from a mixture of Student distributions from $\mathcal{T}_{\nu}^d$, denoted by $\pi$ such that $\pi = \sum_{j=1}^J \xi_{*,j} q_{\mu_{*,j},\Sigma_{*,j}}$. We implement the relaxed EM method described in Algorithm \ref{alg:EM} in this particular case. The resulting scheme is summarized in Algorithm \ref{alg:EM_Student}. Algorithm \ref{alg:EM} only requires to work with a $\lambda$-exponential family satisfying Assumptions \ref{assumption:expFamily} and \ref{assumption:supportCondition}, so Algorithm \ref{alg:EM_Student} is a particular instance for a specific choice of family. We notice that Student distributions benefit from specific properties that would also allow the design of exact EM algorithms~\cite{hasanasab2021}, so we aim here at illustrating as a proof of concept the use of our mixture-based algorithm. 

\begin{algorithm}[H]
     Choose an approximating family $\mathcal{T}_{\nu}^d$. Initialize Let $\mu_{0,j} \in \mathbb{R}^d$, $\Sigma_{0,j} \in \mathcal{S}_{++}^d$, and $\xi_{0,j} \geq 0$ for $j=1,\dots,J$ such that $\sum_{j=1}^J \xi_{0,j} = 1$.\\
     \For{$k = 0,\dots$}{
     For every $j=1,\dots,J$, define the function $\gamma_{k,j}$ following Equation \eqref{eq:definitionGamma}, and update the parameters $\xi_{k+1,j}$ and $\mu_{k+1,j}, \Sigma_{k+1,j}$ such that they satisfy
        \begin{align}
            &\xi_{j,k+1} = \frac{1}{N} \sum_{i=1}^N \gamma_{k,j}(x_i),\\
            &\begin{cases}
                \mu_{k+1,j} = \sum_{i=1}^N \frac{\gamma_{k,j}(x_i)}{\sum_{i'=1}^N \gamma_{k,j}(x_{i'})} x_i,\\
                \Sigma_{k+1,j} = \sum_{i=1}^N \frac{\gamma_{k,j}(x_i)}{\sum_{i'=1}^N \gamma_{k,j}(x_{i'})} x_i x_i^{\top} - \mu_{k+1,j} \mu_{k+1,j}^{\top}.
            \end{cases}
        \end{align}  
     }
     \caption{A sub-optimal EM algorithm to solve Problem \eqref{pblm:MixtMLE} on Student families.}
     \label{alg:EM_Student}
\end{algorithm}

We illustrate the behavior of Algorithm \ref{alg:EM_Student} in dimension $d=2$, with $\nu \in \{3,10\}$. Note that a greater value of $\nu$ corresponds to a value of $\lambda$ closer to $0$, in which case the approximate M-steps are closer to being optimal (they are optimal for $\lambda=0$). We use $N=200$ samples, from $\pi = \sum_{j=1}^J \xi_{*,j} q_{\mu_{*,j},\Sigma_{*,j}}$ with $J = 4$. We use $\{ \xi_{*,j}\}_{j=1}^J = \{0.4, 0.1, 0.2, 0.3 \}$, with locations parameters $\mu_{*,1} = (10, 10)^{\top}$, $\mu_{*,2} = (-10,10)^{\top}$, $\mu_{*,3} = - \mu_{*,1}$, and $\mu_{*,4} = - \mu_{*,3}$. The shape matrices $\Sigma_{*,j}$, $j=1,\dots J$ are constructed in $\mathcal{S}_{++}^d$ with condition number $\kappa = 10$ following \cite{moré1989}. This is a controlled setting which allows to observe precisely the behavior of Algorithm \ref{alg:EM_Student} (i.e., an instance of Algorithm \ref{alg:EM}). Algorithm \ref{alg:EM_Student} is initialized with mixture weights satisfying $\xi_{j,0} = 1/J$ for $j=1,\dots,J$, initial locations parameters $\mu_{j,0}$ sampled from a normal distribution with zero mean and covariance $10 I$ and shape parameters being $\Sigma_{j,0} = 10 I$ for $j = 1,\dots,J$.

Figure \ref{fig:EM-nu3} shows the performance of Algorithm \ref{alg:EM_Student} when mixture components are from $\mathcal{T}_{\nu}^d$, with $\nu = 3$. The resulting mixture is able to identify the different components of the data-generating distribution $\pi$ and to achieve a significant increase in terms of log-likelihood from initialization. In this setting, the suboptimality in solving the M-step of the EM algorithm is more pronounced, as the iterates generated by the algorithm cannot reach the log-likelihood achieved by the data-generating distribution. This is to be expected, as the corresponding value of $\lambda$ is far from $\lambda=0$ where the M-step is optimal.

\begin{figure}[H]
    \centering
    \begin{subfigure}[b]{0.48\textwidth}
        \includegraphics[width = \textwidth]{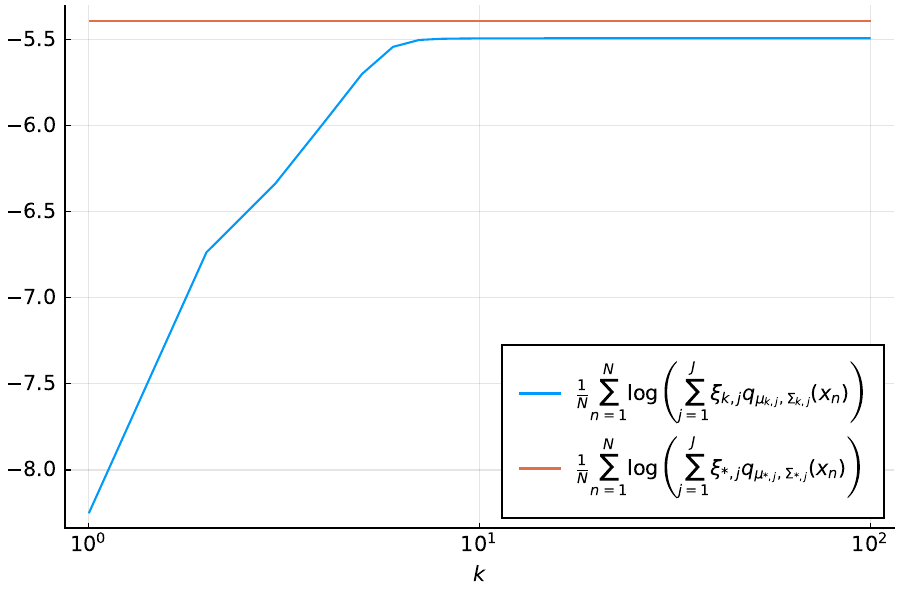}
        \caption{Plot of the log-likelihood achieved by the iterates of Algorithm \ref{alg:EM_Student}, with the log-likelihood of data-generating distribution in orange.}
    \end{subfigure}  
    \hfill
    \begin{subfigure}[b]{0.48\textwidth}
        \includegraphics[width = \textwidth]{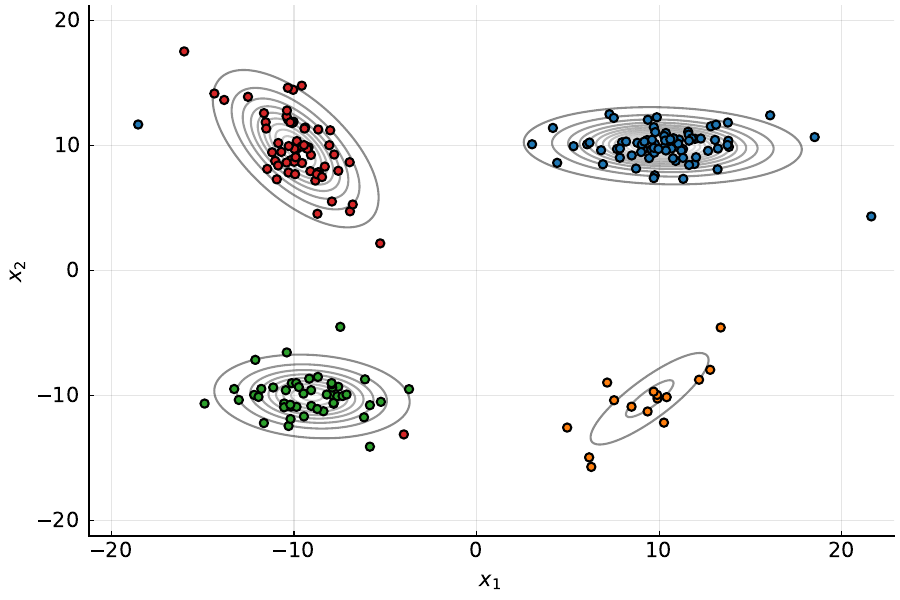}
        \caption{Plot of the samples generated by $\pi$, the different colors denoting the component of the mixture from which they have been drawn. The level lines of the final distribution generated by Algorithm \ref{alg:EM_Student} are shown in grey.}
    \end{subfigure} 
    \caption{Plot of the performance achieved by Algorithm \ref{alg:EM_Student} after $K = 100$ iterations in terms of log-likelihood and graphical representation in the sample space, for mixtures of distributions in $\mathcal{T}_{\nu}^d$, with $\nu =3$ and $d=2$.}
    \label{fig:EM-nu3}
\end{figure}

Figure \ref{fig:EM-nu10} shows the performance of Algorithm \ref{alg:EM_Student} when mixture components are from $\mathcal{T}_{\nu}^d$, with $\nu = 10$. We observe that the proposed algorithm generates iterates whose log-likelihood matches the one of the data-generating distribution. Indeed, this setting is closer to the case $\lambda = 0$ where our algorithm solves the M-step in the EM algorithm exactly, showing that the sub-optimality has no severe effect in this case. 

\begin{figure}[H]
    \centering
    \begin{subfigure}[b]{0.48\textwidth}
        \includegraphics[width = \textwidth]{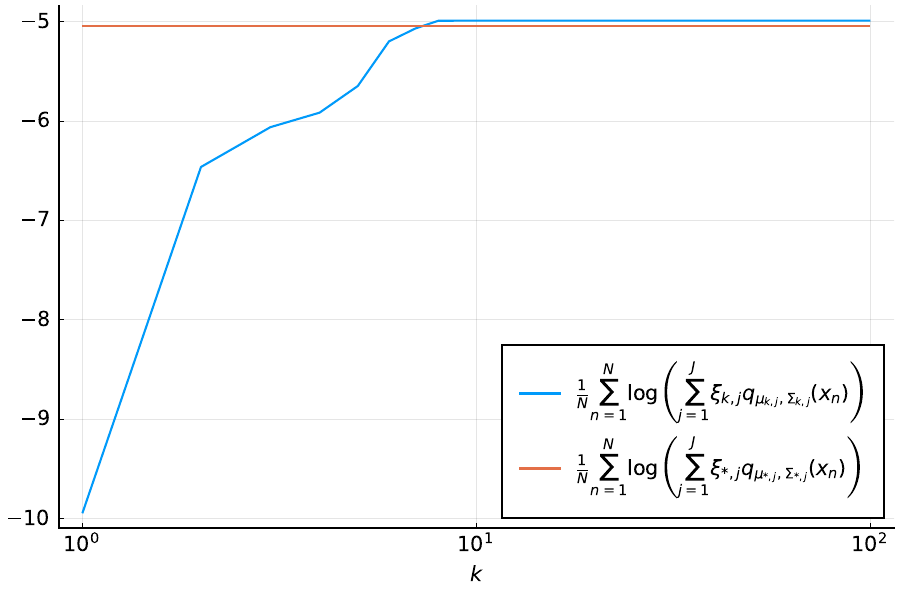}
        \caption{Plot of the log-likelihood achieved by the iterates of Algorithm \ref{alg:EM_Student}, with the log-likelihood of data-generating distribution in orange.}
    \end{subfigure}  
    \hfill
    \begin{subfigure}[b]{0.48\textwidth}
        \includegraphics[width = \textwidth]{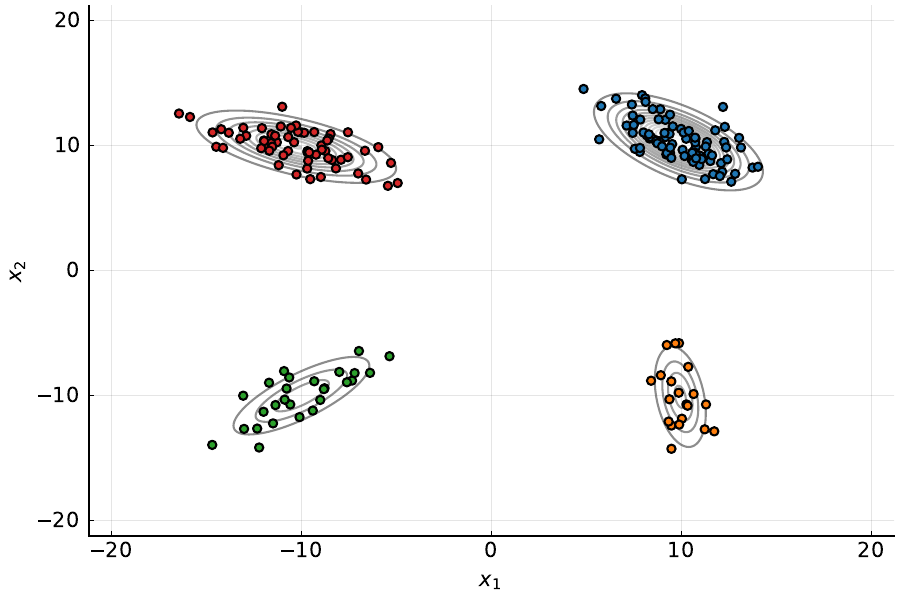}
        \caption{Plot of the samples generated by $\pi$, the different colors denoting the component of the mixture from which they have been drawn. The level lines of the final distribution generated by Algorithm \ref{alg:EM_Student} are shown in grey.}
    \end{subfigure} 
    \caption{Plot of the performance achieved by Algorithm \ref{alg:EM_Student} after $K = 100$ iterations in terms of log-likelihood and graphical representation in the sample space, for mixtures of distributions in $\mathcal{T}_{\nu}^d$, with $\nu =10$ and $d=2$.}
    \label{fig:EM-nu10}
\end{figure}

\section{Conclusion}
\label{section:conclusion}

In this work, we have studied variational inference and maximum likelihood estimation problems over the $\lambda$-exponential family, and we have proposed algorithms to solve theses problems. Several known results on the standard exponential family are retrieved as special cases.

First, we have shown that variational inference problems over the $\lambda$-exponential family can be solved by satisfying a generalized moment-matching condition that extends the existing one for the standard exponential family. We have also proposed an iterative algorithm to solve this problem, which identifies with a Bregman proximal algorithm in the particular case of the exponential family. The usefulness of our optimality conditions and our algorithm is confirmed by numerical experiments on heavy-tailed targets. These experiments show that the $\lambda$-exponential family can be used to capture phenomenon that the standard exponential family fails to represent.

Second, in maximum likelihood estimation problems, we exhibited sub-optimal solutions with a novel algorithm converging to it. In the case of the exponential family, the solutions become optimal and the algorithm reads again as a Bregman proximal algorithm. In the general case, our algorithm is quick and easy to implement, as demonstrated through numerical experiments. For problems with mixtures, we also proposed a relaxed EM algorithm that recovers the standard EM algorithm in the case of the exponential family. An interesting line of research would be the combination of our algorithms, which leads to sub-optimal solutions, with exact methods. 

We achieved these results by extending convex analysis notions to a more general setting, replacing the scalar product by a well-chosen non-linear coupling. By leveraging the specific structure of the problems we consider, we have been able to exhibit optimality conditions and proximal-like algorithm using such tools, which is one of the main novelties of our work. Extending our results and techniques to more general problems, including other divergences and distances over probabilities, or more general couplings, related for instance with elliptical distributions, seems to be an exciting area of research.

\appendix
\section{Proof of Proposition \ref{prop:studentFamily}}
\label{appendix:proof}

\begin{proof}[Proof of Proposition \ref{prop:studentFamily}]
    $(i)$ Consider a distribution in $\mathcal{T}^d_{\nu}$ with location parameter $\mu$ and scale matrix $\Sigma$. Then we compute for any $x \in \mathbb{R}^d$ the following.
    \begin{align*}
        q_{\mu, \Sigma}(x) &\propto \left( 1 + \frac{1}{\nu} (x-\mu)^{\top}\Sigma^{-1}(x-\mu) \right)^{- \frac{\nu + d}{2}}\\
        &\propto  \left( 1 + \frac{1}{\nu} \mu^{\top} \Sigma^{-1} \mu - \frac{2}{\nu} \mu^{\top} x + \frac{1}{\nu} x^{\top} \Sigma^{-1} x  \right)^{- \frac{\nu + d}{2}}\\
        &\propto \left(1 + \frac{1}{\nu} \mu^{\top} \Sigma^{-1} \mu\right)^{-\frac{\nu+d}{2}} \left( 1 + \left(- \frac{2}{\nu+d} \right) \left( -\frac{\nu+d}{2 \nu(1 + \frac{1}{\nu} \mu^{\top} \Sigma^{-1} \mu)} \left(- 2 \mu^{\top} x + x^{\top} \Sigma^{-1} x \right) \right)\right)^{- \frac{\nu + d}{2}}
    \end{align*}
    and since $\mu^{\top} x = \langle \mu, x \rangle$ and $x^{\top} \Sigma^{-1} x = \langle \Sigma^{-1}, xx^{\top} \rangle$, we can identify that $q_{\mu, \Sigma} = q_{\vartheta}$.

    We can identify from the above that $\mathcal{T}_{\nu}^d$ is an instance of the $\lambda$-exponential family with $\lambda = - \frac{2}{\nu+d}$. Its parameters are
    \begin{equation}
        \label{eq:naturalParam}
        \vartheta_1 = \frac{\nu + d}{\nu + \mu^{\top} \Sigma^{-1} \mu} \Sigma^{-1} \mu,\quad \vartheta_2 = - \frac{\nu + d}{2( \nu + \mu^{\top} \Sigma^{-1} \mu)} \Sigma^{-1}.
    \end{equation}

    In order to compute $\varphi_{\lambda}$, let us inverse the mapping $\mu, \Sigma \longmapsto \vartheta_1, \vartheta_2$. First, we can easily compute that $\mu = - \frac{1}{2}\vartheta_2^{-1} \vartheta_1$. Now, we compute the intermediate quantity $\mu^{\top}\Sigma^{-1} \mu$. Remark that
    \begin{equation}
        \label{eq:scalarProductIntermediate}
        \vartheta_1^{\top} \vartheta_2^{-1} \vartheta_1 = -\frac{2(\nu+d)}{\nu + \mu^{\top} \Sigma^{-1} \mu} \mu^{\top}\Sigma^{-1}\mu.
    \end{equation}
    Hence we deduce that 
    \begin{equation}
        \label{eq:scalarProductStudent}
        \nu + \mu^{\top}\Sigma^{-1} \mu = \frac{2 \nu (\nu + d)}{2(\nu+d) + \vartheta_1^{\top} \vartheta_2^{-1} \vartheta_1}.
    \end{equation}
    From Equations \eqref{eq:naturalParam} and \eqref{eq:scalarProductStudent}, it comes that $\Sigma^{-1} = -\frac{4 \nu}{2(\nu+d) + \vartheta_1^{\top} \vartheta_2^{-1} \vartheta_1}\vartheta_2$. Summarizing our results, we thus obtained
    \begin{equation}
        \label{eq:parameters}
        \mu = -\frac{1}{2} \vartheta_2^{-1} \vartheta_1,\quad \Sigma = - \frac{2(\nu+d) + \vartheta_1^{\top}\vartheta_2^{-1} \vartheta_1}{4 \nu} \vartheta_2^{-1}.
    \end{equation}

    Finally, we turn to the computation of $\varphi_{\lambda}$. We can identify
    \begin{align*}
        \varphi_{\lambda}(\vartheta) &= - \log \left( \frac{\det(\Sigma)^{-\frac{1}{2}}}{Z_{\nu}} \left(1 + \frac{1}{\nu} \mu \Sigma^{-1} \mu^{\top}\right)^{-\frac{\nu+d}{2}} \right)\\
        &= \frac{1}{2} \logdet(\Sigma) + \frac{\nu+d}{2} \log \left( 1 + \frac{1}{\nu} \mu \Sigma^{-1} \mu^{\top}\right) + \log(Z_{\nu})\\
        &=\frac{d}{2} \log \left( \frac{2(\nu+d) + \vartheta_1^{\top}\vartheta_2 \vartheta_1}{4 \nu} \right) + \frac{1}{2} \logdet(- \vartheta_2^{-1}) + \frac{\nu+d}{2} \log \left(\frac{2(\nu+d)}{2(\nu+d) + \vartheta_1^{\top} \vartheta_2 \vartheta_1} \right) + \log Z_{\nu}\\
        &= -\frac{d}{2}\log(4\nu) + \frac{1}{2} \logdet(- \vartheta_2^{-1}) + \frac{\nu+d}{2}\log(2(\nu+d)) - \frac{\nu}{2} \log(2(\nu+d) + \vartheta_1^{\top} \vartheta_2^{-1} \vartheta_1) + \log Z_{\nu}.
    \end{align*}
    This shows in particular that $\domain \varphi_{\lambda}(\vartheta) = \{\vartheta \in \mathbb{R}^d \times \mathcal{S}_{--}^d,\,2(\nu+d) + \vartheta_1^{\top} \vartheta_2^{-1} \vartheta_1 > 0\}$, which is non-empty. This shows that $\mathcal{T}_{\nu}^d$ satisfies Assumption \ref{assumption:expFamily}.

    $(ii)$ We now turn to the study of the escort probabilities. We can compute for $x \in \mathbb{R}^d$ the following:
    \begin{align*}
        q_{\mu, \Sigma}^{(\alpha)}(x) &= \frac{1}{Z^{(\alpha)}} \left( 1 + \frac{1}{\nu} (x-\mu)\Sigma^{-1}(x-\mu)^{\top} \right)^{- \alpha\frac{\nu + d}{2}}\\
        &= \frac{1}{Z^{(\alpha)}} \left( 1 + \frac{1}{\nu+2} (x-\mu) \left( \frac{\nu}{\nu+2}\Sigma \right)^{-1}(x-\mu)^{\top} \right)^{- \frac{(\nu+2) + d}{2}}.
    \end{align*}
    
    We recognize that $q_{\mu, \Sigma}^{(\alpha)}$ is a Student distribution with $\nu+2>2$ degrees of freedom, location parameter $\mu$ and scale matrix $\frac{\nu}{\nu+2}\Sigma$. Hence, we obtain that 
    \begin{equation}
        \begin{cases}
            q_{\mu, \Sigma}^{(\alpha)}(x) &= \mu,\\
            q_{\mu, \Sigma}^{(\alpha)}((x-\mu)(x-\mu)^{\top}) &= \frac{\nu+2}{(\nu+2)-2} \left( \frac{\nu}{\nu+2} \Sigma \right) = \Sigma.
        \end{cases}
    \end{equation}

    To show the bijection result, we show that the map $(\mu, \Sigma) \longmapsto (\vartheta_1, \vartheta_2)$ is a bijection between $\mathbb{R}^d \times \mathcal{S}_{++}^d$ and $\domain \varphi_{\lambda}$. Consider $\mu \in \mathbb{R}^d$, $\Sigma \in \mathcal{S}_{++}^d$ and $\vartheta_1, \vartheta_2$ defined as in Equation \eqref{eq:naturalParam}. We can first remark that $\vartheta_1 \in \mathbb{R}^d$ and that $\vartheta_2 \in \mathcal{S}_{--}^d$. Using the result of Equation \eqref{eq:scalarProductIntermediate}, we now compute
    \begin{align*}
        2(\nu+d) + \vartheta_1^{\top} \vartheta_2^{-1} \vartheta_1 &= 2(\nu+d) - \frac{2(\nu+d)}{\nu + \mu^{\top}\Sigma^{-1} \mu} \mu^{\top}\Sigma^{-1} \mu\\
        &= \frac{2\nu(\nu+d)}{\nu + \mu^{\top}\Sigma^{-1} \mu} >0,
    \end{align*}
    showing that $\vartheta_1, \vartheta_2 \in \domain \varphi_{\lambda}$. Consider now $\vartheta_1, \vartheta_2 \in \domain \varphi_{\lambda}$, and $\mu, \Sigma$ as given by Equation \eqref{eq:parameters}. By definition of $\domain \varphi_{\lambda}$, $\mu \in \mathbb{R}^d$ and $\Sigma \in \mathcal{S}_{++}^d$, showing the result.

    We now compute the Rényi entropy of $q_{\mu, \Sigma} \in \mathcal{T}_{\nu}^d$ for $\alpha = 1 - \lambda$ with $\lambda = - \frac{2}{\nu +d}$. By using similar steps as above, we obtain
    \begin{equation}
        H_{\alpha}(q_{\mu, \Sigma}) = \frac{1}{1-\alpha} \log \left( \frac{1}{Z_{\nu}^{\alpha} (\det \Sigma)^{\frac{\alpha}{2}}} Z_{\nu+2} \det \left( \frac{\nu}{\nu+2}\Sigma \right)^{\frac{1}{2}} \right).
    \end{equation}

    $(iii)$ Consider $\vartheta \in \domain \varphi_{\lambda}$, and $p \in \mathcal{P}(\mathcal{X}, dx)$. Consider $\mu, \Sigma \in \mathbb{R}^d \times \mathcal{S}_{++}^d$ given by Equation \eqref{eq:parameters}. We can then compute
    \begin{equation}
        1 + \lambda \langle \vartheta, p^{(\alpha)}(T) \rangle = 1 - \frac{2}{\nu + \mu^{\top}\Sigma^{-1}\mu} p^{(\alpha)}(x)^{\top} \Sigma^{-1} \mu + \frac{1}{\nu + \mu^{\top}\Sigma^{-1}\mu} \tr(\Sigma^{-1}p^{(\alpha)}(xx^{\top})),
    \end{equation}
    which is defined if $p^{(\alpha)}$ has finite first and second order moments. Introducing the quantity $V := p^{(\alpha)}((x-p^{(\alpha)}(x))(x-p^{(\alpha)}(x))^{\top}) = p^{(\alpha)}(xx^{\top}) - p^{(\alpha)}(x)p^{(\alpha)}(x)^{\top} \in \mathcal{S}_+^d$, we get for any $\vartheta \in \domain \varphi_{\lambda}$ that
    \begin{align*}
        1 + \lambda \langle \vartheta, p^{(\alpha)}(T) \rangle &= 1 - \frac{2}{\nu + \mu^{\top}\Sigma^{-1}\mu} p^{(\alpha)}(x)^{\top} \Sigma^{-1} \mu + \frac{1}{\nu + \mu^{\top}\Sigma^{-1}\mu} \tr(\Sigma^{-1}V)\\ &+ \frac{1}{\nu + \mu^{\top}\Sigma^{-1}\mu} p^{(\alpha)}(x)^{\top} \Sigma^{-1} p^{(\alpha)}(x)\\
        &= \frac{1}{\nu + \mu \Sigma^{-1}\mu} \left(\nu + (\mu-p^{(\alpha)}(x))^{\top} \Sigma^{-1} (\mu - p^{(\alpha)}(x))^{\top} + \tr(\Sigma^{-1} V) \right)\\
        &\geq \frac{\nu}{\nu + \mu \Sigma^{-1}\mu}.
    \end{align*}
    This shows that for any $\vartheta \in \domain \varphi_{\lambda}$, and $p^{(\alpha)} \in \mathcal{P}(\mathcal{X},dx)$ with finite first and second order moments, the quantity $c_{\lambda}(\vartheta, p^{(\alpha)}(T))$ is in $\mathbb{R}$. With the result of $(ii)$, this shows that $\mathcal{T}_{\nu}^d$, seen as an instance of the $\lambda$-exponential family, satisfies Assumptions \ref{assumption:supportCondition}.
\end{proof}

\bibliographystyle{abbrv}
\bibliography{biblio}

\end{document}